\documentclass[leqno,11pt]{amsart}
\usepackage{amsmath}
\usepackage{amsfonts}
\usepackage{amssymb}
\usepackage{amsthm}

\usepackage{mathrsfs}
\usepackage{dsfont}
\usepackage{txfonts}
\usepackage{mathabx}

\usepackage{graphicx}
\usepackage[all]{xy}
\usepackage{pgf,tikz}
\usepackage{subfigure}
\usepackage{tikz-3dplot}
\usetikzlibrary{patterns,arrows}
\tikzset{
    partial ellipse/.style args={#1:#2:#3}{
        insert path={+ (#1:#3) arc (#1:#2:#3)}
    }
}

\usepackage{hyperref}

\usepackage{color}
\usepackage[shortlabels]{enumitem}

\newtheorem{theorem}{Theorem}[section]
\newtheorem{lemma}[theorem]{Lemma}
\newtheorem{corollary}[theorem]{Corollary}

\newtheorem{proposition}[theorem]{Proposition}

\newtheorem{claim}[theorem]{Claim}

\newtheorem{step}{Step}[theorem]

\theoremstyle{definition}
\newtheorem{definition}[theorem]{Definition}

\theoremstyle{remark}
\newtheorem{remark}[theorem]{Remark}

\newcommand{\wcX}{{\widetilde {\mathcal X}_1}}

\newcommand{\cL}{\mathcal{L}}

\newcommand{\cX}{\mathcal{X}}

\newcommand{\Hess}{\mathrm{Hess}}

\newcommand{\bom}{{\bf\Omega}}

\newcommand{\bv}{{\bar v}}
\newcommand{\hb}{{\hat  b}}
\newcommand{\cE}{\mathcal{E}}

\newcommand{\lt}{\tau^a}

\newcommand{\bphi}{{\bar{\phi}}}

\newcommand{\vp}{\varphi}

\newcommand{\ve}{\varepsilon}
\newcommand{\hilb}{\mathcal H}

\newcommand{\cH}{\hilb}

\newcommand{\cC}{\mathcal{C}}

\newcommand{\mc}[1]{{\mathcal #1}}

\newcommand{\cO}{\mathcal{O}}

\newcommand{\cS}{\mathcal{S}}

\newcommand{\ep}{\varepsilon}

\newcommand{\be}{\begin{equation}}
\newcommand{\ee}{\end{equation}}
\newcommand{\bee}{\begin{equation*}}
\newcommand{\eee}{\end{equation*}}
\newcommand{\ba}{\begin{aligned}}
\newcommand{\ea}{\end{aligned}}
\newcommand{\sk}{\smallskip}

\newcommand{\R}{{\mathbb R}}

\numberwithin{equation}{section}

\begin{document}

\title{Formation and structural stability of nondegenerate neckpinches}

\author{Kyeongsu Choi, Panagiota Daskalopoulos, Natasa Sesum}
\maketitle

\begin{abstract}
In this paper, we study the formation, precise asymptotics, and structural stability of neckpinch singularities in mean curvature flow. Motivated by the static rigidity of cylindrical self-shrinkers established in \cite{CIM}, we first prove a dynamical rigidity result:  mean curvature flow of hypersurfaces that are initially graphically close to a generalized cylinder  on a sufficiently large scale, subject to a localized quadratic upward bending, inevitably develop a neckpinch singularity in finite time. We also establish sharp asymptotic expansions for the profile functions of these locally evolving graphs. We show that the rescaled graphical radius converges to a specific polynomial profile with quadratic bending, proving that the resulting singularities are  nondegenerate. Finally, we establish an openness theorem showing that  nondegenerate neckpinches are structurally stable under $C^2$ perturbations of the initial data. Combined with recent density theorems for the 3-dimensional mean curvature flow (see  \cite{Sz}), our results confirm that nondegenerate neckpinches constitute a generic and stable phenomenon in 3-dimensional mean curvature flow.
\end{abstract}
\tableofcontents
\section{Introduction}.

Mean curvature flow describes the evolution of a family of surfaces moving to decrease their area as efficiently as possible. More precisely, a smooth family of embedded hypersurfaces $M_t \subset \mathbb{R}^{n+1}$ moves by mean curvature flow if the position vector $x$ satisfies the equation:
\[
\frac{\partial x}{\partial t} = -H\nu
\]
where $H$ is the mean curvature, and $\nu$ is the outward unit normal vector. Because this is a nonlinear equation, closed hypersurfaces evolving under this flow inevitably develop singularities in finite time. Understanding the nature of these singularities is one of the most fundamental problems in the study of geometric flows. 

By Huisken's foundational monotonicity formula \cite{Hu}, tangent flows at these singularities are modeled by self-shrinkers. Among these, generalized shrinking cylinders play a critical role.
The celebrated results of Colding and Minicozzi established that spheres and cylinders are the only linearly stable singularity models for mean curvature flow. The expectation that any initial hypersurface can be perturbed so that the mean curvature flow starting at a perturbed hypersurface develops either spherical or cylindrical singularity has recently been confirmed in the case of surfaces evolving in $\mathbb{R}^3$, due to recent advancements regarding genericity and multiplicity. Specifically, the multiplicity one conjecture has been resolved through the works of Bamler and Kleiner in \cite{BK} for surfaces evolving in $\mathbb{R}^3$. Furthermore, in three dimensions, the  results of Chodosh, Choi, Mantoulidis, and Schulze in \cite{CCMS} definitively establish that mean curvature flow from generic initial data develops only multiplicity-one spherical and cylindrical singularities.

In a closely related setting, Knopf, Sigal, and Zhou in \cite{GKS}  analyzed the dynamics of flows starting as a global graph over a round cylinder in $\mathbb{R}^3$ with quadratic bending. They proved that for an open set of initial data that is globally $C^3$-close to a round cylinder, the flow still  develops a neckpinch singularity in finite time. They derived rigorous, detailed matched asymptotics of the evolving surface, and they demonstrated that the geometry asymptotically regains rotational symmetry near the singular neck as it collapses.

This dynamical expectation is deeply rooted in the static rigidity of the cylinders themselves. It is a well-known rigidity result of Colding, Ilmanen and Minicozzi (\cite{CIM}) for self-shrinkers that any shrinker which is sufficiently close to a cylinder on a sufficiently large set must  globally coincide with that exact cylinder, up to rigid motions. We view our first main result of this paper as the dynamical, flow version of this static rigidity. While the static rigidity dictates that spatial closeness on a large scale forces exact geometric identity for shrinkers, our theorem demonstrates that dynamic closeness. More precisely, we show that starting with initial data that is graphically locally close to a cylinder on a sufficiently large set, specifically with a localized quadratic upward bending, the mean curvature flow must develop a neckpinch singularity. 

We now describe the set up of the current paper. We fix an integer $k$ such that $1 \le k < n$.  Throughout the paper we will denote by $\cC_{n,k}$ the cylinder
\be\label{eqn-cyl}
C_{n,k} := \mathbb{R}^k \times \mathbb{S}^{n-k}(q) \subset \R^{n+1} \qquad \mbox{of  radius} \qquad  q:= \sqrt{2(n-k)}.
\ee 
We use {\em  coordinates $p=(x, \theta) $ in the cylinder}  $C_{n,k}$. Here, $\{x_i\}_{i\in \{1,\dots, k\}}$ denote the coordinates on $\mathbb{R}^k$, and $\{\theta_\alpha\}_{\alpha\in \{1,\dots, n-k+1\}}$ denote the restriction of the coordinate functions of $\mathbb{R}^{n-k+1}$ to the sphere $\mathbb{S}^{n-k}(q)$ of radius $q$. Also,  we will often use the notation   $p=(x,x')$ 
for points in $\R^k \times \R^{n+1-k}$.  

Our class of initial data consists of those closed hypersurfaces $M_0$ that can be written locally as graphs $U(x, \theta, 0)$ over $C_{n,k}$, equipped with a local quadratic upward bending assumption. We will prove that the mean curvature flow starting from any of these hypersurfaces develops an isolated neckpinch singularity, and we will rigorously describe its precise asymptotics.

To achieve this, we locally parametrize the evolving hypersurfaces $M_t$ as graphs, where $U(x, \theta, t)$ represents the radial distance from the axis space $\mathbb{R}^k$ of the cylinder. We write the evolution equation for the graphical radius of the rescaled mean curvature flow and argue that as the rescaled time approaches $+\infty$, it converges to the cylindrical radius $q = \sqrt{2(n - k)}$ with a precise asymptotic profile. The primary technical challenge is that our original reference cylinder $C_{n,k}$ may not coincide with the singularity model itself, but only aligns with it after a suitable translation, dilation  and rotation.

Take any initial data from our class, and consider the family of rescaled flows \be\label{eqn-tildeM}
{\widebar M}^{\cX,S}_\tau := e^{\tau/2} S(M_t - p_0), \qquad \tau = -\log(t_0 - t)
\ee 
where  the parameters $\cX, S$ with $\cX:=(x_0,x_0', t_0)$, center $p_0 := (x_0, x_0') \in \R^k \times \R^{n+1-k}$, and rotation $S$.

For a given  point $\cX=(x_0, x_0', t_0)$ with  $p_0 := (x_0, x_0') \in \R^k \times \R^{n+1-k}$, 
 we will denote by $U^{\cX,S}(x,\theta, t)$   {\em the profile   of the shifted and rotated flow}  $S (M_t-p_0)$. Note that $U^{\cX,S}(x,\theta, t)$ is independent 
 from $t_0$ and  that if $p_0$ is the origin and $S$ is the identity map, then $U^{\cX,S}(x, \theta, t)=U(x,\theta,t)$.
Our goal is to dynamically locate the center $p_0$, the time $t_0$, and the rotation $S$ so that the rescaled graphical radius of $\widebar{M}^{\cX,S}_\tau$, given by:
\begin{equation}
\label{eq-rescaling}
u^{\cX,S}(y, \theta, \tau) = \frac{U^{\cX,S}(x, \theta, t)}{\sqrt{t_0 - t}}, \qquad \mbox{with} \qquad
y = \frac{x}{\sqrt{t_0 - t}}.
\end{equation}
If $(p_0,t_0)$ is a cylindrical singularity, there is some rotation $S$ such that $u^{\cX,S}(y, \theta, \tau)$ converges to $\sqrt{2(n - k)}$, and so that the graphical radius has the desired asymptotics given by quadratic bending, as $\tau\to \infty$. 
Our first result states as follows.
\begin{theorem}
\label{thm-main0}
Suppose  $\text{Ent}(M_{0})\leq \Lambda$.
There exists an $\bar{\varepsilon}_0$,  so that for any $0 < \varepsilon \le \bar{\varepsilon}_0$ there exists an $L_{\varepsilon}$ with the following significance.  If $M_0\subset\mathbb{R}^{n+1}$ is a  hypersurface which can be locally written as a graph $U(x,\theta,0)$ over $\cC_{n,k}$, that is,
\begin{equation}
\label{eq-Ueps}
U(x,\theta,0) = \sqrt{2(n-k)} + \varepsilon \, |x|^2 + o(\varepsilon)\, (1+|x|^2)
\end{equation}
for all $|x| \le L_{\varepsilon}$, in the $C^0$ sense, and also its derivatives up to order four satisfy 
$$\sum_{0\leq |\alpha|+|\beta | \leq 4} \Big | \frac{\partial^{|a|+|\beta|} }{\partial x^\alpha \partial \theta^\beta} U(x,\theta,0) \Big |  \le C_n\, \ve\, (1+|x|^2) \qquad \mbox{for}\,\,\, |x| \leq L_\ep $$
for a uniform constant $C_n$,   then the mean curvature flow starting from $M_0$ develops a first  singularity within the set $\big \{|x| < \frac{L_{\varepsilon}}{5} \big \}$ at some finite time $T < \infty$, and all singularities occurring at this first singular time in this region are cylindrical. Moreover, at each such singular point, the unique tangent flow is a generalized cylinder $\mathbb{R}^k \times \mathbb{S}^{n-k}$.
\end{theorem}
 If in addition to \eqref{eq-Ueps}, we assume that the angular derivative of $U_\ep$ satisfies $|\nabla_{\theta} U_\ep(x, \theta,0)| = o(\varepsilon)$, then we can prove that the neckpinch singularity in Theorem \ref{thm-main0} is a  nondegenerate neckpinch singularity with precise asymptotics. More precisely we  establish the sharp asymptotic expansion of the profile function  for a class of geometric initial data that are locally graphical over a cylinder at time $t=0$.

\begin{theorem}
\label{thm-main}
Suppose $\text{Ent}(M_0)\leq \Lambda$.
 If $M_t$ is a complete mean curvature flow so that $M_0$ satisfies the conditions of  Theorem \ref{thm-main0},  with profile $U(x,\theta,t)$ satisfying
\[
|\nabla_{\theta} U(\cdot,\cdot,0)| = o(\varepsilon)
\]
then the mean curvature flow $\{M_t\}$ develops an isolated nondegenerate neckpinch singularity.

More precisely, for every $\epsilon > 0$ sufficiently small, there exist a rotation $S$, a singular time $T$,  and a  center $p_0 = (\bar x_0,\bar x_0')$ with the following significance.
For any $K > 0$ there is   $\bar \tau$ sufficiently large 
with the property  that for all $\tau \ge \bar \tau$, the rescaled graphical function of $\widebar{M}^{\widebar{\cX}_0}_{\tau} = e^{\tau/2}\, S (M_t - p_0)$
 with respect to the center $\widebar{\cX}_0 = (p_0,T)$,  call it $u^{\widebar{\cX}_0}(y,\theta,\tau)$, defined by \eqref{eq-rescaling}, 
 and whose initial unrescaled graphical radius  is given by \eqref{eq-Ueps}, has the following asymptotics
\[u^{\widebar{\cX}_0}(y,\theta,\tau) =  \frac{\sqrt{2(n-k)}}{4\tau}\, (|y|^2-2k) + o(\tau^{-1}),\]
on $\widebar{M}_{\tau}\cap B_K.$
\end{theorem} 

\begin{remark}
Note that in Theorem \ref{thm-main},  the rotation $S$, the center $p_0 = (\bar x_0,\bar x_0')$, the singular time $T$, and the graphical radius $u^{\widebar{\cX}_0,S}$ all depend on $\varepsilon$, but in our discussion below we will simply write $u^{\widebar{\cX}_0,S}$, having these dependencies in mind. 
\end{remark}

Let $\{M_t\}_{t\in [0,T)}$ be a mean curvature flow with bounded entropy and with cylindrical singularity modeled on $\mc{C}_{n,k}$ at the space time point $(p_0,T)$. In the case when $n = 3$, or positive mean curvature assumption in dimensions $n \ge 4$, due to the uniqueness of cylindrical tangent flows (see \cite{BK}, \cite{CM2}, \cite{HW}) we have that its associated rescaled mean curvature flow $\{\widebar M_{\tau}^{\widebar{\cX}_0,S}\}_{\tau\in [-\log T, \infty)}$ converges, as $\tau\to\infty$, to $\cC_{n,k}$ in $C^{\infty}_{\operatorname{loc}}$-sense. 

Recall that in \cite{SX}, \cite{SXW} it was shown that in the case of a  neckpinch singularity whose singularity model is $\cC_{n,k}$, for every $R > 0$ there exists a time $\tau_R$ so that for all $\tau\ge \tau_R$, the rescaled mean curvature flow is  a graph over $\cC_{n,k}\cap B_R$ and the graphical function (cylindrical radius) has the following asymptotics
\begin{equation}
\label{eq-sun}
u(y,\theta,\tau) = \sqrt{2(n-k)} + \sum_{i\in \mathcal{I}} \frac{1}{4\tau}\, (y_i^2-2) + o(\tau^{-1}),
\end{equation}
holding in the $C^2$-sense,  for some subset  $\mathcal{I} \subset \{1, 2, \dots, k\}$.

\begin{definition}
\label{def-neckpinches}
Let $\{M_t\}_{t\in [0,T)}$ be a mean curvature flow with cylindrical singularity modeled on the cylinder $\cC_{n,k}$ of radius $q:= \sqrt{2(n-k)} $ at the space time point $\widebar{\cX}_0=(p_0,T)$, in the sense that its associated rescaled mean curvature flow $\{ 
\widebar {M}_{\tau}^{\widebar{\cX}_0,S}  \}_{\tau\in [-\log T, \infty)}$ converges, as $\tau\to\infty$, to $\cC_{n,k}$ in $C^{\infty}_{\operatorname{loc}}$-sense. 
A cylindrical singularity with the asymptotics described  in \eqref{eq-sun} is  called {\it nondegenerate} if $\mathcal{I} = \{1,\dots,k\}$, {\it partially nondegenerate} if $\mathcal{I} \neq \emptyset$, but $\mathcal{I} \varsubsetneq \{1,\dots,k\}$, and is called {\it degenerate}, otherwise. 
\end{definition}

It was shown in  \cite{SX} that a  {\em nondegenerate neckpinch singularity},  in the sense of previous Definition, is {\em isolated. }
Furthermore in \cite{SS}, Schulze and Sesum proved   that if a mean curvature flow develops finitely many neckpinch singularities at the first singular time, so that each neck locally disconnects the manifold into two pieces none of which disappears at the first singular time $T$,  then the flow starting from any sufficiently small perturbation of the initial data will also encounter only neckpinch-type singularities.
 As a corollary of Theorem \ref{thm-main} we prove the stronger stability result which shows that being a nondegenerate singularity is an open condition, or equivalently, that the resulting nondegenerate neckpinch singularities are strictly stable under $C^2$-perturbations.
 
\begin{theorem}
\label{thm-open-neck}
Assume that  $\{M_t\}_{t\in [0,T)}$ is a smooth mean curvature flow that develops  at the point $(p_0,T)$ a nondegenerate neckpinch singularity  (in the sense of   Definition 
\ref{def-neckpinches})   modeled on the cylinder  $\cC_{n,k}$. 
Then,  there exists an $\bar{\varepsilon} > 0$ such that for all $U_0: M_0 \to \mathbb{R}$ with $\|U_0\|_{C^2} \le \bar{\varepsilon}$, the mean curvature flow with initial condition $\hat{M}_0 := \{x + U_0(x) \nu(x) \,\,\,|\,\,\, x\in M_0\}$ develops a unique nondegenerate singularity in a $C\bar{\varepsilon}$ neighborhood of $(p_0,T)$, modeled on $\cC_{n,k}$, up to rotation.
\end{theorem}

This openness theorem carries certain global implications. Specifically, it perfectly complements a recent result by Sz\'ekelyhidi (see \cite{Sz}), who proved that in dimension $n=3$, the set of initial data leading to nondegenerate neckpinches is dense. Together, our demonstration of openness and Sz\'ekelyhidi's proof of density establish that nondegenerate neckpinches are a generic phenomenon in three-dimensional mean curvature flow.

The strict stability of nondegenerate neckpinches established in Theorem \ref{thm-open-neck} stands in contrast to the behavior of degenerate neckpinches. For instance, recent work by Angenent-Daskalopoulos-Sesum (\cite{ADS2}) on ``peanut solutions"—which are canonical examples of degenerate neckpinches—demonstrates that such singularities are highly unstable. In that paper, the authors show that within any sufficiently small neighborhood of peanut initial data, there exist both a perturbation whose mean curvature flow develops a spherical singularity, and a perturbation whose flow develops a neckpinch singularity. Thus, our openness theorem highlights that it is precisely the non-degeneracy of the neckpinch that guarantees its stability under small perturbations.

\subsection{Outline of the proofs}
The organization of the paper is as follows: 

\smallskip
Our initial condition  \eqref{eq-Ueps} in Theorem \ref{thm-main0} indicates that the first singularity $(p_0,T)$  of the solution $M_t$ in the set $\big \{ |x| \leq L_\ep \big \}$ 
happens at  time $T=1+t(x_0)$, where $t(x_0)$ is a small error depending on $x_0$ and $\ep$. 
This is because  the singularity time of the cylinder $C_{n,k}$ of radius $q=\sqrt{2(n-k)}$ is $T_{\cC_{n,k}}=1$.
In fact this will follow from the proof of the Theorem.  

Having this in mind, in Section \ref{sec-first-est}, we rescale our mean curvature flow around a point $\cX_0 = (x_0,x_0',1)$, where   $(x_0, x_0') \in \mathbb{R}^k \times \R^{n-k+1}$. 
 We use a consequence of the pseudolocality result for the mean curvature flow to show that the local closeness to the cylinder $\mathcal{C}_{n,k}$, which we assume initially on a large set depending on $\ve$, persists for a long time interval on an even larger set. However, this closeness is given in terms of a small but fixed number $\varepsilon_0$ that could potentially be much greater than $\varepsilon$.

In Section \ref{sec-X0}, we employ $L^2$ theory and the results from Section \ref{sec-first-est} to improve the closeness of the rescaled flow around  $\cX_1=(x_0,0,1)$  to the cylinder $\mathcal{C}_{n,k}$ on a large ball for a large rescaled  time interval, and ultimately we show the closeness of the rescaled flow around $\cX_1$ to $\mathcal{C}_{n,k}$ in terms of a certain power of $\varepsilon$. 

The estimates obtained in Section \ref{sec-X0} hold up to rescaled time $\tau \le -\frac{99}{10}\log\varepsilon$. We cannot expect these estimates to hold all the way up to the singular time, since we may have chosen the wrong time $T = 1$ and the wrong center $x_0$ around which we rescaled. In Section \ref{sec:neckpinch-sing}, for every sufficiently large $x_0$, we adjust the rescaling time to $T = 1+t(x_0)$ and consider the rescaled solution $u^{\wcX}(\tilde{y},\theta, \tilde{\tau})$ around $\wcX = (x_0, 0,  1+t(x_0))$,  in order to improve the estimates obtained in Section \ref{sec-X0}. We show that the improved estimates for the graphical radius hold outside a sufficiently large compact set, extending up to and past the first singular time. Combining this with maximum principle arguments yields scale-invariant curvature estimates. We ultimately use these estimates to conclude the proof of Theorem \ref{thm-main0}, establishing that the flow starting from our initial data develops a neckpinch singularity.

 In Section \ref{sec-rotation}, we find an optimally fitting cylinder. Equivalently, we find an  appropriate rotation to apply to our flow $M_t$ such that when the rotated flow is written locally as a graph over $\mathcal{C}_{n,k}$, all angular derivatives of the graphical radius are very small and exponentially decaying in time. We also estimate the norm of this rotation and show that it is very close to the identity map.

 In Section \ref{sec-barriers}, we construct subsolutions and supersolutions to the rescaled MCF that we will later use as lower and upper barriers.

By Theorem \ref{thm-main0}, we know that for sufficiently small $\varepsilon$, the MCF starting with the initial data \eqref{eq-Ueps} develops a finite-time neckpinch singularity modeled on the cylinder $\mathcal{C}_{n,k}$, up to a possible dilation, translation and rotation. Our goal in Theorem \ref{thm-main} is to show that this neckpinch singularity is a nondegenerate neckpinch in the sense of Definition \ref{def-neckpinches}. More precisely, we aim to show that the upward quadratic bending of the initial data \eqref{eq-Ueps}, after performing a suitable Type I rescaling of the flow, propagates all the way to infinity. We cannot do this directly, as it is difficult to connect the local upward quadratic bending behavior initially with the asymptotic behavior of the solution at $+\infty$. Therefore, to achieve our goal, we treat the \emph{singular time} $t_0$, the \emph{center} $(x_0, x_0')$, and the \emph{rotation} $S$ of our initial cylinder $\mathcal{C}_{n,k}$ as parameters. We denote the vector whose components are these parameters by ${\bf \Omega}$, and the corresponding rescaled solution by $\widebar{M}_{\tau}^{\bf\Omega}$. 

In Section \ref{sec-L2-funnel}, we define the funnel $\mathcal{F}_{\tau}$ as the set of all rescaled mean curvature flow solutions whose norm of the projection onto unstable positive modes at any time $s \le \tau$ is controlled by $\delta  s^{-1}$. Then, we show that as long as our rescaled solution remains inside the funnel, we are able to propagate the upward quadratic bending into the asymptotics of our solution. 

In Section \ref{sec-shooting}, we employ degree theory and shooting-type arguments to show that there exists an ${\bf\Omega}$ such that the rescaled solution $\widebar{M}_\tau^{\bf\Omega}$ stays in the funnel $\mathcal{F}_{\tau}$ all the way up to infinity. Recall that choosing the right parameter ${\bf\Omega}$ corresponds to choosing the correct center of the neck, rescaling (dilating) the solution with respect to the correct singular time, and choosing the correct rigid rotation of the cylinder. In the same section, we prove Theorem \ref{thm-main}. 

Finally, using Theorem \ref{thm-main}, we provide the proof of Theorem \ref{thm-open-neck} in Section \ref{sec-stability}.

\subsection{Notation} 

We summarize here some notation that will be used throughout the paper.

\begin{itemize}

\item We  denote by 
 $\widebar \cX_0:= (\bar x_0, \bar x_0', T) $, where $p_0:=(\bar x_0, \bar x_0') \in \R^k \times \R^{n+1-k}$,
  the singularity point of $M_t$ (as stated in Theorem \ref{thm-main0}).

\smallskip 
 \item 
We denote by $\cX:= (x_0, x_0', t_0)$ any fixed point where $p_0:= (x_0, x_0') \in \R^k \times \R^{n+1-k}$,  $t_0 >0$.
We  use such points to center our rescaled flow. 
In most of the situations $t_0$ will be  close to the singularity time $T$. 

\smallskip
\item For a given  point $\cX=(x_0, x_0', t_0)$ with  $p_0:=(x_0, x_0') \in \R^k \times \R^{n+1-k}$  and a rotation $S$, 
 we  denote by $U^{\cX,S}(x,\theta, t)$ the profile   of the shifted and rotated flow   $S (M_t-p_0)$. 
 If  $S$ is the identity map and $p_0=0$, we have  $U^{\cX,S}(x, \theta, t)=U(x,\theta,t)$.

\smallskip

\item For a given  point $\cX=(x_0, x_0', t_0)$ with  $p_0:=(x_0, x_0') \in \R^k \times \R^{n+1-k}$  and a rotation $S$, we  denote by
$ {\widebar M}^{\cX,S}_\tau$ the parabolically rescaled  and rotated flow around  the point $\cX$ as defined in \eqref{eqn-tildeM}.
We will denote by  $u^{\cX,S}(y,\theta,\tau)$ the profile of the rescaled flow $ {\widebar M}^{\cX,S}_\tau$.

\smallskip 
\item In sections  \ref{sec-first-est}, \ref{sec-X0} and \ref{sec:neckpinch-sing}   we rescale our flow around  points 
  $\cX_0:= (x_0, x_0',1)$, $\cX_1:= (x_0, 0 ,1)$ and $\wcX:=(x_0,0, 1+t(x_0))$
respectively,  where $x_0 \in \R^k$, $x_0' \in   \R^{n+1-k}$ and $t(x_0)$ a small number depending on $x_0$.

\end{itemize} 

\bigskip 
{\bf Acknowledgments:} The first author  has been supported by the KIAS Individual Grant MG078902, and the National Research Foundation (NRF) grants RS-2023-00219980 and RS-2024-00345403 funded by the Korea government (MSIT). The second  author has been supported by the NSF grants DMS
1900702 and DMS 2454018.  The third author has been supported by the NSF grants DMS
2105508 and DMS 2505574.

\section{Graphical radius estimates for the rescaled flow around  $\cX_0=(x_0,x_0',1)$}
\label{sec-first-est}
 
 Given our initial condition \eqref{eq-Ueps} we do not know at first around what point $p_0$, and at which time $t_0$  the singularity occurs. We even do not know whether the MCF with initial data \eqref{eq-Ueps} develops a neckpinch singularity at the first singular time. In order to prove our result we first need to show that the MCF with initial data \eqref{eq-Ueps} develops a neckpinch singularity, and that up to a suitable rotation the singularity is modeled by the cylinder $C_{n,k}$ of radius $\rho = \sqrt{2(n-k)}$ defined in \eqref{eqn-cyl}.
 Note that this  shrinking cylinder, as a solution to the unrescaled mean curvature flow develops a singularity at time $ T = 1$. 
 
 In this section we will consider the graphical radius of the rescaled mean curvature flow around the center  $p_0 = (x_0,x_0') \subset \mathbb{R}^k\times \R^{n-k+1}$ and $T = 1$, since we expect the singular time to be close to one. By abusing the notation we will identify $x_0$ with  $p_0$, since it will be evident from the proof later that the choice of $x_0'$ is not important (we will show that after choosing the right shift along the $\mathbb{R}^k$ direction, i.e. $x_0$, and after choosing the right $T$, which appears to be the singular time, all spherical derivatives will exponentially converge to zero, as rescaled time converges to $+\infty$). Let us denote by 
 $\cX_0 := (x_0,x_0',1)$
  and the corresponding rescaled flow via \eqref{eq-rescaling} by $\widebar{M}_{\tau}^{\cX_0}$. The goal of this section is to prove a  unique continuation result, which will allow us to prove, starting from our initial data \eqref{eq-Ueps}, that   for a very long time interval $[0,\tau_{\ep}]$, and for  a very large ball in space
 $B_{R_\ep}$, the intersection  $\widebar{M}_{\tau}^{\cX_0}\cap B_{R_{\ep}}$ will remain a graph over $\cC_{n,k}$, and moreover will remain very close to it, for all $\tau\le\tau_{\ep}$.

 We say that $M\in C^k(\mathcal{C}_{n,k};r)$ if there is $U(x,\theta) \in C^k(\mathcal{C}_{n,k}\cap B_r)$ such that $M$ can be written as a graph $U(x,\theta)$ over $\mc C_{n,k}$, 
for all $|x| \leq r, \theta \in \mathbb{S}^{n-k}$. We can understand $U(x,\theta)$ as the  height function for $M$, over $\mathcal{C}_{n,k}\cap B_r$, measured from the $\mathbb{R}^k$-axis of the cylinder.   Similarly, we can define $\mathcal{H}(\mathcal{C}_{n,k};r)$ by 
\begin{equation}
\|U\|^2_{\mathcal{H}(\mathcal{C}_{n,k};r)}=\int_{|x| \leq r}U(x,\theta)^2\frac{1}{\sqrt{4\pi e}}e^{-\frac{|x|^2}{4}}dxd\theta.
\end{equation} 
Moreover, for easy of notation we denote
 \begin{equation}
 \|U\|_{C^k_r}=\| U \|_{C^k(\mathcal{C}_{n,k}\cap B_r)},
 \end{equation}
 and denote $\|U\|_{\mathcal{H}_r}$ in the same manner. Furthermore, we may abuse notation so that $\|M\|_{C^k_l} <+\infty$ implies $M \in C^k(\mathcal{C}_{n,k};r)$ with profile $U$ of $\|M\|_{C^k_l}=\|U\|_{C^k_l}$.

\bigskip

  \bigskip
  
\begin{proposition}[weak graphical radius]\label{thm:weak_rad}
Given $\Lambda>0$ and $\varepsilon_0\in (0,1)$, there is small $\varepsilon_1>0$ with the following significance. Let $\widebar{M}_\tau^{\cX_0}=(1-t)^{-\frac{1}{2}} \big(M_t - p_0)$ with $\tau=-\log (1-t)$ for some $p_0\in \mathbb{R}^{n+1}$, and assume $\text{Ent}(\widebar{M}_0^{\cX_0}) \leq \Lambda$. Suppose that there are $\varepsilon \in (0,\varepsilon_1)$ and $l \geq 30|\log \varepsilon|^{1/2}$ satisfying $\|\widebar{M}_0\|_{C^4_l} \leq \varepsilon$. Then,
\begin{equation}
\label{eq-first-bound}
\|\widebar{M}_\tau^{\cX_0} \|_{C^4(\{|y| \leq (l-l_\varepsilon)e^{\tau/2}\})}\leq \varepsilon_0
\end{equation}
holds for $\tau\leq -\frac{99}{100}\log \varepsilon$, where $l_\varepsilon=25|\log \varepsilon|^{1/2}$.
\end{proposition}

\begin{remark}\label{rem-everyX0}
Note that the conclusion of Proposition \ref{thm:weak_rad} holds for every rescaled flow $\widebar{M}_{\tau}^{\cX_0}$, around any center $\cX_0 = (x_0,x_0',1)$ 
with $ x_0  \leq B^k_{L_{\ep}/2}(0)$, where $L_{\ep}$ is as in the statement of Theorem \ref{thm-main}.
\end{remark}

We have the following consequence of pseudolocality (Theorem 1.5 in \cite{INS}), and the interior estimates of Ecker-Huisken (\cite{EH}), which is stated and proved in \cite{SX} and \cite{Sz}.

\begin{proposition}[unique continuation, \cite{INS}, \cite{EH}, \cite{SX}, \cite{Sz}]
\label{prop-unique-cont}
Suppose $\text{Ent}(\widebar{M}_{0})\leq \Lambda$. Given $\ve_0 > 0$, there exist $\ve_2, R_1 > 0$, depending on $\ve_0$ and $\Lambda$ with the following significance. Suppose $\widebar{M}_{\tau}$  is a rescaled mean curvature flow so that  $\widebar{M}_0$  is an $\ve$-graph, where $\ve \le \ve_2$, over $\mc{C}_{n,k}\cap B_R$, and 
$\|  \widebar{M}_0 \|_{C^4_R} < \ve$. Then, for every $\tau\in [0,10]$, $\widebar{M}_{\tau} $ is an $\ve_0$-graph over $\mc{C}_{n,k}\cap B_{e^{\tau/2}\, (R-R_1)}$, with $\|\widebar{M}_{\tau} \|_{C^4_{e^{\tau/2}\, (R-R_1)}} < \ve_0$.
\end{proposition}

\bigskip

Let us first set up some notation and compute some basic equations. Since the equations \eqref{eqn-cv}-\eqref{eqn-error-v} will also be used in the sections that follow, we consider  a   general point $  \cX=(x_0,x_0',t_0)$   and   let   $\widebar{M}^{ \cX}_\tau = (t_0-t)^{-\frac 12} \, (M_t - p_0)$ be the rescaled flow  which has profile 
\be\label{eqn-cv}
u^{\cX_0}(y,\theta,\tau) =  
\tfrac{U^{\cX}(x, \theta, t)}{\sqrt{t_0-t}}, \,\,\,\,  y= \tfrac{x}{\sqrt{ t_0-t }}, \,\, \,\,  \tau:= - \log ( t_0-t),
\ee
where $U^{\cX}$ is the profile of $M_t-(x_0,x_0')$, and let $v^{\cX} = u^{\cX} - \sqrt{2(n-k)}$. To  simplify the notation, denote  $u := u^{\cX}$ and $v:=v^{\cX}$.
 To compute the evolution equation for $u(y,\theta,\tau)$ first, let $\delta_{ij}$ stand for the standard Euclidean metric on $\mathbb{R}^k$, and let $\sigma_{\alpha\beta}$ be the standard round metric on the sphere $\mathbb{S}^{n-k}$. We define the area element (gradient factor) as
 \[W := \sqrt{1 + |\nabla _yu |^2 + \frac{|\nabla_{\theta} u|^2}{u^2}}.\]
Then we have
 \begin{equation}
 \label{eq-rescaled-u}
 \begin{split}
 u_{\tau} &= \Delta_y u + \frac{1}{u^2} \Delta_{\theta} u  - \frac{1}{2}\langle y, \nabla_y u\rangle - \frac{n-k}{u} + \frac{u}{2} - \frac{|\nabla_{\theta} u|^2}{u^3 W^2}\\
 &- \frac{1}{W^2}\, \left(\Hess_y u (\nabla_y u, \nabla_y u) + \frac{2}{u^2}\, \Hess_{y,\theta}u (\nabla_y u, \nabla_{\theta} u) + \frac{1}{u^4}\, \Hess_{\theta} u (\nabla_{\theta,}u, \nabla_{\theta} u)\right),
 \end{split}
 \end{equation}
 where
 \begin{equation}
 \label{eq-hess-yy}\Hess_y\, u(\nabla_y u, \nabla_y u) = \sum_{i,j=1}^k \frac{\partial^2 u}{\partial y_i\partial y_j}\, \frac{\partial u}{\partial y_i}\frac{\partial u}{\partial y_j},
 \end{equation}
 \begin{equation}
 \label{eq-hess-ytheta}
 \Hess_{y,\theta}\,u(\nabla_y u, \nabla_{\theta} u) = \sum_{i=1}^k \sum_{\alpha,\beta=1}^{n-k} \frac{\partial^2 u}{\partial y_i\partial \theta_{\alpha}} \frac{\partial u}{\partial y_i}\, \big(\sigma^{\alpha\beta}\frac{\partial u}{\partial \theta_{\beta}}\big),
 \end{equation}
 \begin{equation}
 \label{eq-hess-thetatheta}
 \Hess_{\theta}\, u (\nabla_{\theta} u, \nabla_{\theta} u) = \sum_{\alpha,\beta,\gamma,\delta=1}^{n-k} \Big(\frac{\partial^2 u}{\partial\theta_{\alpha}\partial \theta_{\beta}} - \Gamma_{\alpha\beta}^{\lambda} \frac{\partial u}{\partial \theta_{\lambda}}\Big)\, \sigma^{\alpha\gamma} \, \frac{\partial u}{\partial\theta_{\gamma}} \sigma^{\beta\delta}\, \frac{\partial u}{\partial \theta_{\delta}}.
 \end{equation}
 
  Note that $\Delta_{\mathcal{C}} u = \Delta_y u + \frac{1}{2(n-k)}\, \Delta_{\theta} u$, where $\mathcal{C} = \mathcal{C}_{n,k}$. It is straightforward to see that $v:= v^{\cX_0}$  then satisfies the equation 
\begin{equation}
\label{eqn-vtau}
v_{\tau} = \mathcal{L} v + \mathcal{E}(u,v,\nabla v, \nabla^2v),
\end{equation}
where $\cL f :=\Delta_{\mathcal{C}} f  - \frac12 \, \langle y, \nabla_y u\rangle  + f$ 
and 
\begin{equation}
\label{eqn-error-v}
\begin{split}
\mathcal{E}(u,v,\nabla v, \nabla^2v):=&   - \frac{v^2}{2u}  -v\, \Delta_{\theta} v\, \frac{u+q}{u^2 \, q^2} - \frac{|\nabla_{\theta} v|^2}{u^3 W^2}\\
 &- \frac{1}{W^2}\, \left(\Hess_y v (\nabla_y v, \nabla_y v) + \frac{2}{u^2}\, \Hess_{y,\theta}v (\nabla_y v, \nabla_{\theta} v) + \frac{1}{u^4}\, \Hess_{\theta} v (\nabla_{\theta}v, \nabla_{\theta} v)\right)
\end{split}
\end{equation}
where the quantities in the second line of equation \eqref{eqn-error-v} are defined as in \eqref{eq-hess-yy}, \eqref{eq-hess-ytheta} and \eqref{eq-hess-thetatheta}.

\smallskip

Now let us restrict to $\cX_0=(x_0, x_0',1)$, and  set 
$v^{\cX_0}:= u^{\cX_0} -  \sqrt{2(n-k)}$.  For a function  $\bar \vp \in C^{\infty}(\R)$ with $\bar \vp(z)=1$ for $|z| \le \frac12$ and $\bar \vp(z)=0$ for $|z| \ge 1$, and let 
\begin{equation}
\label{eqn-cutoff}
  \vp(y,\theta,\tau) := \bar \vp \big ( \frac{|y|}{ \rho(\tau)}\big ) \quad \mbox{and} \quad  \bv  (y,\theta,\tau) := v^{\cX_0}(y,\theta,\tau) \vp(y,\theta,\tau)
\ee 
where $(y,\theta) \in \mathbb{R}^k\times \mathbb{S}^{n-k}$, and where $ \rho(\tau) \geq 100$. We call $\rho(\tau) \ge 100$ an admissible graphical radius if
\begin{align}\label{eqn-graphical-radius}
& |\rho'(\tau)| \leq 100\, \rho(\tau), \qquad  \|v(\cdot,\tau) \|_{C^4_{\rho(\tau)}}\leq \varepsilon_0.
\end{align}

Denote $v:= v^{X_0}$. By \eqref{eqn-vtau} and \eqref{eqn-error-v} it follows that $\bv$ satisfies the equation 
\begin{equation}
\label{eqn-vomega}
\bv_{\tau} = \mathcal{L}\bar{v} + \vp \, \cE  + \mathcal{E}_{\vp},
\end{equation}
where $\cE:= \mathcal{E}(u,v, \nabla v, \nabla^2 v)$ and 
\be\label{eqn-Ephi}
\mathcal{E}_{\vp} := \big (\vp_\tau - \Delta_y \vp + \frac12\langle  y, \, \nabla_y \vp \rangle\big ) \, v - 2 \langle \nabla_y \vp, \nabla_y\, v\rangle
\ee
is the error term coming from the cut off function and is supported in the region  
$\frac{\rho(\tau)}2 \le |y| \le \rho(\tau)$. 

\bigskip

For any function $f(y,\theta)$ on $C_{n,k}$ we define, as  usual, the $L^2$-norm with respect to the Gaussian weight 
$$\| f \|_{\cH} = \Big ( \int_{\R^k} \int_{\mathbb{S}^{n-k}} f^2(y,\theta)  e^{-\frac{|y|^2}4} d\theta \, \, dy \Big )^{\frac 12}$$
and by $\langle \cdot,\cdot\rangle_\mathcal{H}$ the associated inner product.  It is well known that the operator  $\cL$ defined above is self-adjoint with respect to the 
inner product $\langle \cdot,\cdot\rangle_\mathcal{H}$.   For simplicity, we will abbreviate below and replace   $\|\cdot\|_\mathcal{H}$ and $\langle \cdot,\cdot\rangle_\mathcal{H}$ 
by $\|\cdot \|$ and $\langle \cdot,\cdot\rangle$.
\bigskip

\begin{lemma}\label{lemma:rough_L2_increase_rate}
There is an  $\varepsilon_0$ such that if $\rho(\tau)  \geq 100$ is an admissible graphical radius then
\begin{equation}
\tfrac{d}{d\tau} e^{-\frac{1001}{1000}\tau}\|\bar v(\cdot,\tau)\|_{\mathcal{H}} \leq e^{-\rho(\tau)^2/20}. 
\end{equation}
\end{lemma}

\begin{proof}
As in \eqref{eqn-vomega} we have
\[\frac{\partial}{\partial\tau} \bar v = \mathcal{L}\bar v + \bar{\mc E},\]
where we denote by $\bar{\mc E}=\vp\, \cE+\cE_\vp$ the error term. This implies
\[\frac 12\,\frac d{d\tau} \, \|\bar v\|^2 = \langle \mc{L}\bar v, \bar v\rangle + \langle \bar{\mc E}, \bar v\rangle.\]
Integration by parts and Cauchy-Scwarz inequality yield
\[\frac 12\,\frac d{d\tau} \, \|\bar v\|^2 \le -\|\nabla\bar v\|^2 + \|\bar v\|^2 + \|\bar{\mc E}\| \, \|\bar v\|.\]
On the other hand  we have  $\| \bar{\mc E} \| \leq \| \vp  \cE \| + \|\cE_\vp \|$, and  by using the $C^4$ estimate in \eqref{eqn-graphical-radius} and the fact that $\mathcal{E}_\varphi$ is supported on $\rho(\tau)/2 \leq |y| \leq \rho(\tau)$ we have
\begin{equation}
\| \cE_\vp \| \leq C \,  \varepsilon_0  \,  e^{-\rho(\tau)^2/16}.
\end{equation}
To estimate $\| \vp \cE \|$,  we  use \eqref{eqn-graphical-radius}  and  \eqref{eqn-error-v} to get 
\[ \| \vp \cE \| \leq C  \varepsilon_0 \,  \big ( \| \vp  \nabla_y v \| +  \| \vp \nabla_{\theta} v \|   + \| \bar v \| \big ) \leq 2 C  \varepsilon_0 \,  \big ( \| \nabla_y \bar v \| +  \frac {1}{\sqrt{2(n-k)}} \| \nabla_{\theta}\bar v \|   + \| \bar v  \|)+C\varepsilon_0 e^{-\rho(\tau)^2/16}.\]
Observe that,  similarly to estimating $\cE_\vp$, we have $\| \vp_y  v \| + \| \vp_\theta  v \| \leq C \,  \varepsilon_0  \,  e^{- \rho(\tau)^2/16}$.
Using this, while combining the  bounds 
for $\| \vp \cE\|, \| \cE_\vp\|$ and  observing that $\| \bv_y \| = \| \hat v_y \|$, $\| \bv_y \| = \| \hat v_y \|$, yields
\bee
\|  \bar{\mc E} \|  \leq 2 C \varepsilon_0 \, (\| \nabla_y \hat v \|^2 + \tfrac  1{2(n-k)} \| \nabla_{\theta} \hat v \|^2)^{\frac 12}  + 2C  \varepsilon_0 \| \bar v \| + C    \varepsilon_0  \,  e^{- \rho(\tau)^2/16}.
\eee
Namely,
\begin{equation}\label{eqn-error-rough}
\|\bar{\mc E}\| \le C_0 \ve_0\, \big(\|\nabla\bar v\| + \|\bar v\| + e^{-\rho(\tau)^2/16}\big),
\end{equation}
where $C_0$ is a uniform constant that may change from line to line, but in a uniform way. Combining this with Cauchy-Scwartz inequality yields
\[
\begin{split}
\frac12\, \frac{d}{d\tau}\, \|\bar v\|^2 &\le -\|\nabla\bar v\|^2 + \|\bar v\|^2 + C_0 \ve_0 \|\nabla\bar v\| \, \|\bar v\| + C_0\ve_0\, (\|\bar v\| + e^{-\rho(\tau)^2/16})\, \|\bar v\| \\
&\le (1 + C_0\ve_0)\, (\|\bar v\| + e^{-\rho(\tau)^2/16})\, \|\bar{v}\|.
\end{split}
\]
This implies
\[\frac{d}{d\tau} e^{-(1+C_0\ve_0)\tau}\, \|\bar v\| \le e^{-\rho(\tau)^2/20},\]
finishing the proof of Lemma by taking $\ve_0$ small enough.
\end{proof}

\bigskip

\begin{proposition}\label{prop:interior_regularity}
There exists $\delta_1 \in (0,1)$ and $C_1\geq 1$ only depending on $n$ with the following significance. Suppose that for $\tau \in [0,1]$, $M_\tau^{\cX_0}$ is the graph of $v(y,\theta,\tau)$ defined over $\mathcal{C}_{n,k}\cap B_{100}$, where
\begin{align}
&\|v\|_{C^4_{100}}\leq \varepsilon_0 \qquad \mbox{and} \qquad  \|\bar v \|_{\mathcal{H}}\leq 10\delta,
\end{align}
hold for $\tau \in [0,1]$ and for some $\delta \leq \delta_1$. Then,  we have
\begin{equation}
\|v(\cdot,1)\|_{C^4_{10}}\leq C_1\delta.
\end{equation}
Moreover, $100C_1\delta_1\leq \varepsilon_2$.
\end{proposition}

\begin{proof}
We recall that thanks to $\|v\|_{C^4_{100}}\leq \varepsilon_0$, $v$ is a solution to a uniformly parabolic equation. Since $\|\bar v \|_{\mathcal{H}}\leq \delta$ implies the $L^2$-bound for $v$, the standard parabolic interior estimates yield $\|v(\cdot,1)\|_{C^4_{10}}\leq C_1\delta$. We then choose $\delta_1 \leq \varepsilon_2/C_1$. 
\end{proof}

\bigskip

\begin{proof}[Proof of Theorem \ref{thm:weak_rad}]
Without loss of generality, we fix $p_0=0$ and denote  the rescaled MCF with respect to $(0,0,t_0)$ simply by $\widebar{M}_\tau$. 
 Also, we define $l':=(-20\log \varepsilon)^{\frac{1}{2}}+R_1$. We recall that Proposition \ref{prop-unique-cont} implies
\begin{equation}
\|\widebar{M}_\tau\|_{C^4_{(l-R_1)e^{1/2}}}\leq \varepsilon_0
\end{equation}
for $\tau \in [0,10]$. Thus, given $\cX_1=(x_0,0,1)$ with $x_0 \in B_{l-l'}^k(0)$, we have $B^k_{l'}(x_0)\subset B^k_{l}(0)$. Hence,
\begin{equation}
\|\widebar{M}^{\cX_1}_\tau\|_{C^4(\{|y| \leq (l'-R_1)e^{\tau/2}\})} \leq \varepsilon_0
\end{equation}
for $\tau \in [0,10]$. Also the given condition directly yields
\begin{equation}
\|\bar v^{\cX_1}(\cdot,0)\|_{\mathcal{H}} \leq C_2\ve,
\end{equation}
with $\rho=(l'-R_1)e^{\tau/2}$. Then Lemma \ref{lemma:rough_L2_increase_rate} provides
\begin{equation}
e^{-\frac{1001}{1000}}\|\bar v^{\cX_1}(\cdot,1)\|_{\mathcal{H}}\leq \|\bar v^{\cX_1}(\cdot,0)\|_{\mathcal{H}}+\varepsilon \leq (C_2+1)\varepsilon.
\end{equation}
We may assume $\varepsilon \ll \delta_1$, and apply Proposition \ref{prop:interior_regularity} to get
\begin{equation}
\|v^{\cX_1}(\cdot,1)\|_{C^4_{10}} \leq C_1(C_2+1)e^{\frac{1001}{1000}}\varepsilon \leq \varepsilon_2.
\end{equation}
Namely,
\begin{equation}
\|\widebar{M}_1\|_{C^4_{(l-l')e^{1/2}}}\leq \varepsilon_2.
\end{equation}
Hence, applying Proposition \ref{prop-unique-cont} again yields
\begin{equation}
\|\widebar{M}_\tau\|_{C^4_{[(l-l')e^{1/2}-R_1]e^{(\tau-1)/2}}}\leq \varepsilon_0
\end{equation}
for $\tau \in [1,11]$. Therefore, for each $\cX_1$ with $|x_0|\leq l-(1+e^{-\frac{1}{4}})l'$, we have
\begin{equation}
\|\widebar{M}^{\cX_1}_\tau\|_{C^4(\{|y| \leq (l'-R_1)e^{\frac{1}{4}+\frac{1}{2} (\tau-1)}\})} \leq \varepsilon_0
\end{equation}
for $\tau \in [1,11]$. Thus, Lemma \ref{lemma:rough_L2_increase_rate} with $\rho=(l'-R_1)e^{ \frac{1}{4}+\frac{1}{2}(\tau-1)}$ provides
\begin{equation}
e^{-\frac{1001}{1000}}\|\bar v^{\cX_1}(\cdot,2)\|_{\mathcal{H}}\leq \|\bar v^\cX(\cdot,1)\|_{\mathcal{H}}+ C_0\,\varepsilon^{e^{1/2}}\leq C_1(C_2+1)e^{\frac{1001}{1000}}\varepsilon+\varepsilon^{\frac{3}{2}}.
\end{equation}
We may assume
\begin{equation}
 \|\bar v^{\cX_1}(\cdot,2)\|_{\mathcal{H}}\leq C_1(C_2+1)e^{2p}\varepsilon+e^{p}\varepsilon^{\frac{3}{2}} \leq \varepsilon_2,
\end{equation}
where $p:=1001/1000$. Then,  Proposition \ref{prop-unique-cont} and Proposition \ref{prop:interior_regularity}  imply
\begin{equation}
\|\widebar{M}_\tau\|_{C^4_{[(l-(1+e^{-1/4})l')e-R_1]e^{(\tau-2)/2}}}\leq \varepsilon_0
\end{equation}
for $\tau\in [2,12]$. Hence, for each $\cX_1$ with $|x_0| \leq l-(1+e^{-\frac{1}{4}}+e^{-\frac{1}{2}})l'$ with $\rho=(l'-R_1)e^{\frac{1}{2}+\frac{1}{2}(\tau-2)}$, we repeat the above process so that we get
\begin{equation}
 \|\bar v^{\cX_1}(\cdot,3)\|_{\mathcal{H}}\leq C_1(C_2+1)e^{3p}\varepsilon+e^{2p}\varepsilon^{\frac{3}{2}}+e^{p}\varepsilon^{3}.
\end{equation}
After $k$-time repeating this process, for $\cX$ with $|x_0| \leq l-(1+\cdots+e^{-\frac{k}{4}})l'$ with $\rho=(l'-R_1)e^{\frac{k}{4}+\frac{1}{2}(\tau-k)}$, we have
\begin{equation}
 \|\bar v^{\cX_1}(\cdot,k+1)\|_{\mathcal{H}}\leq C_1(C_2+1)e^{(k+1)p}\varepsilon+e^{kp}\varepsilon^{\frac{3}{2}}+\cdots+e^{p}\varepsilon^{\frac{3}{2}k}.
\end{equation}
By choosing small enough $\varepsilon$ satisfying $e^{-p}\varepsilon^{\frac{1}{2}}\leq \frac{1}{2}$, we have
\begin{equation}
 \|\bar v^{\cX_1}(\cdot,k+1)\|_{\mathcal{H}}\leq [C_1(C_2+1)+1]e^{(k+1)p}\varepsilon.
\end{equation}
Hence, we can repeat this process for $\tau \leq k$ satisfying
\begin{equation}
e^{\frac{1001}{1000}k}\varepsilon \leq e^{-\frac{1001}{1000}}[C_1(C_2+1)+1]^{-1}\varepsilon_2:=\varepsilon_3.
\end{equation}
Namely
\begin{equation}
\tau \leq \tfrac{1000}{1001} (\log \varepsilon_3 -\log \varepsilon).
\end{equation}
We can choose small $\varepsilon$ to satisfy
\begin{equation}
-\tfrac{99}{100}\log \varepsilon \leq \tfrac{1000}{1001} (\log \varepsilon_3 -\log \varepsilon).
\end{equation}
Then, for each $\cX_1$ with $|x_0| \leq l-5l'\leq l-25|\log \varepsilon|^{1/2}$, we have
\begin{equation}
 \|\bar v^{\cX_1}(\cdot,\tau)\|_{C^4_{10}}\leq \varepsilon_2.
\end{equation}
for $\tau \leq -\frac{99}{100}\log \varepsilon$. This completes the proof.
\end{proof}

\bigskip

\bigskip

 \bigskip

\section{Improved local estimates for cylindrical radius of the  flow rescaled around  $\cX_1=(x_0,0,1)$} \label{sec-X0} 

In this section we will employ  $L^2$-theory and use the  results from section \ref{sec-first-est} to improve the closeness of the rescaled flow 
$\widebar {M}_\tau^{\cX_1}$   around a point 
 $\cX_1 :=
(x_0, 0,1)$ to the cylinder $\cC_{n,k}$ on a large ball, for a large time interval. In Proposition \ref{thm:weak_rad} this closeness is given in terms of $\varepsilon_0$ and this $\ep_0$ can be much larger than $\ep$ or any power of it (recall that we use $\ep$ to define the initial data in \eqref{eq-Ueps}). In this section, we quantitatively refine this closeness by proving that it is controlled by a power of $\ep$. This bound plays a key role in the subsequent section, where we show that the unrescaled flow develops a neckpinch singularity.

The following is the goal of this section.
\begin{theorem}[local radius bound]\label{the:local_radius}
Let  $U(x,\theta,t)$ be  the profile of a mean curvature flow $\{ M_t \}_{t \in [0,T)}$  that satisfies the assumptions of Theorem \ref{thm-main0},  in particular   $U(x,\theta,0)=\sqrt{2(n-k)}+\varepsilon \,|x|^2+o(\varepsilon)\, (1 + |x|^2)$,  for $ |x| \leq L_{\ve}$
  with $L_{\ve} \ge L$ and $\varepsilon^{-\frac{1}{500}}\leq L\leq \varepsilon^{-\frac{1}{300}}$. Let $x_0$ satisfy $|x_0| \le \frac L2$.  
Then, the rescaled profile $u^{\cX_1}$ centered at $\cX_1=(x_0, 0,1)$ defined by  \eqref{eqn-cv}  satisfies the $L^\infty$ bound\be\label{eqn-LinftyL} \|  u^{\cX_1}(y, \theta, \tau)  - \sqrt{2(n-k)} \|_{L^\infty(\{|y| < \ell (\tau)\})}
 < C\,   \varepsilon^{\frac{1}{102}} \, \sup_{ |x-x_0| < \ell (\tau)e^{-\frac \tau2}} \, (1+|x|^2)
\ee 
for all  $\ell(\tau)$ continuous such that $0 <   \ell(\tau) < \frac L3$, and all $\theta \in \mathbb{S}^{n-k}$, $\tau \leq  - \frac{99}{100} \log \ep$,  where  
$C$ is a   numeric constant. 

Consequently, the following $C^4$ bound  holds 
\be\label{eqn-C4L}
\big \|   u^{\cX_1}(y, \theta, \tau)  - \sqrt{2(n-k)} \big \|_{C^4(\{|y| < \ell (\tau)\})}< C_0 \, \varepsilon^{\frac{1}{102}} \, \sup_{ |x-x_0| < \ell (\tau)e^{-\frac \tau2} } \, (1+|x|^2)
\ee
provided $0 < \ell (\tau)  < \frac L3$,  $\theta \in \mathbb{S}^{n-k}$ and $\tau \leq  - \frac{99}{100} \log \ep$, where   $C_0$ is  another  numeric constant.  
\end{theorem}

\begin{remark} Note that we need the right hand side in \eqref{eqn-LinftyL} to be small, and indeed it is  easy to check that  our  choices  of $x_0$ and $L$
guarantee that it is less than $\ep^{\frac 1{320}}$.  
\end{remark}
\bigskip

Let $v^{\cX_1} := u^{\cX_1} - \sqrt{2(n-k)}$, where $u^{\cX_1}$ is the rescaled profile at $\cX_1:=(x_0,0,1)$  (the same point 
as in the statement of Theorem  \ref{the:local_radius}). Recall that   $u^{\cX_1} (y,\theta,\tau) = \frac{U(x,\theta, t)}{\sqrt{1-t}}$,  with $y=\frac{x-x_0}{\sqrt{1-t}}$, $\tau:=- \log (1-t)$. 

For a function  $\bar \vp \in C^{\infty}(\R)$ with $\bar \vp(z)=1$ for $|z| \le \frac12$ and $\bar \vp(z)=0$ for $|z| \ge 1$, and define the cut-off function 
 $\vp: \R^k \times \mathbb{S}^{n-k} \to \R$ by 
\begin{equation}
\label{eqn-cutoff1}
  \vp(y,\theta,\tau) := \bar \vp \big ( \frac{|y|}{ \ell(\tau)}\big ) \qquad \mbox{with} \quad  \ell(\tau) := 10 |\log \varepsilon|^{\frac 12} \, e^{\frac \tau2}. 
 \ee 

For simplicity denote $v^{\cX_1}$ by $v$.  The function $v$ satisfies equation  \eqref{eqn-vtau} with error \eqref{eqn-error-v}. Set   $ \bar v(y,\theta,\tau) := v(y,\theta,\tau) \vp(y,\theta,\tau)$. Then,  $\bar v$ is supported on the set  
$\{(y,\theta) \in \R^k\times \mathbb{S}^{n-k} \,\,\,|\,\,\, |y| \leq \ell(\tau)\}$ and satisfies the equation \eqref{eqn-vomega}.
Similarly to the previous section we consider the projections

\begin{equation}\label{eqn-abdef} 
a(\tau) := \frac{\langle \bv,1 \rangle_{\mathcal{H}}}{\|1\|^2_{\mathcal{H}}}, \qquad b(\tau):= \| \bv - a(\tau) \|_{\mathcal{H}}, \qquad \hat v :=  \bv - a(\tau). 
\end{equation}

\smallskip 

\begin{lemma}\label{eqn-a0b0}
 Set $a_0=a(0)$ and $b_0 = b(0)$. Then we have:  $a_0 = A_0\,  \ep \,  ( 2k +  |x_0|^2)\, (1+o(1))$ and $b_0 = \ep \sqrt{ B_1 + B_2 \, |x_0|^2}\, (1+o(1))$, 
for  positive numeric constants $A_0, B_1, B_2$, where $o(1)\to 0$ as $\varepsilon\to 0$.

\end{lemma}

\begin{proof}
We use $U(x,\theta,0) =\sqrt{2(n-k)}+\varepsilon \,|x|^2 +o(\varepsilon)\, (1+|x|^2)$. 
Then by  \eqref{eqn-cv},
 the function $v := u^{\cX_1} - \sqrt{2(n-k)}$ satisfies 
\[v(y,\theta, 0) =  \varepsilon \, \sum_{i=1}^k \big (y_i + (x_0)_i \big )^2 \, (1+o(1)) = \ep\,  \Big(  |x_0|^2 + 2 k +\sum_{i=1}^k (y_i^2-2) - 2y_i (x_0)_i\Big)\, (1+o(1))\]
which  readily implies that   $a_0 = A_0\,  \ep \,  ( 2k +  |x_0|^2)\, (1+o(1))$ and $b_0 = \ep \sqrt{ B_1 + B_2 \, |x_0|^2}\, (1+o(1)) $, 
for a positive numeric constants $A_0, B_1, B_2$.
\end{proof}  

Assume now that the conditions of Theorem \ref{thm-main} are fulfilled and that the initial data satisfies \eqref{eq-Ueps} for sufficiently small $\varepsilon$. Let $\ep_0 \in (0,1)$ be a small number. Then by Proposition \ref{thm:weak_rad} and Remark \ref{rem-everyX0}, for sufficiently small $\ep > 0$, keeping in mind that $v = v^{X_0}$, we have 
\be\label{eqn-ck} \| v \|_{C^k (B_{\ell(\tau)})} \leq \varepsilon_0, \qquad \mbox{for}\,\,  
 \tau < \tau_\varepsilon:= - \tfrac{99}{100} |\log \varepsilon|,  \qquad \mbox{where}  \,\,  \,  \ell(\tau) = 10 |\log \varepsilon|^{\frac 12} \, e^{\frac \tau2}. 
\ee

Before we prove Theorem \ref{the:local_radius},  we will first prove a sequence of lemmas. In  Lemma \ref{claim-ab}  below we first  compute the system of ODEs for $a(\tau)$ and $b(\tau)$, which are defined as in \eqref{eqn-abdef}. In Lemma \ref{lem-ratio} we prove that the ratio  $a(\tau)/b(\tau)$ is bounded from below by a fixed positive constant.  As a  consequence,  in Lemma \ref{prop-local-radius} we  give a preliminary  local radius estimate, which we later improve in the proof of Proposition \ref{the:local_radius}. 
\smallskip 

\begin{lemma}\label{claim-ab}
The projections $a(\tau)$ and $b(\tau)$ satisfy 
\begin{equation}\label{eqn-a}
|  \tfrac{d}{d\tau} a  - a |   \leq   C_0\, \varepsilon_0 \, (|a|+b)  +
C_0 \,  \varepsilon_0  \,  e^{- 10 |\log \varepsilon|\, e^{\tau}}
\end{equation}
and 
\begin{equation}\label{eqn-b}
\tfrac{d }{d\tau} b  \leq \tfrac 12 \, b + C_0\,  \varepsilon_0 \, (|a|+b) + C_0 \,  \varepsilon_0  \,  e^{- 10 |\log \varepsilon|\, e^{\tau}}
\end{equation}
where $C_0$ is a   numeric constant. 

\end{lemma}

\begin{proof}

First observe that one can easily   bound  
\be\label{eqn-Ephi2} \| \cE_\vp  \|_{\cH}  \leq  C_1 \,  \varepsilon_0  \, e^{- 10 |\log \varepsilon|\, e^{\tau}}
\ee
by using the $C^k$ estimate in \eqref{eqn-ck} and the fact that $\cE_\vp$ is supported on $\ell(\tau)/2 \leq |y| \leq \ell(\tau)$.

\smallskip 
\noindent {(i)} {\em  Let us first  prove } \eqref{eqn-a}. For this it is sufficient to bound   $| \langle \vp \, \cE + \cE_\vp, 1 \rangle |$. Let us
focus on bounding  $| \langle \vp \, \cE, 1 \rangle |$. All terms in $\cE$ can be treated similarly, so let us just estimate couple of them. Using that $\bar{v} = \vp v$, and that $\vp$ is a function of $y$, independent of $\theta$, we get
\begin{equation}
\label{eq-error-theta}
\begin{split}
 \Big\langle \frac{\vp}{W^2\, u^2}\, \Hess_{y,\theta} \,v \,(\nabla_y v,\nabla_{\theta} v), 1 \Big\rangle &=
 \Big\langle \frac{\vp}{W^2\, u^2}\, \Hess_{y,\theta} \,v \,(\nabla_y v,\nabla_{\theta} v), 1 \Big\rangle\\&=
\int_{\cC_{n,k}} \frac{\vp}{W^2\, u^2}\, \Hess_{y,\theta} \,v \,(\nabla_y v,\nabla_{\theta} v)\, e^{-\frac{|y|^2}{4}}\, dy\, d\theta \\
&= \int_{\mathbb{R}^k} e^{-\frac{|y|^2}{4}}\, \int_{\mathbb{S}^{n-k}} \Big(\frac{1}{2W^2\, u^2}\, \nabla_{\theta}\, (|\nabla_y v|^2), \nabla_{\theta}\bar{v}\Big)_{\sigma} d\theta\, dy \\
&= -\int_{\mathbb{R}^k} \int_{\mathbb{S}^{n-k}} \bar{v}\, \operatorname{div}_{\theta}\Big(\frac{1}{2W^2\, u^2}\, \nabla_{\theta} \,( |\nabla_y v|^2)\Big) e^{-\frac{|y|^2}{4}}\, d\theta,dy
\end{split}
\end{equation}
where $\Big(\cdot,\cdot\Big)_{\sigma}$ is exactly the intrinsic spherical inner product, and where we integrated by parts on $\mathbb{S}^{n-k}$. Using Cauchy-Scwartz inequality above and the $C^k $ bound in \eqref{eqn-ck} we easily get  
\[\Big|  \Big\langle \frac{\vp}{W^2\, u^2}\, \Hess_{y,\theta} \,v \,(\nabla_y v,\nabla_{\theta} v), 1 \Big\rangle \Big| \leq C \, \varepsilon_0 \, \sqrt{a^2 + b^2},\]
for a numeric constant $C$, since $\|\bar{v} \|^2 = |a(\tau)|^2\, \|1\|^2 + b(\tau)^2$. On the other hand $\Big\langle \frac{\vp}{W^2} \Hess_y v\, (\nabla_y v, \nabla_y v), 1\rangle$ can be written as $\Big\langle \frac{1}{W^2} \Hess_y v\, (\nabla_y \bar{v}, \nabla_y v), 1\rangle$ plus the terms involving derivatives in $\vp$ that can be estimated similarly as $\|\mc{E}_{\vp}\|$ in \eqref{eqn-Ephi2}. Note that the term $\Big\langle \frac{1}{W^2} \Hess_y v\, (\nabla_y \bar{v}, \nabla_y v), 1\rangle$ can be estimated similarly as in \eqref{eq-error-theta}, after performing intergration by parts over $\mathbb{R}^k$.

All the other terms in $\langle\vp\, \mc{E} , 1\rangle$ can be estimated similarly,  hence leading to  $| \langle \vp \, \cE, 1 \rangle |  \leq C_2 \, \varepsilon_0 \, (|a| + b)$
for a different numeric constant $C$. By combining this estimate with \eqref{eqn-Ephi2} we obtain the bound 
$$| \langle \vp \, \cE + \cE_\vp, 1 \rangle | \leq C_2 \,  \varepsilon_0 \,  (|a| + b) + C_1 \,  \varepsilon_0  \, e^{- 10 |\log \varepsilon|\, e^{\tau}}$$
which leads  to \eqref{eqn-a} for a sufficiently large numeric constant $C_0$.

\medskip
\noindent {(ii)} {\em We will now  prove}  \eqref{eqn-b}.  We call $\hat \pi$ the projection onto the  eigenspace generated by  all eigenfunctions of $\cL$ corresponding to eigenvalues $\lambda_{m, n} \leq 1/2$.
Note that the only eigenvalue strictly greater  than $1/2$  is the eigenvalue $1$ with corresponding eigenvector $1$.  
Applying the projection $\hat \pi$ to both sides of the equation
$ \bv_\tau - \cL \bv = \cE + \cE_\vp $ we get that $\hat v:= \bv - \frac{\langle  \bv, 1 \rangle}{\|1\|^2}$ satisfies 
$
\hat v_\tau  = \cL\hat v +  \hat \cE
$
where $\hat \cE (\cdot, \tau) =  \hat \pi (\vp \, \cE + \cE_\vp) (\cdot, \tau)$.  This implies \be\label{eqn-di1}
  \tfrac 12\tfrac{d}{d\tau} \|\hat v\|^2 = \langle \hat v, \cL \hat v \rangle +
  \langle \hat v, \hat \cE \rangle.
\ee
Using the spectral decomposition of $\hat v$ we get 
$$\langle \hat v, \cL \hat v \rangle  \leq \tfrac 12 \, \|\hat v\|^2.$$
On the other hand, integration by parts gives
\be\label{eqn-LvL} \langle \hat v, \cL \hat v \rangle = - \| \nabla_y\hat v \|^2 -  \tfrac {1}{2(n-k)} \, \| \nabla_{\theta}\hat v \|^2 + \| \hat v \|^2.
\ee
Combining the two inequalities we get
\be\label{eqn-v23}- \| \nabla_y\hat v \|^2  - \tfrac {1}{2(n-k)} \, \| \nabla_{\theta}\hat v \|^2 + \| \hat v \|^2 \leq \tfrac 12 \, \|\hat v\|^2 \qquad \mbox{that is } \qquad \| \nabla_y\hat v \|^2 +  \tfrac {1}{2(n-k)} \, \| \nabla_{\theta}\hat v \|^2  \geq \tfrac 12 \, \|\hat v\|^2.\ee
For   a small $\delta >0$ (to be chosen at the end of the proof)  we have  $- (1-\delta) \big (  \| \nabla_y\hat v\|^2 + \tfrac {1}{2(n-k)} \| \nabla_{\theta}\hat v \|^2 \big )  \leq - \frac{1-\delta}2 \, \| \hat v \|^2$. Inserting this bound in \eqref{eqn-LvL} 
we obtain 
\be  \langle \hat v, \cL \hat v \rangle = - ( \| \nabla_y\hat v \|^2 + \tfrac {1}{2(n-k)} \, \| \nabla_{\theta}\hat v \|^2 ) +  \| \hat v \|^2 \leq \tfrac{1+\delta}2 \,  \| \hat v\|^2 - \delta \big (  \| \nabla_y\hat v\|^2 + \tfrac {1}{2(n-k)} \, \| \nabla_{\theta}\hat v \|^2\big ).\ee
Therefore  \eqref{eqn-di1}, $( \| \nabla_y\hat v \|^2  + \tfrac {1}{2(n-k)} \, \| \nabla_{\theta}\hat v \|^2 )^{\frac 12}  \geq \tfrac 1{\sqrt{2}}  \, \| \hat v\|$ (see \eqref{eqn-v23}) and $\langle \hat v,  \hat \cE  \rangle \leq \| \hat v \| \, \| \hat \cE \|$, give 
\be\label{eqn-di2}
\begin{split}
\| \hat v\|  \cdot  \tfrac{d}{d\tau} \|\hat v\| &\leq \| \hat v \| \cdot 
 \Big (  \tfrac{1+\delta}2\,  \| \hat v\|  - \tfrac \delta {\sqrt{2}}\,  \big(\| \nabla_y\hat v\|^2 + \tfrac {1}{2(n-k)} \, \| \nabla_{\theta}\hat v \|^2\big)^{\frac12}  + \| \hat \cE \| \Big )   \end{split}  
  \ee
On the other hand  we have  $\| \hat \cE \| \leq \| \vp  \cE \| + \|\cE_\vp \|$, where  $\| \cE_\vp \| \leq C_3 \,  \varepsilon_0  \,  e^{- 10 |\log \varepsilon|\, e^{-\tau}}$, by  \eqref{eqn-Ephi2}. 
To estimate $\| \vp \cE \|$, as in the proof of Lemma \ref{lemma:rough_L2_increase_rate} we employ the $C^k$ estimate in \eqref{eqn-ck},  and  \eqref{eqn-error-v} to get 
$$ \| \vp \cE \| \leq C  \varepsilon_0 \,  \big ( \| \vp  \nabla_y v \| +  \| \vp \nabla_{\theta} v \|   + \| \bar v \| \big ) \leq 2 C  \varepsilon_0 \,  \big ( \| \nabla_y \bar v \| +  \frac {1}{\sqrt{2(n-k)}} \| \nabla_{\theta}\bar v \|   + \| \bar v  \| ).$$
Observe that,  similarly to estimating $\cE_\vp$, we have $\| \vp_y  v \|  \leq C_4 \,  \varepsilon_0  \,  e^{- 10 |\log \varepsilon|\, e^{-\tau}}$.
Using this, while combining the  bounds 
for $\| \vp \cE\|, \| \cE_\vp\|$ and  observing that $\| \bv_y \| = \| \hat v_y \|$, $\| \bv_\theta \| = \| \hat v_\theta \|$, yields
\be\label{eqn-di3}
\| \hat \cE \|  \leq 2 C \varepsilon_0 \, (\| \nabla_y \hat v \|^2 + \tfrac  1{2(n-k)} \| \nabla_{\theta} \hat v \|^2)^{\frac 12}  + 2C  \varepsilon_0 \| \bar v \| + C_5    \varepsilon_0  \,  e^{- 10 |\log \varepsilon|\, e^{-\tau}}
\ee
for numerical constants $C, C_5$. Choose $\delta := 4 C \varepsilon_0$ to absorb the bad positive  term $2C \varepsilon_0 \, (\| \nabla_y v \|^2 + \tfrac  {1}{2(n-k)} \| \nabla_{\theta} v \|^2)^{\frac 12}$ that appears in $\hat \cE$  into  the negative term $- \tfrac \delta {\sqrt{2}}\, ( \| \nabla_y \hat v\|^2  + \tfrac {1}{2(n-k)} \, \| \nabla_{\theta}\hat v \|^2)^{\frac 12}$. Combine \eqref{eqn-di2} and \eqref{eqn-di3}, while dividing \eqref{eqn-di2} by $\| \hat v \|$ to conclude the estimate
\be\label{eqn-di4}
\tfrac{d}{d\tau} \|\hat v\| \leq  \tfrac 12 \,  \| \hat v\|  +   6 C \varepsilon_0 \| \bar v \| + C_5    \varepsilon_0  \,  e^{- 10 |\log \varepsilon|\, e^{\tau}}
\ee
which readily implies \eqref{eqn-b} \big (for a uniform  large constant $C_0 \geq \max (6C, C_5)\big )$, by   using that   $b := \| \hat v \|$ and $\|\bar v \|^2 = a^2\, \|1\|^2 + b^2$. 
We can certainly take $C_0$ sufficiently large so that estimate \eqref{eqn-a} also holds.

\end{proof}


 \begin{lemma} \label{lem-ratio} Let $B_\varepsilon(\tau) :=    e^{- 10 |\log \varepsilon|\, e^{\tau}}$ and  $\hb (\tau) := b(\tau) + B_\varepsilon(\tau).$ Then we have  \be\label{eqn-ration} r(\tau):= \frac{a(\tau)}{\hb(\tau)}  >   c_0, \qquad \mbox{for all } \,\, 0 \leq \tau \leq -\frac{99}{100}\, \log\ep
\ee 
for a numeric  constant $c_0 >0$.   

 \end{lemma} 
 
 \begin{proof}
Recall \eqref{eqn-a} and \eqref{eqn-b} and the definitions of $B_\ep(\tau)$ and $\hat b(\tau)$ at the statement of the lemma. Since   $B_\varepsilon'(\tau) <0$ we have 
$\hb'(\tau)=  b'(\tau)   +   B_\varepsilon'(\tau) < b'(\tau)$ and therefore   by \eqref{eqn-b} we get 
\bee
\begin{split}
\hb'(\tau) &< b'(\tau)   \leq    \tfrac 12 \, b + C_0\,  \varepsilon_0 \, (|a|+b) +C_0\, \varepsilon_0  \, B_\varepsilon(\tau) \\
&\leq \tfrac 12 \, \hb +   C_0\,  \varepsilon_0 \, (|a|+\hb) - \tfrac 12 B_\varepsilon(\tau)  - C_0 \, \varepsilon_0 \, B_\varepsilon(\tau)  + C_0\, \varepsilon_0 \, B_\varepsilon(\tau)
\end{split}
\eee
that is
\be\label{eqn-hb}
\hb'(\tau) < \tfrac 12 \, \hb +   C_0\,  \varepsilon_0 \, (|a|+\hb).
\ee
On the other hand by \eqref{eqn-a} we have 
\be\label{eqn-a2} a'(\tau)   \geq    a - C_0 \, \varepsilon_0 \, (|a|+b)  - 
C_0 \,  \varepsilon_0  \,  B_\varepsilon = a - C_0 \, \varepsilon_0 \, (|a|+\hb)
\ee

\smallskip 
First, by Proposition \ref{eqn-a0b0} we have $a_0 := a(0) \sim A_0 \ep \, (2k+|x_0|^2) > 0$ and $b_0 := b(0)\sim \ep \sqrt{1+|x_0|^2}$. 
Since $B_\varepsilon(0) = e^{- 10 |\log \varepsilon|} = \varepsilon^{10} \ll \varepsilon$
we have $\hat b_0 := b(0) + B_\ep(0)\sim b_0$ and therefore $r_0 :=  r(0) = a_0 /\hat b_0  \sim \sqrt{1+|x_0|^2}$, that is $r_0 \geq c_0 >0$ for a constant $c_0$ that will be defined later in the proof.
We further claim that $a(\tau) > 0$ and $\frac{a(\tau)}{b(\tau)} \ge c_0$ remain to hold along the flow for all $\tau \le -\frac{99}{100}\log\ep$. Assume there exists the first time $\tau_1 < - \frac{99}{100}\log\ep$ so that $a(\tau) > 0$, for $\tau\in [0,\tau_1)$, but $a(\tau_1) = 0$.
To compare $a(\tau)$ and $\hat b(\tau)$ for $\tau\in [0,\tau_1)$ we consider the ratio  $r (\tau):= \tfrac {a(\tau)}{\hb(\tau)}$. 
By direct computation, 
\bee
\begin{split}
\frac{d}{d\tau}  \Big ( \frac a \hb \Big ) &=  \frac{a_\tau} \hb - \frac a\hb \cdot  \frac{\hb_\tau}\hb \geq \frac{a  - C_0\ \varepsilon_0 \, (a+\hb) } \hb - \frac a\hb \cdot 
\, \frac{\tfrac 12 \, \hb +   C_0 \varepsilon_0 \, (a+\hb)}\hb\end{split}
\eee
and in terms of $r(\tau)$, we get $ r'(\tau)   \geq r   - C_0\, \varepsilon_0 \, (1+r)  - \frac{r}2  - C_0\, \varepsilon_0 \,r \, (1+ r ).$
To simplify the notation, call  $K:= C_0\varepsilon_0$ and write the above as 
\be\label{eqn-la}
r'(\tau) \geq  - \big (  K  r^2  +   (   2K - \tfrac 12 ) \,  r + K  \big )= K \, (r- r^-)(r^+ - r) 
\ee
where $r^\pm$ are the roots of $K  r^2  +   (  2K -  \tfrac 1{2}  ) \,  r + K=0$,  that is, by direct computation, $r^\pm = \tfrac{(   \frac 12 - 2 K  ) \pm \sqrt{   \tfrac 1{4} -  2K }}{2K}$. 

Both $r^\pm$ are positive, and  since $K := C_0 \varepsilon_0$, we have 
$r^+ \approx \frac 1{2K} \sim  \varepsilon_0^{-1} $.  For $r^-$, using   $\sqrt{   \frac 1{4} -  2K } = \frac 12 \sqrt{1-8K} \approx \frac 12 \, (1- 4 K - 8 K^2 + O(K^3))$, 
we get $r^- \approx 2K  \sim \ep_0$. Summarizing, we have: (a) $r_0 \geq c_0 >0$  and (b) $ r^+ \sim \ep_0^{-1}$, $ r^- \sim \ep_0$. Therefore, we also have $r^- \ll r_0$. 

Let $R(\tau)$ be the solution to the IVP
\be\label{eqn-R} 
R'(\tau) = K \, (R- r^-)(r^+ - R), \qquad R(0) = r(0) =: r_0. 
\ee
Then,  by the comparison principle $r(\tau) \geq R(\tau)$.  We separate the following two cases:

\smallskip

\noindent (i) {\em Case $r_0 \geq  r^+$.}   In this case $R(\tau)  \geq r^+ $, for all $0 \leq  \tau \leq \tau_\ep$.

\smallskip

\noindent (ii) {\em Case $r^- \ll r_0 < r^+$.}  By standard ODE theory,  the solution $R(\tau)$ is increasing 
and converges  to the equilibrium $r^+$,  as $\tau \to +\infty$. Hence 
 $R(\tau)  \geq r_0 \geq c_0$,  for all $0 \leq \tau < \tau_\ep$,  where we have seen that we can take $c_0>0$ to be a numeric constant independent of $\ep_0$.

Combining the two cases (i) and (ii) we get $r(\tau) \geq R(\tau) \geq \min (r^+, r(0)) = \min (r^+, \sqrt{1+|x_0|^2} =: c_0$ (where $r^+ \sim \ep_0^{-1}$), for all $\tau\in [0, \tau_1)$, implying that $a(\tau_1)\neq 0$, and hence $a(\tau) > 0$ for all $\tau\le -\frac{99}{100}\, \log\ep$.  The same proof above then shows $r(\tau) \ge c_0$ for all $\tau\le -\frac{99}{100}\, \log\ep$.
This finishes the proof of the lemma.  

\end{proof}

As a consequence of the previous two Lemmas,  we have  the following local radius estimate.

\begin{lemma}\label{prop-local-radius}
Under the assumptions of Proposition \ref{the:local_radius}, we have  \begin{equation}
\label{eqn-esti-lr}
\big \|u^{\cX_1}(\cdot,\theta, \tau)-\sqrt{2(n-k)}\, \big  \|_{C^4(\{|y| \leq 100\})} \leq   \varepsilon^{\frac{1}{102}}\, (1+|x_0|^2)
\end{equation}
for  all $\theta \in \mathbb{S}^{n-k}$ and $\tau \leq  - \frac{99}{100} \, |\log \ep|$.  It follows that  the un-rescaled profile $U(x,t)$ satisfies the estimate \begin{equation}
\label{eqn-U-esti}
\big | U( x,\theta,t) - \sqrt{2(n-k)} \sqrt{1-t} \, \big | < \varepsilon^{\frac{1}{102}}\, (1+|x|^2) \, \sqrt{1 -t} 
\end{equation}
for all $| x| \le L - 20 |\log \ep |^{\frac 12}$, $\theta\in \mathbb{S}^{n-k}$ and $t \le 1 - \varepsilon^{\frac{99}{100}}$.
\end{lemma}

\begin{proof} Fix   $x_0$ satisfying $|x_0|   \leq  L-20 |\log\varepsilon|^{\frac{1}{2}}$. To simplify the notation,  we  drop the index $\cX_1$ and simply set  $v=u^{\cX_1}-\sqrt{2(n-k)}$ and $\bv = v \, \vp$,
as before. 
Also, recall our notation of $a(\tau), b(\tau)$ as in \eqref{eqn-abdef} and $\hat{b}(\tau) := b(\tau) + B_{\varepsilon}(\tau),  B_{\varepsilon}(\tau) := e^{-10|\log\varepsilon|e^{\tau}}.$
As in equation \eqref{eqn-a2}) in Lemma \ref{lem-ratio}, 
\[a(\tau) - C_0\varepsilon_0 \big(a(\tau) + \hat{b}(\tau)\big) \le a'(\tau) \leq a(\tau) + C_0\varepsilon_0 \big(a(\tau) + \hat{b}(\tau)\big),\]
for all $\tau\in [0, -\frac{99}{100}\log\varepsilon]$. By the conclusion of Lemma \ref{lem-ratio} we have
\[(1-C_1\varepsilon_0)\, a(\tau) \le a'(\tau) \le (1 + C_1\varepsilon_0)\, a(\tau)\]
for another positive numeric constant $C_1$. Integrating  in time from zero to any $\tau  \in (0, -\frac{99}{100} \log\varepsilon)$,  we get
\begin{equation}
\label{eqn-a0-est}
a_0 e^{(1-C_1\varepsilon_0)\tau} \le a(\tau) \le a_0 e^{(1+C_1\varepsilon_0)\tau},
\end{equation}
$\tau\in [0, -\frac{99}{100}\log\varepsilon]$, where by Proposition \ref{eqn-a0b0} we have $a_0 = A_0 \ve\,(2\, k + |x_0|^2)\, (1+ o(1))$. 

Since $\bar{v}(\cdot,\tau) = a(\tau) + \hat{v}(\cdot,\tau)$, by Lemma \ref{lem-ratio}, \eqref{eqn-a0-est}, $a_0 \sim \ep (1+ |x_0|^2) $ and   given  that $\varepsilon_0$ can be chosen  small in a uniform way, we have 
\begin{equation}
\label{eqn-start}
\|\bar{v}(\cdot,\tau)\|_{\cH} \le C_2 |a(\tau)| \le C_2 a_0 \, e^{(1+C_1\varepsilon_0)\tau} <  C_2 \, \varepsilon^{\frac{1}{101}}\, (1+|x_0|^2)
\end{equation}
for all $\tau \le -\frac{99}{100}\, \log\varepsilon$, where $C_2$ is another positive numeric constant. Recall that $\bv = v \, \vp$ where $\vp \equiv 1$ on $|y| \leq \frac 12 \, \ell(\tau)=5 |\log \varepsilon|^{\frac 12} \, e^{\frac \tau2}$,
that is $\ell(\tau) \gg 200$. Hence, \eqref{eqn-start} implies 
$ \| v \cdot \chi_{|y| \leq 200} \|_{\cH} < C_2 \, \varepsilon^{\frac{1}{101}}\, (1+|x_0|^2)$ and by standard $L^\infty$ and derivative estimates, we get $ \| v \cdot \chi_{|y| \leq 100} \|_{C^4} < C_3 \, \varepsilon^{\frac{1}{101}}\, (1+|x_0|^2) < \varepsilon^{\frac{1}{102}}\, (1+|x_0|^2) $. Since $v=u^{\cX_1}-\sqrt{2(n-k)}$, we obtain   \eqref{eqn-esti-lr}. Evaluating the last estimate at $y=0$ and unrevealing the coordinate change \eqref{eqn-cv} we obtain the bound
$\big  | U( x_0,\theta,t) - \sqrt{2(n-k)} \sqrt{1 - t}\big | < \varepsilon^{\frac{1}{102}}\, (1+|x_0|^2)\, \sqrt{1 -t}$, 
for all  $t \le 1 - \varepsilon^{\frac{99}{100}}$. Since $x_0$ is an arbitrary point satisfying $|x_0| < L - 20 |\log \ep|^{\frac 12}$,
we readily obtain the second assertion of our lemma. 
 
\end{proof} 

Having the uniform estimate \eqref{eqn-U-esti} which holds for all $|x| \leq L - 20 |\log \ep |^{\frac 12}$, enables us to extend the $C^4$  bound on $u^{\cX_1}(y, \theta, \tau)$ from $|y| \leq 100$ to $|y| \leq \frac L4$,
where $L  \in [\ep^{-\frac 1{500}}, \, \ep^{-\frac 1{300}}]$, thus finishing the  proof of Proposition \ref{the:local_radius}

\begin{proof}[Proof of Theorem  \ref{the:local_radius}]
Fix $L \in  [\ep^{-\frac 1{500}}, \, \ep^{-\frac 1{300}}]$, $|x_0| <  \frac L2$,  and $\ell(\tau)$ such that $0 < \ell(\tau)  < \frac L3$.   Consider  the rescaled profile 
$u^{\cX_1}$  as before. 
Our estimate  \eqref{eqn-U-esti} gives  us that $ \big |u^{\cX_1}(y, \theta, \tau) - \sqrt{2(n-k)}  \big  | <  C_2 \, \varepsilon^{\frac{1}{102}}\, (1+|x|^2)$
holds, for any $|y| < \ell (\tau)$ and  $x = x_0 + y \sqrt{1- t}$,  provided   $|x| \leq  L-20 |\log \ep|^{1/2}$ (which holds from our choices 
$|x_0| < \frac L2$ and  $\ell (\tau)< \frac L3$).  
Taking the supremum over all   $|y| < \ell(\tau)$ we obtain the $L^\infty$ bound \eqref{eqn-LinftyL}. 
 Once the $L^\infty$ bound is established, the $C^4$ bound in \eqref{eqn-C4L} holds
by standard local derivative estimates. 

\end{proof} 

Recall the definitions for $a(\tau)$ and $b(\tau)$ in \eqref{eqn-abdef}. With the improved   $C^4$ bound \eqref{eqn-C4L},  we can now  improve \eqref{eqn-a0-est} as follows. This improvement 
will play crucial role in the next section. However, since    our improved estimate \eqref{eqn-C4L} holds on a smaller set than the estimate  \eqref{eqn-ck} that we used previously,  we need to readjust the choice of our cut off function. We choose our {\em new  cut off function $\varphi$}    to be supported on the set $\{(y,\theta)\,\,\,|\,\,\, \theta\in \mathbb{S}^{n-k},\,\,\, y\in \mathbb{R}^k \,\,\, |y| \le 100|\log\varepsilon|^{\frac12} \,e^{\frac{\tau}{1000}}\}$, and such that $\vp\equiv 1$ on $|y| \le 50 |\log\varepsilon|^{\frac12} \,e^{\frac{\tau}{1000}}$. Note that \[\{(y,\theta)\,\,\,|\,\,\, \theta\in \mathbb{S}^{n-k}, \,\, |y| \le 100|\log\varepsilon|^{\frac12} \,e^{\frac{\tau}{1000}}\} \subset \{(y,\theta)\,\,\,|\,\,\, \theta\in \mathbb{S}^{n-k},\,\,y\in \mathbb{R}^k, \,\,\, |y| \le \tfrac{L}{4}\}\]
for any $L\in (\varepsilon^{-\frac{1}{500}}, \varepsilon^{-\frac{1}{300}})$.

Set $\ell_\ep(\tau) := 100|\log\varepsilon|^{\frac12} \,e^{\frac{\tau}{1000}} < \frac L4$ and apply  \eqref{eqn-C4L},  for every $x_0$ satisfying $|x_0| < \frac L2$ to obtain 
\be\label{eqn-C4L5}
\big \|   u^{\cX_1}(y, \theta, \tau)  - \sqrt{2(n-k)} \big \|_{C^4(\{|y| < \ell_\ep (\tau)\})}< C_0 \, \varepsilon^{\frac{1}{102}} \, \big (1+ |x_0|^2 + \ell_\ep(\tau)^2 e^{-\tau}
\big ) \le \bar{C}_0\, \ve^{\frac{1}{102}}\,(1 + |x_0|^2)\ee
since $\sup_{ |x-x_0| < \ell_\ep (\tau)e^{-\frac \tau2} } \, (1+|x|^2)  \leq 1+ |x_0|^2 + \ell_\ep(\tau)^2 e^{-\tau}$.

\smallskip

 If we define projections of our solution with respect to our  new cut off function as defined above, and still keep the same notation for all projections (in particular the definitions
 of $a(\tau), b(\tau)$), then it is easy to see that the conclusion of Lemma \ref{lem-ratio} still holds,  where we just redefine $B_{\varepsilon}(\tau) := 10 |\log\varepsilon| e^{\frac{\tau}{1000}}$. As a conclusion,   we can prove the following improved asymptotics for the dominant mode.

\begin{proposition}\label{prop-atau}
Under the assumptions of Lemma \ref{prop-local-radius}, there is an $\eta \in (0,1)$ so that 
\be \label{eqn-good-atau} a(\tau) = a_0 \,  e^{\tau}\, \big(1+ O(e^{-\eta \tau})\big),\ee
for all $\tau\in [0 -\frac{99}{100}\log\varepsilon]$, where $a_0 = A_0  \varepsilon \, (2k+|x_0|^2)\, (1 +  o(1))$, for a numeric constant $A_0$.
\end{proposition}

\begin{proof}
To prove this asymptotic behavior for $a(\tau)$ we will use analogous computations to the ones in the proof of Lemma \ref{claim-ab}, using the improved estimate \eqref{eqn-C4L}  instead of \eqref{eqn-ck}. 
Under the notation of the  proof of Lemma \ref{claim-ab},  we first observe that
\[\|\mathcal{E}_{\varphi}\|_{\mathcal{H}} \le C_1 \ep^{\frac 1{120}}  \, e^{-10|\log\varepsilon|e^{\frac{\tau}{250}}}.\]
Denote by $\tilde{\varepsilon}:= \bar{C}_0\, \varepsilon^{\frac{1}{102}}\, (1 + |x_0|^2)$. Similar computations to the ones we performed in the proofs of Lemma \ref{claim-ab} and Lemma \ref{lem-ratio} yield,
\[a(\tau) - \tilde{\varepsilon}(a(\tau) + \hat{b}(\tau)) \le a'(\tau) \le a(\tau) + \tilde{\varepsilon}(a(\tau) + \hat{b}(\tau)),\]
\[\hat{b}'(\tau) \le \frac12 \hat{b}(\tau) + \tilde{\varepsilon} (a(\tau) + \hat{b}(\tau)),\]
where $\hat{b}(\tau) := a(\tau) + B_{\varepsilon}(\tau)$ with $B_\ep(\tau):= 10 |\log\varepsilon| e^{\frac{\tau}{1000}}$, and
$a(\tau) \ge c_0 \,\hat{b}(\tau)$ for a numeric constant $c_0 >0$ (see Lemma \ref{lem-ratio}), 
implying the bounds
\[(1-2\tilde{\varepsilon})\, a(\tau) \le a'(\tau) \le (1+2\tilde{\varepsilon})\, a(\tau).\]
Hence, 
\[a_0\, e^{(1-2\tilde{\varepsilon})\tau} \le a(\tau) \le a_0\, a^{(1+2\tilde{\varepsilon})\, \tau},\]
for all $\tau\in [0,-\frac{99}{100}\log\varepsilon]$. From here we get
\begin{equation}
\label{eq-this-a-used}
a(\tau) - a_0\, e^{\tau}| \le C_0 a_0  \tilde{\epsilon} \, \tau  e^{\tau}\le C_0 a_0\, e^{\tau(1-\eta)},
\end{equation}
for some $\eta \in (0,1)$, finishing the proof. 

\end{proof}

\section{Developing a neckpinch singularity} 
\label{sec:neckpinch-sing}

 The goal of this section is to prove Theorem \ref{thm-main0}.  
In section \ref{sec-X0} we have obtained local estimates for the cylindrical radius of the rescaled solution $u^{\cX_1}(y,\theta,\tau)$, where the rescaling was performed around $\cX_1 = (x_0,0,1)$. The estimates hold up to rescaled time $\tau \le -\frac{99}{10}\log\varepsilon$. We can not expect those estimates to hold all the way up to the singular time, since 
 $T = 1$ may not have been the first singular  time and we may have chosen the wrong point $p_0=(x_0,0)$ around which we performed the rescaling. 
 
  In this section, for every $x_0$ large, we will be adjusting the rescaling time $T = 1+t(x_0)$ by an appropriate small number  $t(x_0)$  and will consider the rescaled solution $u^{\wcX}(\tilde{y},\theta, \tilde{\tau})$ around $\wcX = (x_0, 0, 1+t(x_0))$  in order to improve the estimates obtained in section \ref{sec-X0}. We use Proposition \ref{thm:weak_rad} to show  that the improved estimates for the graphical radius hold outside a sufficiently large compact set, extending up to and past the first singular time. Combining this with maximum principle arguments yields scale-invariant curvature estimates. We ultimately use these estimates to conclude the proof of Theorem \ref{thm-main0}, establishing that the flow starting from our initial data develops a neckpinch singularity where the unique tangent flow is a generalized cylinder $\mathbb{S}^{n-k}\times\mathbb{R}^k$.
\smallskip

Let $\varepsilon_0$ be a small number as in Proposition \ref{thm:weak_rad}. We will denote below a uniform multiple of $\varepsilon_0$ still by $\varepsilon_0$ for simplicity. Then we have the following result.
\begin{proposition}\label{the:first_singular_time}
Suppose that $U(x,\theta,t)$ is the profile of a mean curvature flow $\{ M_t \}_{t \in [0,T)}$  that satisfies the assumptions of Theorem \ref{thm-main0} 
and in particular  $U(x,\theta,0)=\sqrt{2(n-k)}+\varepsilon \, |x|^2 \big  (1 + o(1) \big )$ holds for $ |x| \leq L_{\ve}$, 
where   $L_{\ve} \ge L$ and  $\varepsilon^{-\frac{1}{500}}\leq L\leq \varepsilon^{-\frac{1}{300}}$. Then, the profile $U(x,\theta ,t)$ is well defined for $|x| \in (\frac{L}{4} ,\frac{L}{2})$ and $t\leq 1+ \varepsilon^{\frac{249}{250}}$, and  satisfies the bounds 
\begin{equation}
\label{eq-U-quadratic}
 (2(n-k) - A_0\,\varepsilon_0)(1-t) + A_1  \ep  |x|^2 \le U(x,\theta,t)^2 \le  (2(n-k)+A_0\,\varepsilon_0) (1-t) + 3 A_2  \varepsilon |x|^2,
\end{equation}
for $|x| \in (\frac{3L}{16} ,\frac{L}{5})$ and $t\leq 1+ \frac12 \,\varepsilon^{\frac{249}{250}}$, where $A_0, A_1,  A_2$ are fixed numeric constants.  
In addition, we have the better bound 
\begin{equation}
\label{eq-U-quadratic2}
\big (  2(n-k) - A_1 \ep^{\frac 1{102}} \,|x|^2\big ) \, (1-t)  \le U(x,\theta,t)^2 \le  \big ( 2(n-k)  + A_2 \ep^{\frac 1{102}}  |x|^2 \big ) \,  (1-t) 
\end{equation}
holding for $|x| \in (\frac{L}{4} ,\frac{L}{2})$ and $t\leq 1-  \varepsilon^{\frac{99}{100}}.$

 Moreover, any $p\in M_t$ with $t\leq 1+ \varepsilon^{\frac{249}{250}}$, and $|\langle p,e_1\rangle| \in (\frac{L}{4} ,\frac{L}{2})$ lies on an $\varepsilon_0$-neck.

\end{proposition}

\begin{proof}
Fix $L  \in  [ \ep^{-\frac{1}{500}}, \ep^{-\frac 1{300}}]$ and for  any $x_0$ satisfying $|x_0| \leq L$, let  
\be\label{eqn-tx0}
t(x_0) := \frac{\sqrt{2}}{\sqrt{n-k}} \, a_0\ee
where $a_0$ is as in Proposition \ref{eqn-a0b0} and satisfies $a_0 = A_0 \, \ep \, (2k+x_0^2) + o(\ep)$. Our choice of $t(x_0)$ will be justified in Step  \ref{step 1}  of the proof below.

Throughout most  this section we denote by 
 $\wcX  = (x_0,0, 1+ t(x_0))$  and denote by   $u^\wcX(\tilde{y},\theta,\tilde{\tau})$   the profile of the rescaled flow $\widebar  M_\tau^{\wcX} = (1+t(x_0))^{-\frac 12} \, \big ( M_t- p_0'\big )$, 
$p_0'=(x_0,0)$,   
$\tau=-\log (1+t(t_0)-t)$, which is given  in terms of $U(x,\theta,t)$ by  
 \[u^\wcX(\tilde{y},\theta,\tilde{\tau}) = \frac{U(x,\theta,t)}{\sqrt{1+t(x_0) -t}}, \qquad \tilde{y} = \frac{x-x_0}{\sqrt{1+t(x_0)-t}},  \qquad \tilde{\tau} = -\log(1+t(x_0)-t).\]
Also recall that for $\cX_1 = (x_0, 0, 1)$, 
the  profile $ u^{\cX_1} $ of the rescaled flow $\widebar M_\tau^{\cX_1}$ centered at $\cX_1$,  is given  in terms of $U(x,\theta,t)$   by
\[u^{\cX_1}(y,\theta,\tau) = \frac{U(x,\theta,t)}{\sqrt{1 -t}}, \qquad y = \frac{x-x_0}{\sqrt{1-t}} \qquad \tau = -\log(1-t).\]
To simplify the notation, in the course of the proof of this Theorem we will shortly use the notation 
\[u(y,\theta,\tau):= u^{\cX_1}(y,\theta,\tau), \qquad   \bar u(\tilde{y},\theta,\tilde{\tau}):= u^\wcX (\tilde{y},\theta,\tilde{\tau}), \qquad 
\delta := t(x_0).\]

If $v = u-\sqrt{2(n-k)}$ and $\tilde{v} = \bar u - \sqrt{2(n-k)}$, a straightforward computation implies
\begin{equation}
\label{eq-scaling-v}
\tilde{v}(\tilde{y},\theta,\tilde{\tau}) = \sqrt{1-\delta e^{\tilde{\tau}}} \, \, v(y,\theta,\tau) + \sqrt{2(n-k)}\, (\sqrt{1-\delta\, e^{\tilde{\tau}}} - 1),
\end{equation}
and 
\begin{equation}
\label{eq-scaling-ytau}
\tilde{y} = y\, \sqrt{1-\delta \, e^{\tilde{\tau}}}, \qquad e^{\tilde{\tau}} = \frac{e^{\tau}}{1+\delta e^{\tau}}.
\end{equation}

 Next, choose $\bar \vp \in C^{\infty}(\R)$ with $\bar \vp(z)=1$ for $|z| \le \frac12$ and $\bar \vp(z)=0$ for $|z| \ge 1$, and let 
\begin{equation}
\label{eqn-cutoff15}
  \vp(y,\theta,\tau) := \bar \vp \big ( \frac{|y|}{\ell_\ep(\tau)}\big ) \qquad \mbox{where} \quad  \ell_\ep(\tau) := 100|\log\varepsilon|^{\frac12} \,e^{\frac{\tau}{1000}}
\ee 
and $V  (y,\theta,\tau) := v(y,\theta,\tau) \vp(y,\theta,\tau)$.  Recall, from the previous section, our  notation of the projections of $V(y,\theta,\tau)$  on the  constant  mode and the remaining modes, which we called $a(\tau)$ and $b(\tau)$, respectively, and which are defined as in \eqref{eqn-abdef} (here we denote $\bv$ by $V$). By Proposition \ref{prop-atau},  we know that the constant mode is dominant and its  precise asymptotic behavior for all $\tau \le -\frac{99}{100}\log\varepsilon$. Note that $V  (y,\theta,\tau)$ is supported on $|y| \leq \ell_\ep(\tau) $ and
with our choice of $\ell(\tau)$ we have $\ell(\tau) \ll \frac L4$,
for all $\tau \leq - \frac {99}{100} \, \log \ep$. 

Similarly, we set $\tilde{V}(\tilde{y},\theta,\tilde{\tau}) := \tilde{v}(\tilde{y},\theta,\tilde{\tau})\, \tilde{\varphi}(\tilde{y},\theta,\tilde{\tau})$, with 
 $\tilde{\varphi}(\tilde{y},\theta,\tilde{\tau}) := \varphi\Big(\frac{\tilde{y}}{\sqrt{1-\delta e^{\tilde{\tau}}}}, \theta, \tilde{\tau}-\log(1-\delta e^{\tilde{\tau}})\Big)$, where $\varphi$ is the cut off function we defined in  \eqref{eqn-cutoff15}, and keep in mind that we can shrink a little bit its support if necessary.
Denote by $\tilde{a}(\tilde{\tau})$ and $\tilde{b}(\tilde{\tau})$ the  corresponding projections of $\tilde{V}(\tilde{y},\theta,\tilde{\tau})$ on a constant mode and the remaining modes, respectively, defined analogously as in \eqref{eqn-abdef}.  

\smallskip

We will prove our theorem in several steps. The first step is to improve the estimate \eqref{eqn-start} that tell us  
\begin{equation}
\label{eq-old-bound}
\|V(\cdot,\tau)\|_{\mathcal{H}} \le C \varepsilon^{\frac{1}{101}}(1+|x_0|^2),
\end{equation}
for all $\tau\le -\frac{99}{100}\log\varepsilon$.

\begin{step}\label{step 1}
For every $x_0$ satisfying $|x_0| <  \frac{L}{2}$ and  for $t(x_0) = \frac{\sqrt{2}}{\sqrt{n-k}} a_0$, we have 
\[\|\tilde{V}(\cdot,\tilde{\tau})\|_{\mathcal{H}}\le \varepsilon^{\frac{1}{51}}\, (1+|x_0|^2)^2,\]
for all $\tilde \tau \le -\log(t(x_0) + \varepsilon^{\frac{99}{100}})$.
\end{step}

\begin{proof}[Proof]
Recall our notation $\delta=t(x_0)$. Integrating \eqref{eq-scaling-v} over the  cylinder, with respect to the weight $e^{-\frac{|\tilde{y}|^2}{4}}$, and using \eqref{eq-scaling-ytau} we have
\be
\begin{split}
\int_{\mathbb{S}^{n-k}\times\mathbb{R^k}}& \tilde{V}(\tilde{y},\theta,\tilde{\tau}) e^{-\frac{\tilde{|y|}^2}{4}}\, d\tilde{y}d\theta =\\
&= (1-\delta e^{\tilde{\tau}})\, \int_{\mathbb{S}^{n-k}\times\mathbb{R}^k} V(y,\theta,\tau) e^{-\frac{|y|^2(1-\delta\, e^{\tilde{\tau}})}{4}}dy d\theta + \sqrt{2(n-k)}\,\langle 1,1\rangle\,\big(\sqrt{1-\delta e^{\tilde{\tau}}}-1\big).
\end{split}
\ee

Assume that $\delta e^{\tilde{\tau}} = o(1)$ and $|y|^2 \delta e^{\tilde{\tau}} = o(1)$,  for all $\tilde{\tau} \le -\log(\delta + \varepsilon^{\frac{99}{100}})$, and all $(y,\theta)$ in the support of our cutoff function $\varphi$, when  $\varepsilon$ is small enough. This will be verified and turn out to be true at   the end when we choose our $ t(x_0)$. The previous identity yields
\begin{equation}
\label{eq-tildea}
\tilde{a}(\tilde{\tau}) = a(\tau)\, \sqrt{1-\delta\, e^{\tilde{\tau}}} + \frac{1-\delta e^{\tilde{\tau}}}{\langle 1,1\rangle} \int_{\mathbb{S}^{n-k}\times\mathbb{R}^k} V(y,\theta,\tau) e^{-\frac{|y|^2}{4}}\, \Big(e^{\frac{|y|^2\delta e^{\tilde{\tau}}}{4}} - 1\Big) dy d\theta + \sqrt{2(n-k)}\, \big(\sqrt{1-\delta\, e^{\tilde{\tau}}} - 1\big).
\end{equation}
Using that $\sqrt{1-\delta e^{\tilde{\tau}}} = 1 - \frac{\delta e^{\tilde{\tau}}}{2} + O((\delta e^{\tilde{\tau}})^2)$, and that
$e^{\frac{|y|^2\delta e^{\tilde{\tau}}}{4}} - 1 = \frac{|y|^2\delta e^{\tilde{\tau}}}{4} + O((|y|^2\delta e^{\tilde{\tau}})^2)$ we have
\begin{equation}
\label{eq-big-est}
\begin{split}
\Big|\frac{1}{\langle1, 1\rangle} &\int_{\mathbb{S}^{n-k}\times\mathbb{R}^k} V(y,\theta,\tau) e^{-\frac{|y|^2}{4}}\, \Big(e^{\frac{|y|^2\delta e^{\tilde{\tau}}}{4}} - 1\Big) dy\, d\theta\Big| = \frac{\delta e^{\tilde{\tau}}}{\langle1, 1\rangle} \int_{\mathbb{S}^{n-k}\times\mathbb{R}^k} V(y,\theta,\tau) \,  \, (|y|^2-2k)\, e^{-\frac{|y|^2}{4}}  dyd\theta  \\
&+  2 k \, \delta e^{\tilde{\tau}} \int_{\mathbb{S}^{n-k}\times\mathbb{R}^k} V(y,\theta,\tau) e^{-\frac{|y|^2}{4}}\, dyd\theta  + O(\delta^2 e^{2\tilde{\tau}})\,\int_{\mathbb{S}^{n-k}\times\mathbb{R}^k} V(y,\theta,\tau) y^2 e^{-\frac{|y|^2}{4}}\, dyd\theta) \\
& \le C_0 \,\delta e^{\tilde{\tau}} b(\tau) + 2\, k\,\delta e^{\tilde{\tau}} a(\tau) + O(\delta^2 e^{2\tilde{\tau}}) \,\varepsilon^{\frac{1}{101}}\, (1 + |x_0|^2) \\
 &\le C_0 \, \varepsilon^{\frac{1}{101}} (1 + |x_0|^2) \, \delta e^{\tilde{\tau}}.
 \end{split}
    \end{equation}
where at the end we have used Cauchy-Scwartz inequality, \eqref{eq-old-bound}, the fact that $|a(\tau)| \le \|V(\cdot,\tau)\|_{\mathcal H}$ and $b(\tau) \le \|V(\cdot,\tau)\|_{\mathcal H}$.
Using \eqref{eq-tildea}, \eqref{eq-big-est} and \eqref{eq-this-a-used} we get
\begin{equation}
\label{eq-tildea-prel}
|\tilde{a}(\tilde{\tau})| \le a_0 e^{\tau} + C_0 a_0 \,  \varepsilon^{\frac{1}{102}}(1+|x_0|^2)\, \tau e^{\tau}  - \tfrac{\sqrt{2(n-k)}}{2} \delta e^{\tilde{\tau}} + C_0 \varepsilon^{\frac{1}{101}}(1+|x_0|^2) \delta e^{\tilde{\tau}}.
\end{equation}

From the definition of $\tilde \tau$ we have $e^{\tilde{\tau}} = \frac{e^{\tau}}{1+\delta e^{\tau}}$, implying that $e^{\tau} = e^{\tilde \tau} ( 1 + \delta e^{\tau})$, where $\delta e^{\tau} \le C_0\ve^{\frac{1}{100}}\, (1 + |x_0|^2)$, for al $\tau\in [0,-\frac{99}{100}\, \log\ve]$. This together with our choice of $t(x_0) = \frac{\sqrt{2}}{\sqrt{n-k}}\, a_0$, where  by Proposition \ref{eqn-a0b0} we have that $a_0 = A_0\,  \ep \,  ( 2k+  |x_0|^2)\, (1 + o(1))$, for some numerical constant $A_0$, and \eqref{eq-tildea-prel} yield
\begin{equation}
\begin{split}
\label{eq-tildea-bound}
|\tilde{a}(\tilde{\tau})| &\le C_0 a_0 e^{\tilde{\tau}}\, \ve^{\frac{1}{100}}\, (1 + |x_0|^2) + C_0 a_0 \ve^{\frac{1}{102}}\, (1 + |x_0|^2) \tau e^{\tau} + C_0 \ve^{\frac{1}{101}}\, ( 1+ |x_0|^2) \delta e^{\tilde{\tau}}\\
&\le C_0\, e^{\tilde{\tau}}\, \Big(\ve^{1+\frac{1}{100}}(1 + |x_0|^2)^2 + \ve^{1+\frac{1}{102}}\, |\log\ve|\, (1+|x_0|^2)^2 + \ve^{1+\frac{1}{101}}\, (1 + |x_0|^2)^2\Big) \\
&\le C_0 \, \ve^{1+\frac{1}{103}}\, (1 + |x_0|^2)^2\, e^{\tilde{\tau}}.
\end{split}
\end{equation}

Having the bound \eqref{eq-tildea-bound}, it is now easy to estimate $\tilde b(\tilde \tau)$.  First note, that due to \eqref{eq-scaling-v} and \eqref{eq-scaling-ytau}, just by increasing the constants by a little bit and adjusting $\ell_\ep(\tau)$ a bit (we call it $\tilde \ell_\ep(\tau)$),  estimate \eqref{eqn-C4L5} yields
$$\big \|   \tilde u^{\cX}(\tilde y, \theta, \tilde \tau)  - \sqrt{2(n-k)} \big \|_{C^4(\{|\tilde y| < \tilde \ell_\ep (\tau)\})}< C_0 \, \varepsilon^{\frac{1}{102}} \, \big (1+ |x_0|^2 + \tilde \ell_\ep(\tau)^2 e^{-\tau}
\big ) < \varepsilon^{\frac{1}{120}} 
$$
where $\tilde \ell_\ep(\tau)$ denotes the size of the support of the cut off function $\tilde \vp$ which satisfies $\tilde \ell_\ep(\tau) \approx \ell_\ep(\tau)$,
where we recall that $\ell_\ep(\tau) := 100|\log\varepsilon|^{\frac12} \,e^{\frac{\tau}{1000}} \ll \frac L4  $. 
Denote by $\tilde{\varepsilon} := \varepsilon^{\frac{1}{122}}$.
Similar  computation to the one in the proof of Lemma \ref{claim-ab}, using \eqref{eq-tildea-bound} and the $C^4$ estimate above implies 
\begin{equation}
\label{eq-tildeb-comp}
\tilde{b}'(\tilde{\tau}) \le \big(\tfrac12 + \tilde{\ve}\big) \tilde{b} + \tilde{\ve}\, |\tilde{a}| + \tilde{\ve}\,  e^{-100|\log \ve|\, e^{\frac{\tilde{\tau}}{1000}}},
\end{equation}
for all $\tilde{\tau} \in [0,-\log(t(x_0) + \ve^{\frac{99}{100}}]$. Note that for small $\ve$ and all $\tilde{\tau}$ in the considered  range, given \eqref{eq-tildea-bound}, we can absorb the last term on the right hand side in \eqref{eq-tildeb-comp} in the second term on the right hand side of the same differential inequality, and hence getting
\[\frac{d}{d\tilde{\tau}}\Big(\tilde{b}(\tilde{\tau})\, e^{-\big(\frac12 + \tilde{\ve}\big)\tilde{\tau}}\Big) \le  C_0 \ve^{1 + \frac{1}{122} + \frac{1}{103}}\, (1 + |x_0|^2)^2\,  e^{\big(\frac12 - \tilde{\ve}\big)\tilde{\tau}}. \]
If we integrate the previous differential inequality in time we obtain
\[\tilde{b}(\tilde{\tau}) \le e^{\big(\frac12 + \tilde{\ve}\big)\tilde{\tau}}\tilde{b}(0) + C_0 e^{\big(\frac12 + \tilde{\ve}\big)\tilde{\tau}} \ve^{1 + \frac{1}{122} + \frac{1}{103}}\, (1+ |x_0|^2)^2\, \int_0^{\tilde{\tau}} e^{\big(\frac12 - \tilde{\ve}\big)s}\, ds.\]
This furthermore implies
\[\tilde{b}(\tilde{\tau}) \le e^{\big(\frac12 + \tilde{\ve}\big)\tilde{\tau}}\tilde{b}(0) + C_0 \ve^{1 + \frac{1}{122} + \frac{1}{103}}\, (1 + |x_0|^2)^2\, e^{\tilde{\tau}}.\]
Since $\tilde{b}(0) \le C_0\, \ve |x_0|$, we easily see that for all $\tilde{\tau} \in [0, -\log(t(x_0) + \ve^{\frac{99}{100}}]$, since $t(x_0) \ll  \ve^{\frac{99}{100}}$ we have
\begin{equation}
\label{eq-tildeb-bound}
\tilde{b}(\tilde{\tau}) \le \ve^{\frac{1}{51}}\, (1 + |x_0|^2)^2.
\end{equation}
Note that  \eqref{eq-tildea-bound}, $\tilde{\tau} \in [0, -\log(t(x_0) + \ve^{\frac{99}{100}}]$ and  $t(x_0) \ll  \ve^{\frac{99}{100}}$ also  
imply  $\tilde a(\tilde \tau) \leq \varepsilon^{\frac{1}{51}} (1+|x_0|^2)^2$. Having \eqref{eq-tildea-bound} and \eqref{eq-tildeb-bound} we finally conclude that 
$\|\tilde{V}(\cdot,\tilde{\tau})\|_{\mathcal{H}} = \tilde a(\tau) + \tilde b(\tau)$ satisfies
\[\|\tilde{V}(\cdot,\tilde{\tau})\|_{\mathcal{H}} \le  \varepsilon^{\frac{1}{51}} (1+|x_0|^2)^2,\]
for all $\tilde{\tau} \in [0, -\log\big(t(x_0) + \ve^{\frac{99}{100}}\big)]$ as claimed.
\end{proof}

In the last two steps we will use  Proposition \ref{thm:weak_rad}, to show that if  $\frac L8 \le|x_0| \le  \frac L4$, then the flow around $x_0$ can be extended beyond $T=1$,  at least up to time $T= 1+ (1- \eta)\,  t(x_0)$
for some $\eta \in (0,1)$.  Later we will argue that this time interval in which we have this estimate contains the singular time of the flow. 

\begin{step}
\label{step-extension}
Let   $|x_0| \le \frac{L}{4}$  and $t(x_0) = \frac{\sqrt{2}}{\sqrt{n-k}} \, a_0 = A_0 \ep \, (2k+|x_0|^2)(1 + o(1))$,  as chosen in the previous step, where $A_0$ is a numeric constant. Denote by
 \begin{equation}
 \label{eq-Geps}
 \mathcal{G}_{\varepsilon} :=  C_0\,  \Big( \varepsilon^{\frac{1}{51}}(1+|x_0|^2 + |\log \ep|)^2 + \varepsilon^{\frac{1}{100}}\sqrt{|\log\varepsilon|}\, \big( |x_0| + \sqrt{|\log\varepsilon|}\big)\Big)
 \end{equation}
Then, the following $L^\infty$ bound holds 
\begin{equation}
\label{eq-extend}
\Big |  U(x,\theta,t) - \sqrt{2(n-k)(1+t(x_0)-t)} \, \Big |< \varepsilon_0\sqrt{1+t(x_0)-t},
\end{equation}
(we  also have the corresponding, correctly scaled higher order derivative estimates up to the 4th order),  for all $|x-x_0| \le 200\,\sqrt{1+t(x_0)-t}$, all $\theta \in {\mathbb S}^{n-k}$ and all
 $$t \le 1  + t(x_0) - 2\mathcal{G}_{\varepsilon}^{\frac{99}{100}}t(x_0) - 2\mathcal{G}_{\varepsilon}^{\frac{99}{100}}\varepsilon^{\frac{99}{100}}.$$
 \end{step}

\begin{proof}

Using the estimate we obtained in Step \ref{step 1},  standard  interior parabolic estimates imply
 \[\|u^\wcX(\cdot,\tau)-\sqrt{2(n-k)}\,  \|_{C^4(\{|y| \leq 200\})} \leq C\,  \varepsilon^{\frac{1}{51}}(1+x_0^2)^2,\]
 for all $\tau \le -\log(t(x_0) + \varepsilon^{\frac{99}{100}})$, where  $\wcX = (x_0, 0, 1+t(x_0))$. In particular, this implies
 \begin{equation}
 \label{eqn-consequence1}
\Big |U(x,\theta,t) - \sqrt{2(n-k)(1+t(x_0)-t)}\Big |_{L^{\infty}} < C\,  \varepsilon^{\frac{1}{51}}(1+|x_0|^2)^2\sqrt{1 + t(x_0)-t},
 \end{equation}
for all  $ |x-x_0| \le 200 \,  \sqrt{1 + t(x_0)-t}$ and all $t \le 1 - \varepsilon^{\frac{99}{100}}$ and all $|x_0| <  \frac{L}{4}$. We  also have the corresponding, correctly scaled higher order derivative estimates up to the 4th order.
Evaluating at $x_0$ and varying  $x_0$ we get the following $L^\infty$ estimate (and the corresponding appropriately scaled derivative estimates up to the 4th order)
\begin{equation}
 \label{eqn-consequence2}
\Big |U(x,\theta,t) - \sqrt{2(n-k)(1+t(x)-t)}\Big | < C\,  \varepsilon^{\frac{1}{51}}(1+|x|^2)^2\sqrt{1 + t(x)-t},
 \end{equation}
holding for all $|x|< \frac L4$, where $t(x) = A_0\, \ep \, (1+|x|^2)\, (1 + o(1))$, and all $t \le 1 - \varepsilon^{\frac{99}{100}}$, where $A_0$ is a numeric constant, independent of $x$. 

Now fix again $x_0$ with $|x_0| < \frac L4$, and let $|x-x_0| \le 200 \sqrt{|\log\varepsilon|}$.  
Then, \eqref{eqn-consequence2} implies
\begin{equation}
\label{eq-change}
\begin{split}
& |U(x,\theta,t) - \sqrt{2n-k)(1+t(x_0)-t)}| \le \\
&C_0 \varepsilon^{\frac{1}{51}}(1+x_0^2 + |\log\varepsilon|)^2\sqrt{\frac{1+t(x)-t}{1+t(x_0)-t}}\, \sqrt{1+t(x_0) - t} + \sqrt{2(n-k)}\, \Big|\sqrt{1+t(x_0)-t} - \sqrt{1+t(x)-t}\Big|
\end{split}\end{equation}
for all $|x-x_0| \le 200\sqrt{|\log\varepsilon}|$, and all $t \le 1 - \varepsilon^{\frac{99}{100}}$. Note that 
\begin{equation}
\label{eq-repar1}
\sqrt{\frac{1+t(x)-t}{1+t(x_0)-t}} = \sqrt{1 +  e^{\tau}\,(t(x) - t(x_0))} \sim \sqrt{1+ A_0\, \varepsilon \, e^{\tau}\,\big(|x|^2 - |x_0|^2\big)} \le 2,
\end{equation}
for all $t \le 1 - \varepsilon^{\frac{99}{100}}$ and all $|x - x_0| \le 200 \, \sqrt{|\log\varepsilon|}$. On the other hand for the same range for $x$ and $t$ we also have
\begin{equation}
\label{eq-repar2}
\begin{split}
|\sqrt{1+t(x_0)-t} - \sqrt{1+t(x)-t}| &=  \Big|1 - \sqrt{1+e^{\tau}\, (t(x) - t(x_0)}\Big| \, \sqrt{1+t(x_0)-t}\\
&\le C_0\, \varepsilon e^{\tau} \, |x-x_0|\, \big(|x| + |x_0|\big)  \sqrt{1+t(x_0)-t}\\\
&\le C_0\,  \varepsilon^{\frac{1}{100}}  \sqrt{|\log\varepsilon|}\,\big  (|x_0| + \sqrt{|\log\varepsilon|} \big )\, \sqrt{1 + t(x_0) - t}
\end{split}
\end{equation}
Combining \eqref{eq-change}, \eqref{eq-repar1} and \eqref{eq-repar2} yields,
 \begin{equation}
 \begin{split}
 \label{eq-consequence1}
&\Big |U(x,\theta,t) - \sqrt{2(n-k)(1+t(x_0)-t)}\Big | < \\
& < C_0\,  \Big( \varepsilon^{\frac{1}{51}}(1+|x_0|^2 +  |\log \ep|)^2 + \varepsilon^{\frac{1}{100}}\sqrt{|\log\varepsilon|} \big( |x_0| + \sqrt{|\log\varepsilon|}\big)\Big) \sqrt{1+t(x_0)-t},
\end{split}
 \end{equation}
 and the corresponding derivative estimates up to the fourth order for unrescaled flow, for all $|x-x_0| \le 200 \sqrt{|\log\varepsilon|}$ and all $t \le 1 - \varepsilon^{\frac{99}{100}}$. 
 
 On the other hand, Proposition \ref{thm:weak_rad}, implies that for any solution with bounded entropy of the rescaled flow $\widebar M^{\cX_0}_\tau$,
 around any point $\cX_0=(x_0, x_0',1)$, in particular for $x_0' =0$,   the following holds: for a given $\varepsilon_0$ there exists an $\bar{\varepsilon}$, so that if $\tau\in [0, \log 2]$ the profile
 $u^{\cX_0}$ of $\widebar M^{\cX_0}_\tau$ we have
 \[\|u^{\mathcal{X}_0}(y,\theta,\tau) - \sqrt{2(n-k)}\|_{C^4(\{|y| \le 200 \sqrt{|\log\bar{\varepsilon}}|\})} < \bar{\varepsilon}, \]
which in  terms of   the unrescaled flow   is equivalent to the $C^0$ estimate
 \begin{equation}
 \label{eq-previous-unrescaled}
\Big |U(x,\theta,t) - \sqrt{2(n-k)(1-t)}\Big |< \bar{\varepsilon}\, \sqrt{1-t},
 \end{equation}
 and to  the corresponding, correctly scaled higher order derivative estimates up to the 4th order, for $|x-x_0| \le 200 \sqrt{|\log\bar{\varepsilon}}|\, \sqrt{1-t}$, for $t\in [0,\frac12]$, then
 \begin{equation}
 \label{eq-conclusion}
\Big |U(x,\theta,t) - \sqrt{2(n-k)(1-t)}\Big |< \varepsilon_0 \sqrt{1-t},
 \end{equation}
 (and we also have the corresponding, correctly scaled higher order derivative estimates up to the 4th order) for all $|x-x_0| \le 100 \sqrt{|\log\bar{\varepsilon}}|$, and all $t\le 1 - \bar{\varepsilon}^{\frac{99}{100}}$. 
 
 Define $\bar{U}(x,\theta,\bar{t}) := U \big (x,\theta,\bar{t}+(1-\varepsilon^{\frac{99}{100}}-\beta) \big )$, for a $\beta$ that will be appropriately chosen below. Then \eqref{eq-consequence1} can be rewritten as

 \begin{equation}
 \label{eq-intermediate}
\Big |\bar{U}(x,\theta,\bar{t}) - \sqrt{2(n-k)(\beta+t(x_0)+\varepsilon^{\frac{99}{100}}-\bar{t})} \, \Big |< \mathcal{G}_{\varepsilon}\, \sqrt{\beta + t(x_0) + \ve^{\frac{99}{100}} - \bar{t}},
 \end{equation}
 for all $|x-x_0| \le 200 \sqrt{|\log\varepsilon|}$,  and all $\bar{t} \in [0,\beta]$, where $\mc G_\ve$ is defined by \eqref{eq-Geps}. Furthermore, if we define the parabolically  rescaled solution
 $U_{\lambda}(\tilde{x},\theta,\tilde{t}) := \lambda \bar{U}\left(\frac{\tilde{x}}{\lambda},\theta, \frac{\tilde{t}}{\lambda^2}\right)$, then by \eqref{eq-intermediate} we have
 \begin{equation}
 \label{eq-rescaled-lambda}
\Big |U_{\lambda}(\tilde{x},\theta,\tilde{t}) - \sqrt{2(n-k)(t(x_0)+\varepsilon^{\frac{99}{100}}+\beta)\lambda^2-\tilde{t}}\, \Big |< \mathcal{G}_{\varepsilon}\,\sqrt{(t(x_0)+\varepsilon^{\frac{99}{100}}+\beta)\lambda^2-\tilde{t}},
 \end{equation}
 (and we also have the corresponding, correctly scaled higher order derivative estimates up to the 4th order) for all $|\tilde{x}-\lambda x_0| < 200 \sqrt{|\log\varepsilon|}\, \lambda$, and all $\tilde{t} \le \beta \lambda^2$.
 
 Now choose $\lambda$ and $\beta$ so that
 \begin{equation}
 \label{eq-choice}
 (t(x_0)+\varepsilon^{\frac{99}{100}})\lambda^2 = \frac12, \qquad \beta \lambda^2 = \frac12
 \end{equation}
 Note that with this choice we have
 \[\Big |U_{\lambda}(\tilde{x},\theta,\tilde{t}) - \sqrt{2(n-k)(1-\tilde{t})}\Big |<  \mathcal{G}_{\varepsilon}\, \sqrt{1-\tilde{t}},\]
 for $|\tilde{x} - \lambda x_0| < 200 \sqrt{|\log\varepsilon|}\,  \lambda\,\sqrt{1-\tilde{t}}$, and all $\tilde{t} \in [0,\frac12]$. By \eqref{eq-conclusion} applied to $U_{\lambda}(\tilde{x},\theta,\tilde{t})$ we get
 \[\Big |U_{\lambda}(\tilde{x},\theta,\tilde{t}) - \sqrt{2(n-k)(1-\tilde{t})}\Big |< \varepsilon_0\sqrt{1-\tilde{t}},\]
 (and we also have the corresponding, correctly scaled higher order derivative estimates up to the 4th order) for all $|\tilde{x} - \lambda x_0| < 100 \sqrt{|\log\varepsilon|}\, \lambda$, and all $\tilde{t} \le 1- \mathcal{G}_{\varepsilon}^{\frac{99}{100}}$. If we unravel the changes of variables, we can easily see that the corresponding estimate for $\bar{U}(x,\theta,\bar{t})$, given that $\bar{t} = \frac{\tilde{t}}{\lambda^2}$, holds for all $\bar{t} \le 2\,\Big(1-\mathcal{G}_{\varepsilon}^{\frac{99}{100}}\Big)(t(x_0) + \varepsilon^{\frac{99}{100}})$. Finally, since $t = \bar{t} + 1 - \varepsilon^{\frac{99}{100}} - \beta$, 
we get
\begin{equation*}
\label{eq-extend222}
\Big  | U(x,\theta,t) - \sqrt{2(n-k)(1+t(x_0)-t)}\Big |< \varepsilon_0\sqrt{1+t(x_0)-t},
\end{equation*}
 (and we also have the corresponding, correctly scaled higher order derivative estimates up to the 4th order) for all $|x-x_0| \le 100 \sqrt{|\log\varepsilon|}$, and all
 \[t \le 1 - \varepsilon^{\frac{99}{100}} - \beta + 2(1-\mathcal{G}_{\varepsilon}^{\frac{99}{100}})(t(x_0)+\varepsilon^{\frac{99}{100}}).\]
 Since by \eqref{eq-choice} we have $\beta = \frac{1}{2\lambda^2} = t(x_0) + \varepsilon^{\frac{99}{100}}$, we finally get that \eqref{eq-extend} holds for 
 \begin{equation}
 \label{eq-time-extension}
 t \le 1  + t(x_0) - 2\mathcal{G}_{\varepsilon}^{\frac{99}{100}}t(x_0) - 2\mathcal{G}_{\varepsilon}^{\frac{99}{100}}\varepsilon^{\frac{99}{100}}
  \end{equation}
  \end{proof}
  
 In the last step we will use  Step \ref{step-extension} to show that for large values of $x_0$, the $L^{\infty}$ estimate \eqref{eq-extend}, together with corresponding, appropriately rescaled derivative estimates up to the 4th order,  hold past time $T = 1$.
 
 \begin{step}\label{step-precise}
 Let $\frac{L}{8} \le |x_0| \le \frac L4$, where $L \in (\varepsilon^{-\frac{1}{500}}, \varepsilon^{-\frac{1}{300}})$. Then \eqref{eq-extend} holds for all $|x-x_0| \le 100\sqrt{|\log\varepsilon|}$, and all 
 $ t \le 1 + (1-\eta) \, t(x_0)$, where $\eta = \eta(\varepsilon) \to 0$ as $\varepsilon\to 0$.
 \end{step}

\begin{proof}
  It is easy to see that if  $|x_0| <  \varepsilon^{-\frac{1}{250}}$ then $\mathcal{G}_{\varepsilon} = o(1)$ as $\varepsilon\to 0$, and hence the third term on the right hand side in \eqref{eq-time-extension} can be absorbed in $t(x_0)$. We would like to verify that the second term can be also absorbed in $t(x_0)$. We will consider two cases.

Assume first that $\mathcal{G}_{\varepsilon} \sim \varepsilon^{\frac{1}{100}}\sqrt{|\log\varepsilon|} |x_0|$.  In this case we would like to have 
\[
t(x_0) \gg \Big(\varepsilon^{\frac{1}{100}}\sqrt{|\log\varepsilon|} |x_0|\Big)^{\frac{99}{100}} \varepsilon^{\frac{99}{100}}.
\]
 Since our $|x_0|$ is large,  keeping in mind that in this case $t(x_0) \sim \varepsilon \,x_0^2$, it is straightforward to see that above holds when 
 \begin{equation}
 \label{eq-req1}
 |x_0| \gg |\log\varepsilon|^{\frac{99}{202}}\, \varepsilon^{-0.000099}
 \end{equation}
 Assume now that $\mathcal{G}_{\varepsilon} \sim  \varepsilon^{\frac{1}{51}}(1+|x_0|^2)^2$. Since $|x_0|$ big, we have that $1 + |x_0|^2 \sim |x_0|^2$. This implies that  in order to have that $t(x_0) \gg \Big(\varepsilon^{\frac{1}{51}}(1+|x_0|^2)^2\Big)^{\frac{99}{100}} \, \varepsilon^{\frac{99}{100}}$, it is enough to have 
\[t(x_0) \gg  \Big( \varepsilon^{\frac{1}{51}}\,|x_0|^4\Big)^{\frac{99}{100}}\, \varepsilon^{\frac{99}{100}},\]
 where $t(x_0) \sim \varepsilon |x_0|^2$. Now it is straightforward to see that
 \begin{equation}
 \label{eq-req2}
 |x_0| \ll \frac{1}{\varepsilon^{1/250}}.
 \end{equation}
 It is straightforward to check that if we take $|x_0| \in (\frac L8, \frac L4)$, where $\varepsilon^{-\frac{1}{500}} \le L \le  \varepsilon^{-\frac{1}{300}}$, then both,  \eqref{eq-req1} and \eqref{eq-req2}  are satisfied.
  More precisely, we get that \eqref{eq-extend} holds for $t \le 1 + (1-\eta) \, t(x_0)$, where $\eta = \eta(\varepsilon) \to 0$ as $\varepsilon \to 0$. This finishes the proof of Step \ref{step-precise}.
 \end{proof}

 \smallskip
 We will finish the proof of the theorem by showing that  \eqref{eq-U-quadratic} holds. To this end we evaluate   \eqref{eq-extend} 
 at $x_0$ and vary it (we now call is $x$)   to get  the bound 
\begin{equation}
\label{eq-neck-1}
\Big  | U(x,\theta,t) - \sqrt{2(n-k)(1+t(x)-t)}\Big |< \varepsilon_0\sqrt{1+t(x)-t},
\end{equation}
holding for any  $|x| \in (\frac{3L}{16} ,\frac{L}{5})$ and $t \le 1 + (1-\eta) t(x)$, where  by definition $t(x) = A_0 \ep  (2k + |x|^2)\, (1 + o(1))$.
This gives
$$
(\sqrt{2(n-k)} - \ep_0) \, \sqrt{1+t(x)-t} \leq     U(x,\theta,t) \leq (  \sqrt{2(n-k)} + \ep_0 ) \, \sqrt{1+t(x)-t}
$$
implying, after squaring, the bound
$$
 ( 2(n-k) - A_0\,\ep_0) \, (1-t) +  ( 2(n-k) - A_0\,\ep_0) \, t(x) \leq  U^2(x,\theta,t)    \leq   (2(n-k) + A_0\, \ep_0)  \, (1-t) + (2(n-k) + A_0\,\ep_0) \, t(x),
$$
where $A_0$ is a numeric constant. Now for $|x|$ large we have $t(x)  \sim \ep \, |x|^2$ and the above bound takes the form 
\begin{equation}
\label{eq-help1111}
 ( 2(n-k) - A_0\,\ep_0) \, (1-t) + A_1 \, \ep \, |x|^2  \leq  U^2(x,\theta,t)    \leq   (2(n-k) + A_0\,\ep_0)  \, (1-t) + A_2 \, \ep \, |x|^2
\end{equation}
for some fixed numeric constants $A_0,A_1, A_2$ and it holds for all $t \leq 1+ (1-\eta) \, t(x)$, where $\eta = \eta(\varepsilon)\to 0$ as $\varepsilon \to 0$.

 Finally, keeping in mind that $t(x) = A_0 \ve (2k + |x|^2)\, (1+o(1))$, and $|x| \in (\frac L8, \frac L4)$, where $\ve^{-\frac{1}{500}} \le L \le \ve^{-\frac{1}{300}}$, since \eqref{eq-neck-1} holds for all $t \le 1+ \varepsilon^{\frac{248}{250}}$, this implies right away that  any $p\in M_t$ with $t\leq 1+ \ve^{\frac{248}{250}}$,  and $|\langle p,e_1\rangle| \in (\frac{3L}{16} ,\frac{L}{5})$ lies on an $\varepsilon_0$-neck. 
 \end{proof}

\bigskip

\bigskip

We will next show the  local estimates on cylindrical radius obtained in this section are sufficient to conclude that the singularity time of the flow with initial data $U(x,\theta,0)=\sqrt{2(n-k)}+\varepsilon \,|x|^2 \, +o(\varepsilon)\, (1 + |x|^2)$, for $|x| \leq L$,  with $\varepsilon^{-\frac{1}{500}}\leq L\leq \varepsilon^{-\frac{1}{300}}$, where $\varepsilon > 0$ is small,   is not much bigger than one.

\begin{lemma}
\label{lemma-singular-time}
The first singular time $T$ satisfies  $T \leq 1+\,\varepsilon^{0.9996}$. In addition, the estimate \eqref{eq-U-quadratic} holds for all times $t\le T$.
\end{lemma}

\begin{proof}
By Step \ref{step-extension} in the proof of Proposition \ref{the:first_singular_time} we have that for all $|x_0| < \tfrac L4$, we have
\begin{equation}
\label{eq-radius-x0}
\Big | U(x,\theta,t) - \sqrt{2(1+t(x_0) - t)}\Big | < \ep_0 \, (1+|x_0|^2) \, \sqrt{1+t(x_0) -t},
\end{equation}
for all  $\theta \in \mathbb{S}^{n-k}$, all $|x - x_0| \le 200\, \sqrt{|\log|\varepsilon|}$, and all $t \le 1  + t(x_0) - 2\mathcal{G}_{\varepsilon}^{\frac{99}{100}}t(x_0) - 2\mathcal{G}_{\varepsilon}^{\frac{99}{100}}\varepsilon^{\frac{99}{100}}$, where $\mathcal{G}_{\varepsilon}$ is given by \eqref{eq-Geps}. In particular, apply this to $x_0 = 0$. Using that $t(0) \sim A_0\, \varepsilon$, and that at $x_0 = 0$, $\mathcal{G}_{\varepsilon}  \sim C_0 \varepsilon^{\frac{1}{100}}\, |\log\varepsilon|$, we conclude that \eqref{eq-radius-x0} holds for $x_0 = 0$, for all $|x - x_0| \le 200\, \sqrt{|\log\varepsilon|}$ and all  $t \le 1 - C_0 |\log\varepsilon|^{\frac{99}{100}} \, \varepsilon^{0.9999}$ (since $\mc G_{\ve}^{\frac{99}{100}} \,\ve^{\frac{99}{100}} \gg t(0)$, as $\ve \to 0$). Thus, it certainly holds for all $t \le 1 - \varepsilon^{0.9998}$.

Plug in the time $t_\ep:=1-\ep^{0.9998}$ in \eqref{eq-radius-x0} for $x_0 = 0$ to obtain the bound 
\begin{equation}
\label{eq-radius-1}
U(x,\theta,t_\ep) \leq   (\sqrt{2} + \varepsilon_0) \, \sqrt{t(0) +\ep^{0.9998}} \le 2 \varepsilon^{\frac{0.9998}{2}}.
\end{equation}
By \eqref{eq-radius-1} we can place outside of $M_{1-\varepsilon^{0.9998}}$  a tiny Angenent torus  centered at the origin $x=0$  whose minimal geodesic is the circle of radius 
$\ell(\tau)_A := 4 \ep^{\frac{0.9998}{2}}$. Note that such  the torus develops a singularity before 
$t_{\ep, A}=  \frac 12 \, \rho_A^2  = 8 \,\varepsilon^{0.9998}.$
  Thus, by the comparison principle, $M_t$ develops a singularity before $1- \ep^{0.9998} + t_{\varepsilon,A} \le 1 + \varepsilon^{0.9996}$.
  
 Note that by Proposition \ref{the:first_singular_time} we have that \eqref{eq-U-quadratic} holds for all $t \le 1 + \frac12\,\varepsilon^{\frac{249}{250}}$. Since $T \le 1 + \varepsilon^{0.9996} < 1 + \varepsilon^{\frac{249}{250}}$, we conclude that \eqref{eq-U-quadratic} holds for all $t \le T$ as claimed.

\end{proof}

\bigskip

\begin{corollary}\label{cor:cylindrical_est}
For $t < T$ in $M_t\cap \{|x| \le \frac25 L\}$ we have
\begin{equation}
\label{eq-cylindrical-curv-est}
|A|^2 \le (\frac{1}{n-k} + C\ep_0)\, H^2, \qquad |\nabla A|^2 \le C\, \ep_0 H^4
\end{equation}
for a uniform constant $C$. In particular, the first singularities are cylindrical singularities.
\end{corollary}

\begin{proof}
Note that by   the proof of Proposition \ref{the:first_singular_time}, see in particular Step \ref{step-precise}, and Lemma \ref{lemma-singular-time}, we have that \eqref{eq-extend}, and the corresponding correctly scaled higher order derivative estimates up to the 4th order, hold for all $|x-x_0| \le 100\, \sqrt{|\log\ve|}$,  where $|x_0| \in [\frac L8, \frac L4]$, and all $t\le T$, where $T$ is the first singular time.  By evaluating at $x_0$ and then varying $x_0$, so that $|x_0| \in [\frac L8, \frac L4]$, above discussion in particularly implies that $H > 0$, $\frac{|A|^2}{H^2} \leq \frac{1}{n-k} + C\varepsilon_0$, and $\frac{|\nabla A|^2}{H^4} < C\varepsilon_0$ on the lateral boundary $\{|x|=\frac{L}{5}\}$.  We also have it at time $t = 0$, on the set $|x| \le \frac45 L$. Recall that
\[\Big(\frac{\partial}{\partial t} - \Delta\Big) \frac{|A|^2}{H^2} = \frac{2}{H} \Big\langle \nabla H, \nabla\frac{|A|^2}{H^2}\Big\rangle - \frac{2}{H^4} |H\nabla_i h_{kl} - \nabla_i H h_{kl}|^2\]
\[\big(\frac{\partial}{\partial t} - \Delta\big) H = |A|^2 H,\]
We can now apply the maximum principle with the boundary to conclude that we first have $H > 0$ for all $|x| \le \frac45 L$, and all $t \le T$, where $T$ is the first singular time, and second, that at the maximum of $\frac{|A|^2}{H^2}$, we have $\frac{d}{dt}\big(\frac{|A|^2}{H^2}\big)_{\max} \le 0$,  hence implying the first estimate in \eqref{eq-cylindrical-curv-est} holds. 

To see that the second estimate in \eqref{eq-cylindrical-curv-est} holds as well, recall that if we define $\Phi := \frac{|\nabla A|^2}{H^4}$, then
\[\Big(\frac{\partial}{\partial t} - \Delta\Big)\Phi = \frac{2}{H}\langle \nabla H, \nabla\Phi\rangle - \frac{2}{H^4} | H \nabla_k h_{ij} - \nabla_k H h_{ij}|^2 + \mathcal{R}.\]
Since we have just proved that the first pinching curvature estimate in \eqref{eq-cylindrical-curv-est} holds, using that, it is straightforward to see that $|\mathcal{R}| \le C_0 \Phi$, where $C_0$ is a uniform constant. This implies
\[\Big(\frac{\partial}{\partial t} - \Delta\Big)\Phi \le \frac{2}{H}\langle \nabla H, \nabla\Phi\rangle  + C_0 \Phi.\]
The second estimate in \eqref{eq-cylindrical-curv-est} is now direct application of the maximum principle with boundary applied to $\Phi$.
\end{proof}  
  
We are now ready to prove  Theorem \ref{thm-main0}.

\begin{proof}[Proof of Theorem \ref{thm-main0}]
This result is a direct consequence of Corollary \ref{cor:cylindrical_est}. More precisely, the scale-invariant curvature estimates \eqref{eq-cylindrical-curv-est} guarantee that any tangent flow at a singular point in the region where $|x| \le \frac{L}{5}$ must be a generalized cylinder $ \mathbb{R}^k \times \mathbb{S}^{n-k}$. Furthermore, these estimates ensure that the convergence of the rescaled flows to the tangent flow is smooth. Because the evolving surfaces are embedded, this smooth convergence implies that the cylindrical tangent flow occurs with multiplicity one. It then follows from the results of Colding and Minicozzi \cite{CM2} that at each singular point in this region, the generalized cylinder is the strictly unique tangent flow.
\end{proof}

  
\section{Choice of rotation and estimate for angular derivatives}  
\label{sec-rotation}

In Section \ref{sec:neckpinch-sing}, we showed that the mean curvature flow, starting with the initial data described in Theorem \ref{thm-main0} and satisfying \eqref{eq-Ueps}, develops only neckpinch singularities on a large compact set. In this section, we will assume that in addition to \eqref{eq-Ueps}, there exist universal constants $\varepsilon_n > 0$ sufficiently small and $L_n$ sufficiently large, depending only on the dimension, such that the flow $M_0$ is $\varepsilon_n$-close in the $C^k$-sense on   $B_{L_{\varepsilon}}(0)$ to the cylinder   $\mathbb{R}^k \times \mathbb{S}^{n-k}(\sqrt{2(n-k)})$.  We assume its profile $U(x,\theta,0)$ 
satisfies $|U_\theta(\cdot,\cdot,0)| = o(\varepsilon)$, where $L_{\varepsilon}$ is defined as in Theorem \ref{thm-main0}.

To prove Theorem \ref{thm-main}, for an arbitrary spacetime point $\cX = (q, t_0)$, where $q = (x_0,x_0') \in \mathbb{R}^k \times \mathbb{S}^{n-k}$, and an arbitrary rotation $S$, we consider the parabolically rescaled flow around $\cX$. After applying the rotation $S$, this flow can be locally written as a graph over the same cylinder $\mathcal{C}_{n,k}$ as in the statement of Theorem \ref{thm-main}. In this section we find an optimally fitting cylinder, or quivalently, we find the rotation to apply to our flow $M_t$ such that when the rotated flow is written locally as a graph over $\mathcal{C}_{n,k}$, all angular derivatives of the graphical radius are very small and exponentially decaying in time. We also estimate the norm of this rotation and show that it is very close to the identity map. If we choose the correct center of the neck and the correct singular time, our results yield to  the unique rotation with respect to which our rescaled solution converges to $\mathcal{C}_{n,k}$,  as $\tau \to \infty$.

We recall the following notions from Du-Zhu \cite{DZ}. Given constants $r, \rho  >0$  and a space-time point 
$(\bar p,\bar t) \in  \mathcal{M} :=  \{ M_t \} $  with $H(\bar p,\bar t)>0$, we denote by $\hat P(\bar p,\bar t,r,\tau)$ the parabolic neighborhood of $(\bar p,\bar t)$, defined as  
\begin{equation}
\hat P(\bar p,\bar t,r,\rho)=B(\bar p,H(\bar p,\bar t)^{-1}r)\times [\max\{0, \bar t-H(\bar p,\bar t)^{-2}\rho\},\bar t].
\end{equation}

 \begin{definition}[$\ep$-close to cylinder] \label{def:eps-cyl}We say that $(\bar p,\bar t)$  {\em lies on a $\varepsilon$-cylinder  $\mathbb{R}^k\times {\mathbb S}^{n-k}$ if $\mathcal{M}$} is $\varepsilon$-close (in $C^{4}$-norm) to a shrinking cylinder $\mathbb{R}^k\times {\mathbb S}^{n-k}$ in $\hat P(\bar p,\bar t,100n^{5/2},100^2n^2)$ after rescaling by $H(\bar p,\bar t)$. 
 
 \end{definition}

\begin{definition}[$(\varepsilon,n-k)$-symmetric points]\label{def:normal_rot_vec}
Assume that $H(\bar p, \bar t) >0$. We say that $(\bar p,\bar t)$ is {\em $(\varepsilon,n-k)$-symmetric}  for some $k$ in  $1\leq k\leq n-1$, if there is a normalized set of rotation vector fields $\mathcal{K}=\{K_\alpha:1 \leq \alpha \leq \frac{1}{2}(n-k+1)(n-k)\}$  such that $\max_\alpha | K_\alpha| H \leq 5n$  and $\max_\alpha |\langle K_\alpha,\nu\rangle|H \leq \varepsilon$ hold in the parabolic neighborhood $\hat P(\bar p,\bar t,100n^{5/2} ,100^2n^5)$, where the vector fields are defined by
\begin{equation}
K_\alpha (p)=SJ_\alpha S^{-1}(p-q), \qquad S\in O(n+1), \qquad  q \in \mathbb{R}^{n+1},
\end{equation}
and
\begin{equation}
J_\alpha=\begin{pmatrix}
0 & 0 \\ 0 & \bar J_\alpha
\end{pmatrix}\in \textsc{SO}(n+1),
\end{equation}
and $\{\bar J_\alpha: 1 \leq \alpha \leq \frac{1}{2}(n-k-1)(n-k)\}$ is an orthonormal basis of $\textsc{SO}(n-k+1)\subset \textsc{SO}(n+1)$.
\end{definition}

\bigskip

\begin{theorem}[Symmetry Improvement, Du-Zhu \cite{DZ}]\label{thm:sym_improve}
There exist $L_1 \geq 1$ and $\varepsilon_1\in (0,1)$ depending on the dimension $n$ with the following significance. Let $\bar t \geq L_1^2 H(\bar p,\bar t)^{-2}$. Suppose that every point in $\hat P(\bar p,\bar t,L_1,L_1^2)$  lies on an $\varepsilon_1$-cylinder $\mathbb{R}^k\times {\mathbb S}^{n-k}$ and every point in $\hat P(\bar p,\bar t,L_1,L_1^2)$ is $(\varepsilon,n-k)$-symmetric, where $0<\varepsilon \leq\varepsilon_1$. Then, $(\bar p,\bar t)$ is $(\varepsilon/2,n-k)$-symmetric.
\end{theorem}

\begin{remark}
If $\bar t\leq L_1^2 H(\bar p,\bar t)^{-2}$, there exists some $p_0$ such that $(p_0,0)\in \hat P(\bar p,\bar t,L_1,L_1^2)$. Then, $\hat P(p_0,0,100n^{5/2} ,100^2n^5)\subset \mathbb{R}^{n+1}\times \{0\}$, which is not a parabolic open set. However, the proof of the cylindrical symmetry improvement theorem in \cite{DZ} works even in this initial value problem as long as we have $\bar t\geq L_1^2 H(\bar p,\bar t)^{-2}$. 
\end{remark} 
 
\bigskip

\begin{theorem}[Symmetry estimate with Cauchy data]\label{thm:sym_estimate_cauchy}
There exist $L_1 \geq 1$ and $\varepsilon_1\in (0,1)$ depending on the dimension $n$ with the following significance. Suppose that for $t\in [0,\frac{3}{4}]$, $M_t$ is $\varepsilon_1$-close in $C^4$-sense to the shrinking cylinder $\mathbb{R}^k\times {\mathbb S}^{n-k}(\sqrt{2(n-k)(1-t)}\,)$ for $x \in B^k_{L_1}(0)$, and 
\begin{equation}
\sup_{|x| \leq L}|\nabla_\theta U (x,\theta,0)| \leq \delta 
\end{equation}
as well as 
\begin{equation}
\sup_{|x| \leq L}|\nabla_\theta U (x,\theta,t)| \leq \varepsilon
\end{equation}
hold,  for  all $t\in [0,\frac{3}{4}]$,  where $\delta \leq \varepsilon \leq \varepsilon_1$ and $L\geq L_1$. Then, we have
\begin{equation}
\sup_{|x| \leq 1}|\nabla_\theta U (x,\theta,t)| \leq  2\delta+ \tfrac{1}{4}\varepsilon,
\end{equation}
for all $t\in [0,\frac{3}{4}]$.
\end{theorem}

\begin{proof}
Suppose there is $(x_0,\theta_0,t_0)$ such that $t_0 \in (0,\frac{3}{4}]$ and 
\begin{equation}
|\nabla_\theta U(x_0,\theta_0,t_0)|=\sup_{t \leq t_0} \sup_{|x| \leq L}|\nabla_\theta U (x,\theta,t)|.
\end{equation}
Then, we consider a local coordinate in a neighborhood $\mathscr U$ of $(x_0,\theta_0)$ at $t=t_0$ such that
$\nabla_1,\cdots, \nabla_{n-k}$ are covariant derivatives along $\mathbb{S}^{n-k}$-factor and $\nabla_{n-k+1},\cdots,\nabla_n$ are derivatives along $\mathbb{R}^k$-factor, and the metric $g_{ij}$ satisfies $g_{11} \leq 1$ in $\mathscr U$. Moreover, at $(x_0,\theta_0,t_0)$, $g_{ij}(x_0,\theta_0,t_0)=\delta_{ij}$, the second fundamental form $h_{ij}(x_0,\theta_0,t_0)=\delta_{ij}$ for $i,j\leq n-k$, and $|\nabla_\theta U(x_0,\theta_0,t_0)|=\nabla_1 U(x_0,\theta_0,t_0)$. Then, observing $|\nabla_1 U|^2 \leq g_{11}|\nabla_\theta U|^2\leq |\nabla_\theta U|^2$, $f:=\nabla_1 U$ attains its maximum in the set $\{|x| \leq L\}\cap \{t\leq t_0\}$ at the point $(x_0,\theta_0,t_0)$.  

\bigskip

Next, we derive the evolution equation of $U$ as in \eqref{eq-rescaled-u}.
\begin{equation}
\begin{split}
 U_t &= \Delta_x U + \frac{1}{U^2} \Delta_{\theta} U   - \frac{n-k}{U}- \frac{|\nabla_{\theta} U|^2}{U^3 W^2}\\
 &- \frac{1}{W^2}\, \left(\Hess_x U (\nabla_x U, \nabla_x U) + \frac{2}{U^2}\, \Hess_{x,\theta}U (\nabla_x U, \nabla_{\theta} U) + \frac{1}{U^4}\, \Hess_{\theta} U (\nabla_{\theta,}U, \nabla_{\theta}U)\right),
 \end{split}
 \end{equation}
 where
 \begin{equation}
W=\sqrt{1+|\nabla_x U|^2+\frac{|\nabla_\theta U|^2}{U^2}}.
 \end{equation}
Next, for $i,j \leq n-k$,  we have $g^{ij}\nabla_i\nabla_1 \nabla_jU=g^{ij}\nabla_1\nabla_i\nabla_jU+g^{ij}R_{i1jl}\nabla_l U=\nabla_1\Delta_\theta U+(n-k-1)\nabla_1U$ at $(x_0,\theta_0,t_0)$. Namely,
\begin{equation}
\nabla_1\Delta_\theta U=\Delta_\theta f-(n-k-1)f.
\end{equation}
Therefore,
\begin{equation}
f_t  = \Delta_x f + \frac{\Delta_{\theta} f  + f}{2(n-k)(1-t)} + a_{ij}f_{ij}+b_if_i+cf:= {\mathscr P} V
 \end{equation}
where
\begin{equation}
|a_{ij}|+|b_i|+|c|  \leq C\varepsilon.
\end{equation}

Now, we recall that first Dirichlet eigenfunction  $Y(x)$ on the unit  ball  $B_1(0)$ in $\R^k$, namely the solution of  $\Delta_{\mathbb{R}^k}Y+\lambda Y=0$ in $B_1(0)$
with $Y(x)=0$ on $|x|=1$. It is well known that $\lambda >0$, $Y(x) >0$ in $B_1(0)$ and $Y(x)$ is symmetric.   We also assume $Y(0)=1$. Then, it satisfies 
\begin{equation}\label{eq:first_eigen_ball}
Y(x)=1- \frac{\lambda}{2k}|x|^2+O(|x|^4) 
\end{equation}
for small $|x|$. Define $\varphi(x,\theta,t)$ by 
\begin{equation}
\varphi(x,\theta,t)=(1-t)^{-\frac{1}{2(n-k)}}\, \big ( \varepsilon  - (\varepsilon -\delta )Y(L^{-1} x)e^{-\lambda t/L^2} \big ).
\end{equation}
Then,
\begin{equation}
|\varphi_t- {\mathscr P}\varphi|= |a_{ij}\varphi_{ij}+b_i\varphi_i+c\varphi|\leq C\varepsilon^2.
\end{equation}
Thus, there is some numeric constant $B$ only depending on $n$ such that $\psi:=\varphi+B\varepsilon^2t$ is a supersolution to $\psi_t\geq {\mathscr P}\psi$. Also, we observe
\begin{equation}
\psi(x,\theta,0)-\delta= (\varepsilon- \delta) (1-Y) \geq 0. 
\end{equation}
Moreover, for $|x|=L$
\begin{equation}
\psi(x,\theta,t)\geq \varepsilon.
\end{equation}
Namely, $\psi \geq |\nabla_\theta U| \geq f$ on the parabolic boundary of $B_{L}(0)\times {\mathbb S}^{n-k}\times [0,t_0]$. Therefore, by the maximum principle we have $f\leq \psi$. Thus, by \eqref{eq:first_eigen_ball}, in $B_1(0)\times {\mathbb S}^{n-k}\times [0,\frac{3}{4}]$ we get
\begin{equation}
f \leq  (1-t)^{-\frac{1}{2(n-k)}}[\delta + C \varepsilon L^{-2}]+B\varepsilon^2.
\end{equation}
Hence, choosing sufficiently small $\varepsilon_1$ and large $L_1$, we have
\begin{equation}
V \leq  2\delta+ \tfrac{1}{4}\varepsilon.
\end{equation}
This completes the proof.
\end{proof}

\bigskip
 
\begin{corollary}\label{cor:sym_estimate_cauchy}
There exist $\varepsilon_0\in (0,1)$,  a constant $L_1\geq 1$, and a constant $C$ only depending on the dimension $n$,  with the following significance. Suppose that $M_0$ is $\varepsilon$-close in $C^4$-sense  to the round cylinder $\mathbb{R}^k\times {\mathbb S}^{n-k}(\sqrt{2(n-k)}\,)$ in $B_{L}^k(0)\times {\mathbb S}^{n-k}$,
where  $\varepsilon \in (0,\varepsilon_0)$, $\delta \leq \varepsilon$ and   $L > 4L_1 |\log \delta|$, and also suppose 
\begin{equation}
\label{eq-help-555}
\sup_{|x| \leq L}|\nabla_\theta U (x,\theta,0)| \leq \delta. 
\end{equation}
Then, for $t\in [0,1-(2L_1)^{-2}]$ we have
\begin{equation}
\sup_{|x| \leq L- 4L_1|\log \delta| }|\nabla_\theta U (x,\theta,t)| \leq  C\delta .
\end{equation}
\end{corollary} 

\begin{proof}
Let $L_1,\varepsilon_1$be as  in  Theorem \ref{thm:sym_estimate_cauchy} and Theorem \ref{thm:sym_est}. By Proposition \ref{prop-unique-cont}, there are $R_1,\varepsilon_0$ only depending on $n$ such that $M_t$ is $\varepsilon_1$-close to the cylinder $ \R^k\times { \mathbb S}^{n-k}(\sqrt{2(n-k)(1-t)}\, )$ in the $C^4$-sense for 
 $ x \in  B_{L_\varepsilon-R_1}^k(0)$ and $t\in [0,1-(2L_1)^{-2}]$. Namely,
\begin{equation}\label{eq:rot_rough_base}
\sup_{|x| \leq L-R_1}|\nabla_\theta U(x,\theta,t)|\leq C_0\varepsilon_1  \sqrt{1-t}
\end{equation}
for $t\in [0,1-(2L_1)^{-2}]$. Applying Theorem \ref{thm:sym_estimate_cauchy} yields
\begin{equation}
|\nabla_\theta U (x,\theta,t)| \leq 2\delta + \tfrac 14 C_0\varepsilon_0\leq \tfrac 12 C_0\varepsilon_1
\end{equation}
for $|x| \leq L-R_1-L_1$ and $t \leq \frac{3}{4}$, assuming $2^3\delta  \leq C_0\varepsilon_1$. We apply Theorem \ref{thm:sym_estimate_cauchy} again to get
\begin{equation}
|\nabla_\theta U (x,\theta,t)| \leq 2\delta+\tfrac{1}{8}C_0\varepsilon_1 \leq \tfrac{1}{4}C_0\varepsilon_1
\end{equation}
for $|x| \leq L-R_1-2L_1$ and $t \leq \frac{3}{4}$, assuming $2^4\delta  \leq C_0\varepsilon_1$. We may assume $2^{2+l}\delta \leq C_0\varepsilon_1< 2^{3+l}\delta $ and repeat this process. Then,
\begin{equation}
|\nabla_\theta U (x,\theta,t)| \leq 2\delta+2^{-l-1}C_0\varepsilon_1 \leq 2^{-l}C_1\varepsilon_1 \leq 8\delta
\end{equation}
for $|x| \leq L-R_1-l \, L_1$ and $t \leq \tfrac{3}{4}$ assuming $\frac{3}{4}\leq 1-(2L_1)^{-2}$.

Now, we take a parabolic rescaling $M_t^1$ of $M_t$ by $2$ around $(0,1)$. Namely, $M_t^1=2M_{1-4(1-t)}$. Then,  
\begin{equation}\label{eq:rot_rough_first}
\sup_{|x| \leq 2(L-R_1)}|\nabla_\theta U^1(x,\theta,t)|\leq C_0\varepsilon_1 \sqrt{1-t}
\end{equation}
holds for $t \leq 1-4(2L_1)^{-2}$, and 
\begin{equation}\label{eq:rot_initial_first}
|\nabla_\theta  U^1 (x,\theta,t)| \leq 16\delta:=\delta_1
\end{equation}
holds for $|x| \leq 2(L-R_1-lL_1)$ and $t \leq 0$. Note that the estimates \eqref{eq:rot_rough_first} and \eqref{eq:rot_initial_first} for $U^1$ replace \eqref{eq-help-555} and \eqref{eq:rot_rough_base}. Hence,  by repeating the above process we obtain
\begin{equation}
|\nabla_\theta  U^1 (x,\theta,t)| \leq 2\delta_1+2^{-l+3}C_0\varepsilon_1 \leq 2^{-l+4}C_0\varepsilon_1 \leq 8\delta_1=128\delta
\end{equation}
for $|x| \leq 2(L-R_1-lL_1)-(l-4)L_1$ and $t \leq \frac{3}{4}$ assuming $\frac{3}{4}\leq 1-4(2L_1)^{-2}$. Hence, considering the relation $U^1=2U$, we have
\begin{equation}
|\nabla_\theta   U (x,\theta,t)|  \leq 8^2\delta
\end{equation}
holds for $|x| \leq L-R_1-l(1+2^{-1})L_1$ and $t \leq 1-4^{-2}$ assuming $(2L_1)^{-2}\leq 4^{-2}$. By repeating this process for $M^k:=2M^{k-1}_{1-4(1-t)}$, we get $|\nabla_\theta U^k| \leq  2\delta_k+2^{-l+4k-1}C_0\varepsilon_1 \leq 2^{-l+4k}C_0\varepsilon_1 \leq 8\delta_k$ where $\delta_k:=16^k\delta$. Namely, 
\begin{equation}
|\nabla_\theta U^{k+1}| \leq 16\delta_k=\delta_{k+1}
\end{equation}
holds for $|x| \leq 2^{k+1}(L-R_1-l(1+2^{-1}+\cdots+2^{-k})L_1)-(l-4^{k+1})L_1$ holds for $t\leq  \frac{3}{4}$ assuming $(2L_1)^{-2}\leq 4^{-k-1}$. In other words,
\begin{equation}
|\nabla_\theta   U (x,\theta,t)|  \leq 8^m\delta
\end{equation}
for $|x| \leq L-R_1-l(1+\cdots+2^{-m+1})L_1\leq L-R_1-2lL_1$ and $t \leq 1-4^{-m}$ assuming $(2L_1)^{-2}\leq 4^{-m}$. Hence, we choose $m$ by $4^{-m-1}< (2L_1)^{-2} \leq 4^{-m}$ so that $m$ only depends on $n$ as well. Hence, remembering $2^{1+2} \delta \leq C_0\varepsilon_1 < 2^{l+3}\delta$, we can obtain $R_1+2lL_1 \leq 4L_1 |\log \delta|$. This completes the proof.
\end{proof} 
 
 \bigskip
\begin{remark}\label{rmk:rough_cylindricality}
We note that by Corollary \ref{cor:cylindrical_est}, given $\varepsilon_1$ there is small enough $\varepsilon_0$ such that for any $(x_0,x_0')\in M_{t_0}$ with $|x_0|\leq \frac{2}{5}L$ and $t_0<t$, $(x_0,x_0',t_0)$ lies on the center of an $\varepsilon_1$-shrinking cylinder.
\end{remark}

\begin{theorem}\label{thm:sym_est}
For $1\leq k\leq n-1$ and $t\in [0,T)$, where $T$ is the first singular time, we denote by $\Gamma_t\subset M_t$ the connected component that $\partial \Gamma_t=\{(x,U(x,\theta,t) \theta):|x| =\frac{1}{3}L_\varepsilon, \theta \in {\mathbb S}^{n-k}\subset \mathbb{R}^{n-k+1}\}$. We assume that for $\delta \in (\varepsilon^2,\varepsilon)$ and $\varepsilon\in (0,\varepsilon_0)$, $\Gamma_0$ is $\varepsilon_0$-close in $C^4$-sense to the round cylinder $\mathbb{R}^k\times {\mathbb S}^{n-k}(\sqrt{2(n-k)}\,)$, and the profile $U(x,\theta,0)$ of $\Gamma_0$ satisfies
\begin{equation}\label{eq:delta_symmetry_negative_time}
|\nabla_\theta U(\cdot,\cdot, 0)| \leq \delta.
\end{equation}
Then, for every $\bar p=(\bar x,\bar x') \in \Gamma_{\bar t}\subset M_{\bar t}$ with $\bar t \in (0,T)$ and $|\bar x|\leq \frac{1}{4}L_\varepsilon$, $(\bar p,\bar t)$ is $(C\delta  H(\bar p,\bar t)^{-\sigma},n-k)$-symmetric, where $\sigma:=\log 2 / \log(2L_1)$ and $C$ is some constant  only depending on $n$,  where  $L_1 \geq 2$ is the number from Theorem \ref{thm:sym_estimate_cauchy}.
\end{theorem}

\begin{proof}
Let $\varepsilon_1$ and $L_1$ be as  in  Theorem \ref{thm:sym_estimate_cauchy}. We may assume without loss of generality that
$L_1\geq 2$.
In this proof, we always assume $\bar p\in \Gamma_{\bar t}$. Also, we denote by $\pi:\mathbb{R}^{n+1}\to \mathbb{R}^k$ the projection $\pi((x,x'))=x$. To begin with, we combine the above conditions,  by Remark \ref{rmk:rough_cylindricality} so that we obtain
\begin{equation}\label{eq:eps1_cyl_condition}
(\bar p,\bar t)\;  \text{lies on an} \; \varepsilon_1\text{-cylinder} \;\mathbb{R}^k\times \mathbb{S}^{n-k}\;\text{for every}\; \bar t \in [0,T)\;\text{and}\;\bar p\in \Gamma_{\bar t},
\end{equation}
and
\begin{equation}\label{eq:eps1_sym}
(\bar p,\bar t)\;  \text{is} \; (\varepsilon_1,n-k)\text{-symmetric}\;\text{for every}\; \bar t \in [0,T)\;\text{and}\;\bar p\in \Gamma_{\bar t}.
\end{equation}
Also, remembering $U(x,\theta,0)\approx    \sqrt{2(n-k)} +\varepsilon\, |x|^2 $, \eqref{eq:eps1_cyl_condition} implies 
\begin{equation}\label{eq:H_lower_bound}
H(\bar p,\bar t) =(1+o(1)) \sqrt{\tfrac{n-k}{2(1-\bar t)}}  \qquad \text{for}\,\,  \, \bar t \leq 1-(2L_1)^{-2}.
\end{equation}
Moreover, by Corollary \ref{cor:sym_estimate_cauchy}, $(\bar p,\bar t)$ is  $(\delta',n-k)$-symmetric, where $\delta' \in [\delta, C\delta]$,
  $\delta \in [\ep^2, \ep]$,  for some $C$ only depending on $n$, for $\bar t \leq 1-(2L_1)^{-2}$ and  $|\pi (\bar p) | \leq \tfrac{1}{3}L_\varepsilon- 4|\log \delta|$. Combining it with   \eqref{eq:H_lower_bound} (the latter  gives  that  $\bar t \leq L_1^2 \, H(\bar p,\bar t)^{-2}$ implies  $\bar t \leq 1-(2L_1)^{-2}$), yields
\begin{equation}\label{eq:delta_sym}
(\bar p,\bar t)\;  \text{is} \; (\delta',n-k)\text{-symmetric}, \;\text{for every}\; \bar t \leq L_1^2 \, H(\bar p,\bar t)^{-2},\;\bar p\in \Gamma_{\bar t},\; \text{and}\; |\pi (\bar p) | \leq \tfrac{1}{3}L_\varepsilon - 4L_1 |\log \delta|.
\end{equation}
\bigskip

\noindent \textbf{Step 1 :} $(\bar p,\bar t)$ is $(\delta',n-k)$-symmetric,  for every $\bar t \in [0,T)$ and  $\bar p\in \Gamma_{\bar t}$ with $|\pi(\bar p)|\leq \frac{1}{3}L_\varepsilon - 10L_1|\log \delta|$. 

\bigskip

To begin with, we find $m$ such that $2^{-m-1}\leq \delta'\leq 2^{-m}$. To show  Step 1, for every $\bar t\geq L_1^2H(\bar p,\bar t)^{-2}$ and $\bar p$ with $|\pi(\bar p)| \leq \frac{1}{3}L_\varepsilon - 4L_1 |\log \delta|-2L_1$, remembering \eqref{eq:eps1_cyl_condition}, \eqref{eq:eps1_sym}, and \eqref{eq:H_lower_bound}, we apply Theorem \ref{thm:sym_improve} so that $(\bar p,\bar t)$ is $(\frac{1}{2}\varepsilon_1,n-k)$-symmetric for $\bar t \in [0,T)$ and  $|\pi(\bar p)| \leq \frac{1}{3}L_\varepsilon- 4L_1|\log \delta| -2L_1$. Thus, if $m \geq 1$ then it is $(\frac{1}{2}\varepsilon_1,n-k)$-symmetric for $\bar t \in [0,T)$ and $|\pi(\bar p)| \leq \frac{1}{3}L+4L_1\log \delta-2L_1$ by \eqref{eq:delta_symmetry_negative_time}.

Next, applying Theorem \ref{thm:sym_improve} again at $(\bar p,\bar t)$ with $\bar t \geq L_1^2H(\bar p,\bar t)^{-2}$ and $|\pi(\bar p)| \leq \frac{1}{3}L+4L_1\log \delta-4L_1$ implies that it is $(\frac{1}{4}\varepsilon_1,n-k)$-symmetric. Hence, if $m\geq 2$ then 
it is $(\frac{1}{4}\varepsilon_1,n-k)$-symmetric when $\bar t \in [0,T)$ and $|\pi(\bar p)| \leq \frac{1}{3}L+4L_1\log \delta-4L_1$  by \eqref{eq:delta_symmetry_negative_time}. Therefore, by mathematical induction, $(\bar p,\bar t)$ is $(2^{-m-1}\varepsilon_1,n-k)$-symmetric when $\bar t \geq  L_1^2H(\bar p,\bar t)^{-2}$ and $|\pi(\bar p)| \leq \frac{1}{3}L+4L_1\log \delta-2(m+1)L_1$. Since $2^{-m}  \geq \delta' \geq \delta$, combining this with \eqref{eq:delta_symmetry_negative_time} completes the proof of this step.

\bigskip\bigskip

\noindent \textbf{Step 2 :} $(\bar p,\bar t)$ is $(2^{-l}\delta',n-k)$-symmetric for every $\bar t \in [0,T)$ and $\bar p\in \Gamma_{\bar t}$ with $H(\bar p,\bar t) \geq (2L_1)^l \sqrt{n-k}$, $|\pi(\bar p)|\leq \frac{1}{3}L+20L_1\log \delta$, and $l\in \mathbb{N}$.

\bigskip
We notice that remembering \eqref{eq:H_lower_bound}  and 
$ \varepsilon^2\leq \delta' \leq 2^{-m} $, where $\delta' \in (\delta, C\delta)$, \textbf{Step 2} already implies the desired result by putting $(2L_1)^\sigma=2$.

\bigskip

To begin with, we assume $H(\bar p,\bar t)(n-k)^{-\frac{1}{2}} \geq  2L_1$ and $|\pi(\bar p)| \leq \frac{1}{3}L_\varepsilon- 10L_1|\log \delta| -  L_1 H(\bar p,\bar t)^{-1} $. Then, a   direct calculation shows that  \eqref{eq:H_lower_bound} implies the lower bound $\bar t \geq L_1^2 \, H(\bar p,\bar t)^{-2}$.  Thanks to \textbf{Step 1} and \eqref{eq:eps1_cyl_condition}, we can apply Theorem \ref{thm:sym_improve} so that $(\bar p,\bar t)$ is $(\frac{1}{2}\delta',n-k)$-symmetric.

Next, we consider $H(\bar p,\bar t)(n-k)^{-\frac{1}{2}} \geq (2L_1)^2$ and $|\pi(\bar p)| \leq \frac{1}{3}L_{\ve}+10L_1\log \delta -  L_1(1+(2L_1)^{-1})$. Then, by \eqref{eq:eps1_cyl_condition} every point $(p',t') \in \hat P(\bar p,\bar t,L_1,L_1^2)$ has $H(p',t')(n-k)^{-\frac{1}{2}} \geq 2L_1$ and $|\pi(p')| \leq \frac{1}{3}L_{\ve}+10L_1\log \delta-L_1$. Thus, it is $(\frac{1}{2}\delta',n-k)$-symmetric. Hence, remembering \eqref{eq:eps1_cyl_condition} we apply Theorem \ref{thm:sym_improve} so that $(\bar p,\bar t)$ is $(\frac{1}{4}\delta',n-k)$-symmetric

By mathematical induction, if $H(\bar p,\bar t)\, (n-k)^{-\frac12} \geq (2L_1)^{l+1}$ and $|\pi(\bar p)| \leq \frac{1}{3}L_{\ve}+10L_1\log\delta-  L_1(1+(2L_1)^{-1}+\cdots+(2L_1)^{-l})$, then $(\bar p,\bar t)$ is $(2^{-l-1}\delta\varepsilon,n-k)$-symmetric. Since $L_1\geq 2$, we have $1+(2L_1)^{-1}+\cdots+(2L_1)^{-l} \leq 2$. This completes the proof. 
\end{proof}
 
\bigskip

The following result describes the profile $U(x,\theta,t)$ and as well the rescaled profile $u^{X,S}(y,\theta,\tau)$.

\begin{corollary}
\label{cor-used-funnel}
Let $T$ denote the first singular time and $\Gamma_t\subset M_t$ denote the connected component with boundary $\partial \Gamma_t=\{(x,U(x,\theta,t):|x|=\frac{1}{3}L_\varepsilon, \theta \in {\mathbb S}^{n-k}\}$. We assume the conditions in Theorem \ref{thm:sym_est}  hold, with the same $\varepsilon,\varepsilon_0,\delta$  as below. Then, given $ x_0  \in B^k_{L/4}(0)$ , there is a point $X=(x_0,x_0',t_0)\in \mathbb{R}^k\times \mathbb{R}^{n+1-k}\times [T,2]$ and a rotation $S$ with the following significance.  The pair $(x_0',t_0,S)$ is a continuous function of $x_0$ satisfying
\begin{equation}
\label{eq-x0'}
|x_0'| \leq C \delta,
\end{equation}
and
\begin{equation}
\label{eq-S-id}
|S-\text{id}| \leq C\delta.
\end{equation}
Moreover, if $M_t$ develops a singularity at $(x_0,\tilde x_0',T)$, then we define $ x_0'(x_0)= \tilde x_0'$, $t_0(x_0)=T$, and a rotation $S$ such that the tangent flow  $(T-t)^{\frac{1}{2}} \Sigma_0$ at $(x_0,  \tilde x_0',T)$ satisfies $S\Sigma_0=\mathbb{R}^k \times {\mathbb S}^{ n-k}_{\sqrt{2(n-k)}} $. 
Moreover, given $x\in B_{L_\varepsilon/4}^k(0)$,  for $t\in [0,T)$, $\Gamma_t^x:=\{(\tilde x,\tilde x')\in \Gamma_t:\tilde x=x\}$ is a continuous one-parameter family of circles, which are $10\varepsilon_1$-close to the unit sphere $S^{n-k}$ after rescaling, and $t_0=t_0(x)$  is defined by  
\begin{equation}\label{def:time_center}
t_0:=T+ (n-k)\lim_{t\to T}(T-t)\bar H(x,t)^{-2},
\end{equation}
where $\bar H(x,t):=\sup_{p\in \Gamma_t^x}H(p,t)$.

In addition, the rescaled flow $\overline{M}^{X,S}_\tau :=e^{\tau/2} S(M_t-(x_0,x_0')) $ with $\tau:=-\log (t_0-t)$ is $\varepsilon_0$-close to $\Sigma :=\mathbb{R}^k\times {\mathbb S}^{n-k}_{\sqrt{2(n-k)}}$ in $C^4(B_{100}(0))$ for $\tau < -\log (t_0-T)$, where $-\log (t_0-T)=+\infty$
 if $t_0=T$. Furthermpre, the profile $u^{X,S}$ satisfies
\begin{align}\label{eq:prof_rot_est}
|u^{X,S}_\theta(y,\theta,\tau)| \leq C\delta e^{-\frac{\sigma}{2}\tau},
\end{align}
for $|y| \leq 10$ and $ \tau  \in [-\log t_0,-\log (t_0-T))$.
\end{corollary}

\begin{proof}
We observe that  by the continuity of $\Gamma_t$, $\Gamma^x_t$ is non-empty for $t<T$ and $x\in B_{L_\varepsilon/3}^k(0)$. Also,  $\bar H(x, t)$ monotone increases in $t$ by Remark \ref{rmk:rough_cylindricality}. Hence, we can define $t_j(x)$ for each $j\in \mathbb{N}$ by
\begin{equation}
\bar H(x,t_j(x))=2^{j-1}.
\end{equation}
If $\lim_{t\to T}\bar H(x,t) \in (2^{j-2},2^{j-1}]$, then  we define $t_j(x)=T$.
We also define
\begin{equation}
\bar r(x):=(n-k)\lim_{t\to T}\bar H(x,t)^{-1}.
\end{equation}
Then, $\bar r(x)$ is a continuous function on $B^k_{L_{ \varepsilon / 3}}(0)$, and in particular $\bar r(x)=0$ if $M_T$ has a singularity at some $(x,x')$. 

\bigskip

Denote by $\pi'$  the projection $\pi' (x,x') =x'$. We next consider $T_1\in (0,T)$ such that $\pi'(\Gamma_t^x)$  is $10\varepsilon_1$-close to the unit sphere $S^{n-k}$, after performing a parabolic rescaling,   for every $t\in [0,T_1)$ and $x\in B^k_{L_\varepsilon/4}(0)$. Then, by Theorem \ref{thm:sym_est}, for each $x\in B_{L_\varepsilon/4}^k(0)$ and $J(x,T_1) \in \mathbb{N}$ satisfying $t_{J-1}(x)< T_1\leq t_J(x)$, there are $p_j,S_j,q_j$ with $1 \leq j \leq J$ and $q_j:=(x,x_j')$ such that $p_j\in \Gamma^x_{t_j}$, $K_\alpha=S_j^{-1}J_\alpha S_j(p-q_j)$ satisfy $\sup_\alpha |K_\alpha|H(p_j,t_j)\leq 5n$ and $\max_{\alpha}|\langle K_\alpha,\nu\rangle| H(p_j,t_j) \leq C\delta \, H(p_j,t_j)^{-\sigma}$, where $\sigma:=\log 2 /\log (2L_1)$ is given in Theorem \ref{thm:sym_est}.  Since $(p_j,t_j)$ is $(C\delta H^{-\sigma},n-k)$-symmetric, the centers $q_j,q_{j-1}$ of the asymptotic cylinders, which has overlapping region. Thus, $q_j$ and $q_{j-1}$  should be aligned on the almost same axis. Thus we have  
\begin{equation}\label{eq:center_align}
|q_j-q_{j-1}|\leq C\delta H(p_j,t_j)^{-\sigma-1},
\end{equation}
and in particular $|q_1-(x,0)|\leq C\delta$. In the same manner, we can choose $S_j$ to satisfy 
\begin{equation}\label{eq:rotation_align}
|S_j-S_{j-1}|\leq C\delta H(p_j,t_j)^{-\sigma},
\end{equation}
and $|S_1-\text{id}|\leq C\delta $. Since $\sum_{j\leq J}C\delta H(p_j,t_j)^{-\sigma}\leq C\delta$, we have $|S_j-\text{id}|\leq C\delta$. Therefore, $M_{t_j}-q_j$ is $\varepsilon_1$-close to the cylinder $\mathbb{R}^k\times S^{n-k}$ after a suitable dilation. Namely,
\begin{equation}\label{eq:graphicality_explanation}
\text{the asymptotic cylinder around} \; q_j\; \text{is locally a graph over}\;\mathbb{R}^k\times \mathbb S^{n-k}. \end{equation}

Therefore, there is $T_2\in (T_1,T)$ such that $\Gamma_t^x$ is $10\varepsilon_1$-close to the unit sphere $S^{n-k}$ after rescaling for every $t\in [0,T_2)$ and $x\in B^k_{L_\varepsilon/4}(0)$. We take the limit $T_i\to \bar T$. If $\bar T<T$ then we repeat this process from $\bar T$ so that we extend the interval. This shows that $\Gamma_t^x$ is close to the sphere $S^{n-k}$ after rescaling for every $t<T$ and $x\in B^k_{L_\varepsilon/4}(0)$. In other words, for each $x_0 \in B_{L_\varepsilon/3}^k(0)$  
\begin{equation}\label{eq:local_shape_rough}
\text{in a neighborhood of}\; \Gamma_t \cap \{x=x_0\}, \; \Gamma_t,\;\text{is close to a single cylinder.}
\end{equation}

\medskip

Now, we consider $J(x)\in\mathbb{N}\cup \{+\infty\}$ given by $t_J=T$, and recall $p_j,S_j,q_j$. Then, $S_j,q_j$ are convergent if $J=+\infty$, namely if $J=\infty$, then $S_{\infty},q_{\infty}$ are well-defined as the limits. Thus, by summing up again, we have  
\begin{equation}\label{eq:series_sum}
|S_J-S_j|\leq C\delta H(p_j,t_j)^{-\sigma}, \qquad |q_J-q_j|\leq C\delta H(p_j,t_j)^{-\sigma-1}
\end{equation}
holds for all $0 \leq j \leq J$, where $S_0:=\text{id}$ and $q_0:=(x,0)$. 

We denote by $\mathcal{S}\subset B^k_{L_\varepsilon/4}(0)$ the set of points $x$ such that $(x,x',T)$ is a singularity for some $x'\in \mathbb{R}^{n-k+1}$.  Then, by remembering \eqref{eq:graphicality_explanation} and \eqref{eq:local_shape_rough}, $(q_J,T)$ is the singularity space-time point, and $S_J$  determines  the unique tangent flow at $(q_J,T)$. Also, for $X=(q_J(x),T)$ and $S=S_J(x)$ with $x\in \mathcal{S}$, the rescaled flow $\overline{M}^{X,S}_\tau$ satisfies \eqref{eq:prof_rot_est}, due to \eqref{eq:series_sum} and the fact that $e^{\frac{\tau_j}{2}}$ with $\tau_j=-\log (T-t_j)$ is comparable to $H(p_j,t_j)$.

\bigskip

Given $x \in B^k_{L_\varepsilon/4}(0)\setminus \mathcal{S}$, we have $\bar r(x)>0$. We observe the $\bar r$ is the radius of the cylinder asymptotic to $M_T-q_J$. Hence, given any $\tilde x \in B_{\bar r}(x)$, by Theorem \ref{thm:sym_est} as in \eqref{eq:center_align} we get
\begin{equation}\label{eq:local_q_est}
|[q_{J(\tilde x)}(\tilde x)-q_{J(x)}(x)]-(\tilde x-x,0)|=|x'_{J(\tilde x)}(\tilde x)-x'_{J(x)}(x)| \leq C\, \delta\bar r(x)^{1+\sigma}.
\end{equation}
Let $t_0(x)$  be the time defined  by \eqref{def:time_center}, which is related to the radius of the asymptotic cylinder around $q(x)$ at the singular time $T$.  Then, the   rescaled flow $\overline{M}_\tau^{X,S_J(\tilde x)}=e^{\tau/2}S_J(M_t-q_J(x))$, with  $X=(q_J(x),t_0(x))$,
has profile $u^{X,S_J(\tilde x)}$ satisfying
\begin{equation}\label{eq:local_S_est}
|u_\theta^{X,S_J}| \leq C\delta e^{-\frac{\sigma}{2}\tau},
\end{equation}
due to \eqref{eq:series_sum} and the fact that $e^{\frac{\tau_j}{2}}$ with $\tau_j=-\log (t_0-t_j)$ is comparable to $H(p_j,t_j)$ again.  We define $S$ and $q$ by
\begin{align}
S(x):=\int_{\mathbb{R}^k}\varphi_{\bar r(x)}(x-\tilde x)S_J(\tilde x)d\tilde x, \qquad  \quad q(x):=\int_{\mathbb{R}^k}\varphi_{\bar r(x)}(x-\tilde x)q_J(\tilde x)d\tilde x
\end{align}
where $\varphi_{\bar r}$ is the standard symmetric molifier compactly supported in $B_{\bar r}$. Then, $S,q$ are continuous in $x\not \in \mathcal{S}$, and $q(x)=(x,x_0'(x))$,  for some $x_0'$,  due to the symmetry of $\varphi_{\bar r}$. 
 This defines the point $x_0'(x)$ as a function of $x, S$. Moreover, by \eqref{eq:local_q_est} and  \eqref{eq:local_S_est}, we have \eqref{eq:prof_rot_est}, since $u_\theta^{X,S}$ is a normalization of $u_\theta^{X,S_J(\tilde{x})}$.
 
\bigskip

We observe next that  $x_0'$ is continuous on $\mathcal{S}$. This is because by \eqref{eq:local_shape_rough}  given $x_0\in \mathcal{S}$, $q_J(x_0)=(x_0,x_0'(x_0))$ is a singularity at $T$ and thus it is the only element of 
$\{(x_0,\tilde x):(x_0,\tilde x)\in \Gamma_t\}$, and $d_H(\Gamma_T^x,q_J(x))\leq 10\bar r(x)$ for all $x\in B^k_{L_\varepsilon/4}(0)$, where $d_H$ denotes the Hausdorff distance. 

 Also, for $x \in B_d(x_0)$ with $x_0\in \mathcal{S}$ and $d\leq 100^{-1}$, we observe the asymptotic cylinders of $M_{T-d^2}$ centered at $q_{J(x_0)}(x_0)$ and $q_{J(x)}(x)$. Then, the cylinders are aligned by Theorem \ref{thm:sym_est}. Thus, considering their overlapping region, we can obtain 
\begin{equation}\label{eq:local_S_diff}
|S_J(x_0)-S_J(x)|\leq C\delta \bar H(x,T-d^2)^{-\sigma} \leq C\delta d^\sigma.
\end{equation}
This implies the continuity of $S$.
\end{proof}
 
 \bigskip
\begin{corollary}\label{cor-rot-symm-good}
Let $\delta,\sigma$  be the same constants as in Theorem \ref{thm:sym_est}. For any $ x_0  \in B^k_{L_\varepsilon/5}(0)$,   let  $(x_0', S)$ be  chosen as in Corollary \ref{cor-used-funnel}, 
and let $t_0$ be any number such that $|t_0 - 1| \leq O(\ep)$. Call  $X=(x_0, t_0)$. Then, assuming that 
\begin{equation}\label{eq:radius_comparasion_condition}
|u^{X,S}(y,\theta,\tau)-\sqrt{2(n-k)}| \leq 1 
\end{equation}
holds in  $ y \in B^k_{r(\tau)}(0)$ for $\tau \leq \tau_1$, where $r(\tau)\in [100, \frac{1}{20}L_\varepsilon e^{\tau/2}]$ and $\tau_1<+\infty$, we have 
\begin{equation}
|u^{X,S}_\theta(y,\theta,\tau)|   \leq C\delta r(\tau) \, e^{-\frac{\sigma}{2}\tau}
\end{equation}
holds in $|y| \leq r(\tau)$ for $\tau \leq \tau_1$.
\end{corollary}

\begin{proof}
In this proof, we always denote the unrescalaed variables by $x$ and the rescaled variables by $y$.

We denote by $\bar t_0(x)$ the time defined in \eqref{def:time_center}. Then, \eqref{eq:radius_comparasion_condition} implies
\begin{equation}
\tfrac{1}{3} u^{(x_0,x_0'(x_0),t'),S(x_0)}(y,\theta, \tau) \leq  u^{(x,x_0'(x),\bar t_0(x)),S(x_0)} (y,\theta, \tau)\leq 3 u^{(x_0,x_0'(x_0),t'),S(x_0)}(y,\theta, \tau).
\end{equation}
Since $u^{(x_0,x_0'(x_0),t'),S(x_0)}(y,\theta, \tau) =a(\tau) \, u^{(x,x_0'(x),\bar t_0(x)),S(x_0)}(y,\theta, \tau)$ for some function $a(\tau)>0$, by Corollary \ref{cor-used-funnel} we have 
\begin{equation}
|\nabla_{\theta}u^{(x_0,x_0'(x_0),t'),S(x_0)}(y,\theta,\tau')|   \leq C\delta  e^{-\frac{\sigma}{2}\tau'},
\end{equation}
in $y\in B^k_{10}(0)$. Similarly, given $\tilde x_0 \in B^k_{r(\tau')e^{-\tau'/2}}(x_0)$, we have
\begin{equation}\label{eq:far_center_sym}
|\nabla_{\theta}u^{(\tilde x_0,x_0'(\tilde x_0),t'),S(\tilde x_0)}(y,\theta,\tau')|   \leq C\delta e^{-\frac{\sigma}{2}\tau'},
\end{equation}
in $y\in B^k_{10}(0)$.
\bigskip

Next, we suppose $|\tilde x_0-z_0|\leq \bar H (\tilde x_0,t')^{-1}$, $\tau':=-\log(t_0-t')\leq \tau_1$, and $\tilde x_0,z_0 \in 
B^k_{r(\tau')e^{-\tau'/2}}(x_0)$. Then, considering the cylindrical region $B_{10/\bar H(\tilde x_0,t')}(\tilde x_0,x'_0(\tilde x_0))$, Theorem \ref{thm:sym_est} and Corollary \ref{cor-used-funnel} imply
\begin{equation}
|S(\tilde x_0)-S(z_0)| + \bar H(\tilde x_0,t') |x_0'(\tilde x_0)-x_0'(z_0)|\leq C\delta \bar H(\tilde x_0,t')^{-\sigma}.
\end{equation}  
Since $\bar H(\tilde x_0,t')$, $\bar H(x_0,t')$, and $e^{-\tau'/2}$ are comparable by \eqref{eq:radius_comparasion_condition}, for every $\tilde x_0 \in B^k_{r(\tau')e^{-\tau'/2}}(x_0)$, as in \eqref{eq:center_align} we can obtain
\begin{equation}
|S(\tilde x_0)-S(x_0)| \leq C \delta \bar H(x_0,t')^{-\sigma} r(\tau'),
\end{equation}  
and
\begin{equation}
 |x_0'(\tilde x_0)-x_0'(x_0)|\leq C \delta \bar H(x_0,t')^{-\sigma-1} r(\tau').
\end{equation}  
We observe
\begin{align}
|\nabla_{\theta} u^{(\tilde x_0,x_0'(x_0),t'),S(x_0)}(y,\theta,\tau')|   \leq &\sup_{|y| \leq 10, \theta \in S^{n-k}}|\nabla_{\theta}u^{(\tilde x_0,x_0'(\tilde x_0),t'),S(\tilde x_0)}(y,\theta,\tau')|\\
&+C|S(\tilde x_0)-S(x_0)|+C|x_0'(\tilde x_0)-x_0'(x_0)|e^{-\frac{1}{2}\tau'}.
\end{align}
Hence, combining the inequalities with \eqref{eq:far_center_sym}, in $y\in B^k_{10}(0)$ we have 
\begin{equation}
|\nabla_{\theta} u^{(\tilde x_0,x_0'(x_0),t'),S(x_0)}(y,\theta,\tau')|   \leq C\delta r(\tau') e^{-\frac{\sigma}{2}\tau'}.
\end{equation}
Namely, for $ y\in B^k_{10}((\tilde x_0-x_0)e^{\tau'/2})$ we have
\begin{equation}
|\nabla_{\theta}u^{(x_0,x_0'(x_0),t'),S(x_0)}(y,\theta,\tau')|   \leq C\delta r(\tau') e^{-\frac{\sigma}{2}\tau'}.
\end{equation}
This completes the proof.  
  \end{proof}
\bigskip
 
\section{Barriers}
\label{sec-barriers}

In this section we construct subsolutions and supersolutions to the rescaled MCF, that  will be later used as lower and upper barriers, respectively, for our solution $u_\bom(y,\theta,\tau)$.  The subsolutions and supersolutions we construct are $O(k) \times O(n-k+1)$ symmetric graphs over the cylinder $\mathbb{R}^k \times 
\mathbb{S}^{n-k}$ and 
therefore they are  defined by  profiles $u^+(y,\theta, \tau) =\bar u^+(r, \tau)$ and  $u^+(y,\theta, \tau)=\bar u^-(r, \tau)$, where $r=|y|$.  
The profiles   $\bar u^+(r,\tau)$ and $\bar u^-(r,\tau)$
are subsolutions and supersolutions respectively of the  rescaled MCF
with that symmetry, that is,  equation 
\begin{equation}
\label{eq:u}
\bar{u}_\tau = \frac{\bar{u}_{rr}}{1+\bar{u}_r^2} - \frac{r}{2} \bar{u}_r + \frac{k-1}{r} \,\bar{u}_r - \frac{n-k}{\bar{u}}  +\frac12 \bar{u}.
\end{equation}

It is simpler to work with   $\bar q(r, \tau) := \bar{u}(r, \tau)^2-2(n-k)$ that satisfies the  equation
\begin{equation}
\label{eq:q}
\bar q_\tau = \bar q_{rr} + \frac{k-1}{r}\, \bar q_r - \frac r2 \bar q_r + \bar q  -  \frac{(\bar q_{rr} + 2)\bar q_r^2}{8(n-k) + 4\bar q + \bar q_r^2}
\end{equation}
that doesn't become singular at $\bar u=0$, as the equation for $\bar u$ does. 
\medskip

\begin{lemma}[Supersolution of \eqref{eq-rescaled-u}]\label{lem-sup}

For any  $p >1$,  $\bar q^+(r,\tau) := \tau^{-1}  \, r^{2p}$ is a supersolution of  equation \eqref{eq:q}  
  for $r \geq \ell$, $\tau \geq \tau_0$,  provided $ \tau_0 \gg 1$ and $\ell  \gg 1$. 

Equivalently, $u^+(y, \theta,\tau) >0$, $(y,\theta) \in \R^k \times {\mathbb S}^{n-k}$, defined by $\bar q^+(|y|, \tau) := u^+(y, \theta, \tau)^2-2(n-k)$
defines the profile of a  $ O(k) \times O(n-k+1)$ symmetric supersolution  of rescaled MCF, 
that is equation \eqref{eq-rescaled-u}, defined for  $|y| \geq \ell$, $\theta \in {\mathbb S}^{n-k}$, for $\tau \geq \tau_0$.

\end{lemma}

\begin{proof}
For simplicity, denote $\bar q^+$ by $f$, that is   $f(r,\tau) := \tau^{-1}   \, r^{2p}$. 
A direct calculation shows that for $r \gg 1$ and $\tau \geq \tau_0 \gg 1$, we have 
$$f_{rr} + \frac{k-1}{r}\, f_r - \frac r2 f_r+  f = \frac{1}{\tau}\Big( (1-p) r^{2p} + 2p\, (2p+k-2)\, r^{2p-2}\Big)  = - \frac {p-1}{\tau} \,  r^{2p} \big (1 + o(1) \big )$$
and $$ f_\tau =    - \frac 1{\tau^2} \, r^{2p}.$$
Therefore,
\begin{equation*}
\label{eq-lin-part}
f_\tau -  \Big ( f_{rr} + \frac{k-1}{r}\, f_r - \frac r2 f_r+  f  \Big ) = \frac{r^{2p}}{\tau}\, \Big(p-1-\frac{1}{\tau} \Big)\, (1+ o(1))>0
\end{equation*}
for $\tau \geq \tau_0 \gg 1$ and $r \gg 1$. Furthermore, since  $2 + f_{rr} = 2 + 2p(2p-1)\, \frac{r^{2p-2}}{\tau} >0$, the whole nonlinear error term $\frac{(f_{rr}+2)f_r^2}{8(n-k)+4f+f_r^2} >0$. 
We conclude that for any  $p > 1$, $\tau \geq \tau_0 \gg 1$   and $r \gg 1$ we have 
\[\mc A(f) := f_{\tau} -  \Big ( f_{rr} + \frac{k-1}{r}\, f_r - \frac r2 f_r+  f  \Big ) + \frac{(f_{rr}+2) f_r^2}{8(n-k) + 4f + f_r^2} > 0\]
that is,  $f(r,\tau)$ is a supersolution of \eqref{eq:q}.

\end{proof}

\begin{lemma}[Subsolution  of \eqref{eq-rescaled-u}]\label{lem-sub}  For any $b\in (0,1]$, the   function $\bar q^-(r, \tau) = -  2(n-k)  \, e^{r-\tau^b}$,  is a  subsolution of equation \eqref{eq:q}  
in the region $\ell \leq r \leq \tau^b$, \,  $\tau \geq \tau_0$,  provided $\tau_0 \gg1 $,  which satisfies the conditions:
\begin{enumerate}
\item [{\em (a)}] $\bar q^-(r(\tau), \tau) = - 2(n-k)$   at $r(\tau) = \tau^b$, for all $\tau \geq 0$, and 
\item  [{\em (b)}] $|\bar q^-(r,\tau)| \leq  2(n-k) e^{-\frac {\tau^b}{2}}$, for all $(r, \tau)$ such that $ \ell \leq  r \leq  \tau^b/2$ and  $\tau \geq \tau_0$.

\end{enumerate}

Equivalently, $u^-(y, \theta,\tau) >0$, $(y,\theta) \in \R^k \times {\mathbb S}^{n-k}$, defined by $\bar q^-(|y|, \tau) := u^-(y, \theta, \tau)^2-2(n-k)$
defines the profile of a  $O(k) \times O(n-k+1)$ symmetric subsolution of   \eqref{eq-rescaled-u})  on $\ell \leq |y| \leq \tau^b$, $\theta \in {\mathbb S}^{n-k}$ and  $\tau \geq \tau_0$,  such that $u^-(y(\tau),\theta,\tau) =0$ at $|y(\tau)|=\tau^b$ (the tip our  subsolution with profile $\bar u^-$). 
\end{lemma}

\begin{proof} We will  first show  that  $f(r,\tau) := - 2(n-k)  \, e^{r-\tau^b}<0$ is a subsolution of equation \eqref{eq:q},  for $r \gg 1$ and $\tau \geq 0$
having  the desired properties. 

 A direct calculation,  using that $f_\tau = -f, f_r=f_{rr} = f$,  gives 
 \begin{equation}\label{eqn-subA}
 \begin{split}
 \mc A(f) &:= f_{\tau} - \big ( f_{rr} + \frac{k-1}{r}\, f_r - \frac r2 f_r+  f  \big ) + \frac{(f_{rr}+2) f_r^2}{8(n-k) + 4f + f_r^2} \\&= \big (\frac r2 - 2 - \frac{k-1}r - b\tau^{b-1}
 \big )\, f + \frac{(2+f) f^2}{ 8(n-k- \frac 12)  + (2+f)^2}.
 \end{split}
 \end{equation}
Since $r \geq \ell \gg 1$ and $f <0$, the term $\big (\frac r2 - 2 - \frac{k-1}r - b\tau^{b-1})\, f $ is negative. Hence, in  the region where $2+f \leq 0 $, we have that $\mc A(f) \leq \big (\frac r2 - 2 - \frac{k-1}r - b\tau^{b-1})\, f <0$.
Now  in the region where $2+f \geq 0$, since $f <0$,  we have $|f| \leq 2$ and $2+f \leq 2$,  hence 
$$\frac{(2+f) f^2}{ 8(n-k- \frac 12) + (2+f)^2} \leq \frac{4 |f|}{ 8(n-k- \frac 12) + (2+f)^2} \leq \frac{4 |f|}{ 8(n-k- \frac 12) }\leq |f| =  - f$$
where in the last inequality we used   that $n-k- \frac 12 > \frac 12$ (since $k \leq  n-1$) and $f <0$. Inserting this in \eqref{eqn-subA} we obtain that 
$\mc A(f) \leq \big (\frac r2 - 3 -  \frac{k-1}r - b\tau^{b-1} )\, f <0,$
 therefore concluding that $\bar q^-(r,\tau):=-2(n-k) \, e^{r-\tau^b}$ is a subsolution of \eqref{eq:q}, provided  $r \gg \ell \gg 1$ and $\tau \geq \tau_0$.  

The conditions (a) and (b) in our statement follow immediately from the definition of $\bar q^-$. Also, it is clear that $\bar q^-(r,\tau)$ defines the profile function $\bar{u}^-(y,\theta,\tau)$ of $O(k)\times O(n-k+1)$ symmetric subsolution of \eqref{eq-rescaled-u} for $|y| \ge \ell$, all the way to the tip at $|y| = \tau^b$ at which $\bar{u}^- = 0$.
\end{proof}

We have the following definition.

\begin{definition}[Super and sub solutions with boundary at $|y| = \ell$]
\label{def-barrier-set}
Let $M^-_{\tau}, M_{\tau}^+$ each be one parameter family of smooth  $O(k)\times O(n-k+1)$ symmetric hypersurfaces with boundary at $|y| = \ell$ defined as follows. We require $M^+_{\tau}$ and $M^-_{\tau}$ to be  graphs over $\Sigma^k := \mathbb{R}^k\times\mathbb{S}^{n-k}$ with profiles 
 $u^+(y,\theta,\tau)$ and  $u^-(y,\theta, \tau)$, as constructed  in Lemma \ref{lem-sup} and \ref{lem-sub} respectively, so that the following holds.  
While $u^+(y,\theta,\tau)$ defines a supersolution for all $|y| \geq \ell$, we will consider 
$M_{\tau}^+$ to be the one parameter family of hypersurfaces defined for  $\ell \leq |y| \leq L\, e^{\tau-\tau_0}$, with boundaries at $|y|=\ell$ and $|y| \leq L\, e^{\tau-\tau_0}$. 
On the other hand, we consider  $M_{\tau}^-$ to be the one parameter family of smooth closed hypersurfaces with boundary  at $r = \ell$, and tip
at $|y|=\tau$  consistent with the precise definition of  $u^-(y,\theta,\tau)$. 
\end{definition}

\begin{lemma}
\label{lemma-barrier-set}
Each   $M^+_{\tau}$ is a complete smooth hypersurface with boundaries at $|y|= \ell$ and $|y| = L\, e^{\tau-\tau_0}$, while each  $M^-_{\tau}$ is a smooth closed hypersurface with boundary at $|y| = \ell$, all the way up to its  tip,   where $\bar u^- (y,\theta, \tau)=0$,   that happens at $|y|=\tau^b$. 
\end{lemma}

\begin{proof}
Our claim regarding  $M_\tau^+$, $\tau \geq \tau_0$  is immediate from its definition.  In the case of our  subsolution $M^-_\tau$, $\tau \geq \tau_0$, it is sufficient  to show that it is smooth at its tip. For that we look at its rotationally symmetric  profile $\bar u^-(r,\tau)$
defined by   $\bar q^-(r, \tau) := \bar u^-(r, \tau)^2-2(n-k)$, where $\bar q^-(r, \tau) = -  2(n-k)  \, e^{r-\tau}$ is as  in Lemma  \ref{lem-sub}. 
Then,  $\bar u^- (r, \tau) = \sqrt{2(n-k) \, (1- e^{r -\tau^b} )}$.  
To show that $M_\tau^-$ is smooth at its tip $r=\tau^b$, we 
 flip the coordinates near the tip, that is denoting $\bar u^- (r, \tau)=:\rho$, we solve $\rho=\sqrt{2(n-k)\, (1- e^{r -\tau^b} )}$ to obtain $r$ as a function of $(\rho, \tau)$. We readily get
$e^{r-\tau^b} = 1 - \frac{\rho^2}{2(n-k)}$ that is 
$$r(\rho,\tau) = \tau^b + \ln \big ( 1 - \frac{\rho^2 }{2(n-k)} \big )$$
concluding that $r(\rho, \tau)$ is a smooth function in $(\rho, \tau)$ near $\rho =0$.  Thus,  $M^-_\tau $, $\tau  \geq \tau_0$ defines a family of  $C^\infty$-smooth closed hypersurfaces 
with  boundary at $r=\ell$.

\end{proof}

\begin{remark}
\label{rem-choice-const}
In section \ref{sec-L2-funnel} we will choose $b = \frac{1}{400}$, and $p = 150$.
\end{remark}

\section{The $L^2$ theory and the funnel}
\label{sec-L2-funnel}

By Theorem \ref{thm-main0} we know the MCF starting with initial data \eqref{eq-Ueps}, for sufficiently small $\ep$, develops  a finite time neckpinch singularity, modeled on the cylinder $\mc{C}_{n,k} = \R^k \times \mathbb{S}^{n-k}$, up to possible rotation. Our goal is to show that this neckpinch singularity is a nondegenerate neckpinch, in the sense of Definition \ref{def-neckpinches}. More precisely, we would like to show that the upward quadratic bending of the initial data \eqref{eq-Ueps}, after performing suitable Type I rescaling of the flow,  propagates all the way up to infinity, hence defining the asymptotics of a non-degenerate  neckpinch singularity at infinity. We can not do it directly, since it appears hard to connect local upward quadratic bending behavior initially, with the asymptotic behavior of the solution at $+\infty$, due to the complicated interactions and dynamics between the stable and unstable modes of the linearized equation around the cylinder along the flow. Therefore, in order to achieve our goal, we  keep  the {\em time $t_0$}, the {\em center} $(x_0, x_0') \in \R^k \times \mathbb{S}^{n-k}$, and  the {\em rotation $S$} of our initial cylinder $\mc{C}_{n,k}$,  as parameters, and by certain shooting type  arguments we  show that we can choose the  best fitting cylinder, that is the right  $(x_0, x_0', t_0)$, and rotation $S$ so that when we parabolically rescale around $(x_0,t_0)$, after applying the suitable rotation, the rescaled flow has precisely the asymptotics  given by the upward quadratic bending that we want. 

\medskip

\subsection{Choosing the class of initial data}\label{sec-in-data}
We will  define bellow our class of rescaled initial data $u_{\bom}(y,\theta, \tau_0)$ depending on $k+1$-parameters $\Omega_0 \in \mathbb{R}$ and $\Omega_1 = (\Omega_1^1, \cdots, \Omega_1^k) \in \mathbb{R}^k$
that correspond more or less to varying $t_0$ and $x_0$ respectively, and we will see that  the other parameters responsible for choosing $x_0'$ and  the rotation $S$
will be chosen as functions of $\Omega_0, \Omega_1$. Let us illustrate below the correspondence between varying $(x_0, t_0)$ and varying the parameters
$\Omega_0, \Omega_1$, in the simple case when $S = Id$ is the identity map, and $x_0' = 0$.

For the moment assume the rotation $S = Id$ is the identity map, and $x_0' = 0$. Then given our unrescaled profile function $U(x, \theta, t)$ and a point $\widebar \cX=(x_0, 0, t_0) \in \R^k \times \mathbb{S}^{n-k}$,   consider the rescaled profile 
\[u^{\widebar \cX}(y,\theta, \tau ) := \frac{U(x,\theta,t)}{\sqrt{t_0-t}}, \qquad y = \frac{x-x_0}{\sqrt{t_0-t}},  \qquad \tau = -\log(t_0-t).\]

Assume that $\widebar \cX$ is close to $(0, 0, 1)$ (this will be shown at the end) so lets write  $t_0=1+\eta$, where $\eta$ is a small parameter. 
Define the parameters $\tilde \Omega_0$, $\bar \Omega_0$  and $\bar \Omega_1$,  in terms of $\eta$ and $x_0 \in \R^k$, so that 
\be\label{eqn-om1} \eta = \ep\,  \tilde{\Omega}_0,  \quad  \bar{\Omega}_0 :=  2k + |x_0|^2 -  \tfrac{\sqrt{n-k}}{\sqrt{2}} \, \tilde{\Omega}_0, \quad  \bar{\Omega}_1 = 2 x_0 \in \mathbb{R}^k.
\ee
Given \eqref{eq-Ueps} and the above rescaling, at the initial rescaled time $\bar \tau_0 :=  -\log (1+\eta)$ (equivalently $t=0$) we can write 
\begin{equation}
\label{eq-lookat1}\begin{split}
u^{\widebar \cX}(y,  \theta, -\log(1+\eta)) &= \sqrt{2(n-k)} + \ep\, \Big(|y|^2 + |x_0|^2 + 2\langle x_0,y\rangle - \tfrac{\sqrt{n-k}}{\sqrt{2}} \, \tilde{\Omega}_0\Big)\, (1+o(1)) \\
&= \sqrt{2(n-k)} + \ep\, \Big((|y|^2 - 2k) + \langle \bar{\Omega}_1, y\rangle + \bar{\Omega}_0\Big) \, (1+o(1)). 
\end{split}
\end{equation}
We will be assuming that $|\bar{\Omega}_0|$ and $|\bar{\Omega}_1|$ are both small, with the goal of showing later that we will be  able to pick such small parameters for which the asymptotics of our solution will hold up to time $+\infty$. More precisely, we will assume  $\big (|\bar \Omega_0|^2 + |\bar \Omega_1|^2 \big )^{\frac 12}  \le \frac{\delta}{1000}$, where $\delta > 0$ is a small number that comes from the $o(\ep)$ in \eqref{eq-Ueps} (more precisely, $|o(\ve)| \le \delta \ve$).
\begin{equation}
\label{eq-lookat}
u^{\widebar \cX}(y, \theta, - \log (1+\eta)) =  \sqrt{2(n-k)} + \ep\,\big(|y|^2 - 2k\big) + \tfrac{\ep\, \delta}{1000}\,\Big(\langle \Omega_1, y\rangle + \Omega_0\Big) \, (1+o(1))
\end{equation}
where by  \eqref{eqn-om1}-\eqref{eq-lookat} we have 
\be\label{eqn-om3}
\tfrac{\delta}{1000} \Omega_0 = \bar \Omega_0  \qquad \mbox{and} \qquad  \tfrac{\delta}{1000} \Omega_1 = 2 x_0
\ee
where $\bar \Omega_0$ is defined in terms of $\eta:=t_0-1$ and $ x_0$ in \eqref{eqn-om1}. Hence, in view of \eqref{eqn-om1} we have 
\be\label{eqn-om4}
t_0(\bom)  = 1 + \tfrac  {\sqrt{2}}{\sqrt{n-k}}  \, \ep \, \big ( 2k + |x_0|^2 - \Omega_0 \, \tfrac {\delta}{1000} \big ) \qquad \mbox{and} \qquad x_0 (\Omega_1) = \tfrac {\delta}{2000} \, \Omega_1.
\ee
We should keep this in mind throughout 
this section: {\em changing $ t_0, x_0$ is related to changing the parameters $\Omega_0, \Omega_1$.}

\begin{remark}
We will  assume below  that $|\Omega_0| + |\Omega_1| \le 1.$ This implies (i) $|x_0| \leq \frac{\delta}{2000}$, and (ii) that $|\bar \Omega_1  |  \leq \frac{\delta}{1000}$.
The last inequality implies that, for all parameters  $|\Omega_0| + |\Omega_1| \le 1$ that we are considering, the corresponding $t_0$ satisfies $t_0  = 1 + \tfrac{2\sqrt{2}k}{\sqrt{n-k}}  \, \ep + o(\ep)$. 

\end{remark} 
\medskip
Motivated by the discussion above, we will now define our class of initial data that includes,  in addition to the parameters $\Omega_0, \Omega_1$ that correspond to choosing $t_0, x_0$, also
other parameters that correspond to choosing the angular directions of the center of the neck and also the rotation. 

Let $\tau_0 >0$ be defined by 
\be\label{eqn-def-tau0}
\ep = \frac{\sqrt{2(n-k)}}{4\tau_0} \qquad \mbox{that is} \qquad \tau_0 :=  \frac{\sqrt{2(n-k)}}{4} \, \ep^{-1}.
\ee To simplify the notation later,  we shift the initial (rescaled) time from $-\log(1+\eta)$ to $\tau_0$, where $\tau_0$  is coming from  \eqref{eqn-def-tau0},  and 
consider perturbations $u_\bom (y,\theta, \tau_0)$ of our initial data starting at $\tau_0$ 
by parameters $\bom = (\Omega_0, \Omega_1^i) \in \mathbb{B}^{k+1}$ and 
$\bom'=(\Omega_2^\alpha, \Omega_3^{i\alpha})$,  $i,j \in \{1,\cdots,k\}$, $\alpha \in \{ 1, \cdots, n-k\}$, 
that are defined as 
\begin{equation}
\label{eq-initial-data}
\begin{split}
u_{\bf \Omega}(y,\theta,\tau_0) &= \sqrt{2(n-k)} + \frac{\sqrt{2(n-k)}}{4\tau_0}\, \left\{ (|y|^2-2k) \,  +10\delta \Big(\frac{\Omega_0}{\|1\|} + \sum_{i=1}^k\Omega_1^i\, \frac{y_i}{\|y_i\|} \Big ) \right .
\\& + \left .  \delta \Big ( \sum_{\alpha =1}^{n-k+1} \Omega_2^{\alpha}\frac{\theta_{\alpha}}{\|\theta_{\alpha}\|} +  
 \sum_{i=1}^k\sum_{\alpha=1}^{n-k+1} \Omega_3^{i\alpha}\, \frac{y_i\, \theta_{\alpha}}{\|y_i\theta_{\alpha}\|} \Big)\,\right \} (1 + o(1))
\end{split}
\end{equation}
for $|y| \le L_{\ve}$,  where $L_\ep$ is the number in the statement of Theorem \ref{thm-main0}. We will  choose $L_\ep$ such  that  $L_{\ve} \ge L$ and $L  \in (\ve^{-\frac{1}{500}}, \ve^{-\frac{1}{300}})$. Note  that we choose $L$ depending on $\ve$ so that Proposition \ref{the:first_singular_time} holds.  

\smallskip

\begin{remark}
\label{rem-uomega}
Here $u_{\bf\Omega}$ can be seen as the graphical function of  the rescaled mean curvature flow  $\widebar{M}_{\tau}^{\cX,S} = e^{\frac{\tau}{2}}\, S (M_t - p_0)$, where $p_0 = (x_0,x_0')\in \mathbb{R}^k\times \mathbb{R}^{n-k+1}$, $\cX = (p_0,t_0)$ and  $\tau = -\log(t_0 - t)$. The initial data of the flow is given by \eqref{eq-initial-data}. Recall that parameters $\Omega_0, \Omega_1$ correspond to the choice of $t_0$ and $x_0$, while the remaining parameters that occur in the description of initial data \eqref{eq-initial-data} correspond to choosing $x_0'$ and the rotation $S$.
\end{remark}

\bigskip

\subsection{The Funnel and main  Proposition}

Let us denote by $\mc{H}$ the space of all functions $f(y,\theta)$ on $\mc{C}_{n,k}$ such that
\be\label{eqn-L2cyl}
\|f\|_{\mc H}^2 = \int_{\mathbb R^k}\int_{\mathbb {\mathbb S}^{n-k}} e^{-\frac{|y|^2}{4}} f^2 \, d\theta dy < \infty.\ee
Let us denote $\mc{D}$ to be another Hilbert space defined by $$\mc{D} = \{f\in \mc{H} \,\,\,|\,\,\, f, |\nabla f|\in \mc{H}\} \qquad \mbox{and} \qquad 
\|f\|_{\mc D} := \|f\|_{\mc H} + \|\nabla f\|_{\mc H}.$$ Denote by $\mc D^*$ the dual space of $\mc{D}$.

We define an operator $\mc L$ on the cylinder $\mc{C}_{n,k}$ by
\[\mc L f = \Delta_{\mc{C}_{n,k}} f - \frac 12\langle z^{\tan}, \nabla^{\mc{C}_{n,k}} f\rangle + f,\]
where $z\in \mc{C}_{n,k}$, and which in coordinates $(y,\theta)$ on $\mc{C}_{n,k}$ takes the form
\begin{equation}
\label{eq-lin-operator}
\mc L f = \Delta_{y} f + \frac{1}{2(n-k)} \Delta_{\theta} f - \frac12 \langle y, \nabla_y f\rangle  + f.
\end{equation}

Let us recall the unstable eigenvalues and eigenvectors of the linearized mean curvature flow operator around the cylinder $\mc{C}_{n,k}$, given by $\mathcal{L}$ above, where we choose coordinates $y\in \mathbb{R}^k$ and $\theta \in {\mathbb S}^{n-k} \subset \mathbb{R}^{n-k+1}$. This will be very important in the $L^2$ theory that we perform below.
\begin{itemize}
\item
$\lambda = 1$, which eigenfunction is constant function $1$, which corresponds to a dilation of the cylinder, or equivalently to a shift in the singular time of the flow.
\item
$\lambda = \frac12$, which eigenfunctions are $\theta_{\alpha}$, for $\alpha=1,\dots n-k+1$ (responsible for the transverse shift, i.e. for infinitesimal rigid translations in the $\mathbb{R}^{n-k+1}$ directions), and $y_j$, for $j=1,\dots k$ (responsible for axial drift which introduces asymmetry in the radius). Note that both types of eigenfunctions $\{\theta_{\alpha}\}_{\alpha=1}^{n-k+1}$ and $\{y_j\}_{j=1}^k$ are responsible for controlling the center of the neck.
\item
$\lambda = 0$, which has all together $K:= \frac{k(k+1)}{2} + k(n-k+1)$ eigenfunctions, that are:  $y_j \theta_{\alpha}$ for $j\in \{1,\dots, k\}$, and $\alpha\in \{1,\dots, n-k+1\}$ (responsible for rigid rotation of the cylinder mixing the $\mathbb{R}^k$ and $\mathbb{R}^{n-k+1}$ coordinates), eigenfunctions $y_i\, y_j$, for $i\neq j$, and $i, j \in \{1,\dots, k\}$, and $y_i^2 -2 $, for $i\in \{1,\dots,k\}$. The latter  eigenfunctions are responsible for quadratic bending along the $\mathbb{R}^k$ spine of the cylinder.
\end{itemize}

Denote the span of eigenfunctions that correspond to positive eigenvalues of $\mc L$ by $\mc H_+$. Eigenfunctions of $\mc L$ that correspond to  eigenvalue $0$  span the space $\mc H_0$. The span of all other eigenfunctions (that correspond to all negative eigenvalues of $\cL$) will be denoted by $\mc H_-$.
Hence,  the Hilbert space $\hilb = L^2(\cC_{n,k}, e^{-|y|^2/4})$ is decomposed into positive, negative and neutral eigenspaces
\[
  \hilb = \hilb_+ \oplus \hilb_0 \oplus \hilb_-.
\]

\smallskip 

In order to restrict  ourselves to the region where we will perform the $L^2$ theory, we next  choose the cut-off function $\bphi\in C^{\infty}(\R)$ with $\bphi(y)=1$ for $|y| \le \frac12$ and $\bphi(y)=0$ for $|y| \ge 1$, and
\begin{equation}
\label{eq-cutoff}
  \phi(y,\theta,\tau) := \bphi\left(\frac{y}{2\lt}\right)
\end{equation}
so that $\phi(y,\theta,\tau)=1$ on $|y| \leq \lt$ and $\phi(y,\theta,\tau)=0$ on $|y| \geq 2\lt$, where $a > 0$ is a small number that will be specified 
in \eqref{eqn-fix}.  Also note that
from our choice of  $\tau_0 \sim  \ep^{-1}$ in \eqref{eqn-def-tau0},  since we may take $a$ to be small, we may assume that our initial surface is a graph over a cylinder in
a much larger region that $|y| \leq  2\tau_0^a$ (since e.g. we can take $L_{\ve} \ge 1000\,  \tau_0^a$). 
Define 
\be\label{eqn-vom}
\bar{v}_{\bf\Omega}(y,\theta,\tau) =  v_{\bf\Omega}(y,\theta,\tau) \phi(y,\theta,\tau).
\ee
 It satisfies 
\begin{equation}
\label{eq-vbar-again}
(\bar{v}_{\bf\Omega})_{\tau} = \mathcal{L} \, \bar{v}_{\bf\Omega} + \mathcal{E}(u,v_{\bf\Omega}, \nabla v_{\bf\Omega}, \nabla^2 v_{\bf\Omega}) \phi + \mathcal{E}_{\phi},
\end{equation}
where $\mc{E}$ is given by \eqref{eqn-error-v} and $\mathcal{E}_{\phi}$, defined by \eqref{eqn-Ephi} is the error term coming from the cut off function and supported in the region  $\frac12\tau^a\le |y| \le \tau^a$.

We let $P_\pm$, $P_0$ be the corresponding orthogonal projections on $\mc{H}_{\pm}$ and $\mc{H}_0$, respectively, and we define
\[
  (v_{\bf\Omega})_\pm(\cdot, \tau) = P_\pm\bigl[\bv_{\bf\Omega}(\cdot, \tau)\bigr],\quad
  (v_{\bf\Omega})_0(\cdot, \tau) = P_0\bigl[\bv_{\Omega}(\cdot, \tau)\bigr]
\]
so that
\begin{equation}
  \label{eq-proj}
  \bar{v}_{\bom} (y,\theta,\tau) = (v_{\bf\Omega})_+(y,\theta,\tau) + (v_{\bf\Omega})_0(y,\theta,\tau)  + (v_{\bf\Omega})_-(y,\theta,\tau).
\end{equation}

Let $\widebar{M}^{\bf\Omega}_{\tau}$ denote the RMCF solution starting at $\widebar{M}^{\bf\Omega}_{\tau_0}$ with profile $u_{\bom}$, defined in \eqref{eq-initial-data}
and recall that by  Corollary \ref{cor-used-funnel} and Corollary \ref{cor-rot-symm-good}, we can view $\bom=(\Omega_0, \Omega_1^i) \in B^{k+1}$ as free parameters while 
$\Omega_2^i, \Omega^{i\alpha}_3$ depend continuously on $\Omega_1:=(\Omega_1^i)$. 

Note that
\begin{equation} (\bar v_{\bf \Omega})_0(y,\theta,\tau) = \sum_{i=1}^k\,\alpha_{ii} (\tau) \, (y_i^2-2) + 2\sum_{1 \le i < j  \le k} \alpha_{ij} (\tau)  \, y_i\, y_j + \sum_{i=1}^k\sum_{\alpha=1}^{n-k+1}\beta_{i\alpha}(\tau)\,  y_i\, \theta_{\alpha},\end{equation} 
where 
\begin{equation}
\label{eq-modes-names}
\alpha_{ii}  = \frac{\langle \bar{v}_{{\bf \Omega}}, y_i^2-2\rangle}{\|y_i^2-2\|}, \qquad \alpha_{ij} = \frac{\langle \bar{v}_{\bf \Omega}, y_i\, y_j\rangle}{\|y_i\, y_j\|}\,\, (i\neq j), \qquad  \beta_{i\alpha} \,\, = \frac{\langle\bar{v}_{\bf\Omega}, y_i\theta_{\alpha}\rangle}{\|\theta_{\alpha}\, y_i\|} .
\end{equation}
Define
\begin{equation}
  \label{eq-proj-norms}
  (V_{\bf\Omega})_0(\tau) :=  \|(v_{\bf\Omega})_0(\cdot, \tau)\|,\qquad
  (V_{\bf\Omega})_-(\tau) := \|(v_{\bf\Omega})_-(\cdot, \tau)\|, \qquad 
  (V_{\bf\Omega})_+(\tau) := \|(v_{\bf\Omega})_+(\cdot, \tau)\|. 
\end{equation}
All projections $a_{ij}, \beta_{i\alpha}$ above also depend on $\bom$, however we have  dropped  the  subscript $\bom$ to simplify the notation. 

Note that 
\begin{equation}
\label{eq-plustau00}
(V_{\bf\Omega})_0 = \left (\sum_{i=1}^k \alpha_{ii}^2\, \, \|y_i^2-2\|^2  + 4\sum_{1 \le i < j  \le k}  \alpha_{ij}^2  \, \| y_i\, y_j \|^2 + \sum_{i=1}^k\sum_{\alpha+1}^{n-k+1}\,   \beta_{i\alpha}^2 \, \|\theta_{\alpha}y_i\|^2 \right )^{\frac 12} \end{equation}
Given \eqref{eq-initial-data} and \eqref{eq-plustau00} we find that
\begin{equation}
\label{eq-sum-V}
(V_{\bf\Omega})_+(\tau_0) = \frac{10 \delta\, \sqrt{2(n-k)}}{4\,\tau_0}\, \left ( \Omega_0^2 + \sum_{i=1}^k(\Omega_1^i)^2 + \sum_{\alpha=1}^{n-k+1} (\Omega_2^{\alpha})^2 \right )^{\frac 12} \, \big ( 1+  o(1) \big), 
\end{equation}
and
\begin{equation}
\label{eq-plustau02}
(V_{\bf\Omega})_0(\tau_0) =   \Big (\sum_{i=1}^k \alpha_{ii}(\tau_0)\,  \|y_i^2-2\|^2 \Big )^{\frac 12} (1+ o(1)), \quad \mbox{where} \quad \alpha_{ii}(\tau_0) =
\frac{\sqrt{2(n-k)}}{4\tau_0}. 
\end{equation}

\medskip

In the view of  Corollary  \ref{cor-rot-symm-good}  we expect the most important part in $(V_\bom)_+(\tau)$ as defined in \eqref{eq-sum-V}, will be coming from the projections of $\bar{v}$ onto the space spanned by eigenfunctions $\langle 1, y_1, \dots, y_k\rangle$. We separate this main part by defining  $({\hat V_\bom})_+(\tau)$ to be the norm of the projection of $\bar{v}$ onto the space spanned by eigenfunctions $\langle 1, y_1, \dots, y_k\rangle$. Note that we have
\begin{equation}
\label{eq-hatV100}(\hat{V}_{\bf\Omega})_+(\tau_0) =  \frac{10 \delta\, \sqrt{2(n-k)}}{4\,\tau_0}\, \Big (\Omega_0^2 + \sum_{i=1}^k (\Omega_1^i)^2 \Big )^{\frac 12} \, (1 + o(1))
\end{equation}
and that 
\be\label{eqn-vplus}
(V_{\bf\Omega})_+(\tau) = (\hat{V}_{\bf\Omega})_+(\tau) + (V^\theta_{\bf\Omega})_+(\tau)
\ee
where $(V^\theta_\bom)_+$ comes  from the projections of $\bar{v}$ onto the space spanned by eigenfunctions $\{\theta_{\alpha}\}_{\alpha=1}^{n-k+1}$.

\smallskip

This motivates us to  define the {\em  funnel }  $\mc F_{\tau}$ as follows. Consider first the set  $\mc{S}_{\tau}$, the set of all hypersurfaces in $\mathbb{R}^{n+1}$, which can be  locally written as graphs over a cylinder $\mc{C}_{n,k}$ with graphical profile function $u(y,\theta,\tau)$, for an interval in $y$ that contains $|y|\le 2\, \tau^{a}$. Consider the RMCF solution $\widebar{M}_{\tau} \in \mc S_{\tau}$, and let $V_0(\tau)$, $ V_-(\tau)$, $V_+(\tau)$ and $\hat{V}_+(\tau)$  be the corresponding norms of projections defined as above for $\widebar{M}^{\bf\Omega}_{\tau}$. Define 
\begin{equation}
\label{eq-funnel}
\mc F_{\tau} := \{\mbox{MCF} \,\,\, \widebar{M}\,\,\,|\,\,\, \widebar{M}_s \in \mc S_{s}\,\,\, \mbox{and}\,\,\, \hat V_+(s)  < \frac{\delta \sqrt{2(n-k)}}{s}, \,\,\, \mbox{for all}\,\,\, s\in [\tau_0,\tau]\}.
\end{equation}
for the same $\delta > 0$ as above. The number $\delta > 0$ will be a small fixed number. Note that \eqref{eq-initial-data} yields

\begin{equation}
\label{eq-small-ini}
(V_{\bf\Omega})_+(\tau_0) + (V_{\bf\Omega})_-(\tau_0) +   \sum_{i=1}^k\sum_{\alpha+1}^{n-k+1} (\beta_{i\alpha}(\tau_0)^2 \|\theta_{\alpha}y_i\|^2 \le A({\bf\Omega})\,\frac{\delta}{\tau_0},
\end{equation}
for a small $\delta > 0$, where $A({\bf\Omega}) \to 0$ as $|{\bf\Omega}| \to 0$. All above yields 
\begin{equation}
\label{eq-motivation-funnel}
(\hat{V}_{{\bf\Omega}})_+(\tau_0)  \le \frac{\delta\, \sqrt{2(n-k)}}{\tau_0} \,\,\, \implies \,\,\, | \bom |:= \Big(\Omega_0^2 + \sum_{i=1}^k (\Omega_1^i)^2 \Big)^{\frac12}  < \frac12
\end{equation} 
This will be used in Section \ref{sec-shooting}. 

Recall also that $\bar{v}_\bom(y,\theta,\tau) = \big(u_{{\bf\Omega}}(y,\theta,\tau) - \sqrt{2(n-k)}\big) \phi(y,\theta,\tau)$, for a cut off function $\phi$ as in previous sections. Also recall that by our choice of class of initial data in  \eqref{eq-initial-data}  and  $\psi_i(y) = y_i^2-2$, we have 
\begin{equation}
\label{eq-plustau0}
v_{\bf\Omega}(y, \theta, \tau_0) =  \frac{\sqrt{2(n-k)}}{4\tau_0} ( |y|^2 - 2k)  (1+ o(1)).
\end{equation}

 The goal of the rest of this  section is to show the following proposition.

 \begin{proposition}
\label{prop-first-step}
Let $\tau_0$ a large number, defined in terms of $\ve$ by \eqref{eqn-def-tau0}, where $\ve$ is as in \eqref{eq-Ueps}.
Then,   as long as $\widebar{M}^{\bf\Omega} \in \mc F_{\tau}$,  for $\tau\in [\tau_0,\tau_1]$, for some $\tau_1 >\tau_0$, the truncated solution $\bar v_\bom$, as defined above,
 satisfies  
\[\bar{v}_\bom(y,\theta,\tau) =  \frac{\sqrt{2(n-k)}}{4  \tau}  \, (|y|^2 - 2k) + o(\tau^{-1}),\]
in the $L^2$ sense,  for all $\tau\in [\tau_0,\tau_1]$, where $o(\tau^{-1})$ is understood in a way that it is bounded by $\eta_0 \, \tau^{-1} $, with $\eta_0 = \eta_0(\delta) \to 0$ as $\delta\to 0$,  where   $\delta$  is the same small constant as in \eqref{eq-funnel}. 
Note that we can make $\tau_0$ to be sufficiently large, by taking $\ep$ sufficiently small.
 \end{proposition}

The {\em strategy}  to show the above proposition is as follows.  For our RMCF solution $u_{\bf\Omega}(y,\theta,\tau)$ with   \eqref{eq-initial-data} that satisfies  $\widebar{M}^{\bf\Omega}_{\tau} \in \mc F_{\tau}$ for all $\tau\in [\tau_0,\tau_1)$, we make the a'priori assumption  that   
\begin{equation}
\label{eq-desired-C2}
\big  | v_{\bf\Omega}(y,\theta,\tau) \big |_{C^3} = \Big |u_{\bf\Omega}(y,\theta,\tau) - \sqrt{2(n-k)} \Big |_{C^3} \le  \frac{1}{\tau^{1-400\,  a}}, 
\end{equation}
for $|y| \le 2\,\tau^a$, $\theta \in [0,2\pi]$, and all $\tau\in [\tau_0,\tau_1)$. Note that the a'priori $C^3$ estimate above is inspired by the subsolutions and the supersolutions constructed in section \ref{sec-barriers}. In other words, given that we prove $q_{\bf\Omega}^-(y,\theta,\tau)$ and $q_{\bf\Omega}^+(y,\theta,\tau)$ are lower and upper barriers for $u_{\bf\Omega}^2(y,\theta,\tau) - 2(n-k)$, respectively, then it would follow that in the case $p = 150$, we would have that 
\[\big  | v_{\bf\Omega}(y,\theta,\tau) \big |_{C^3} = \Big |u_{\bf\Omega}(y,\theta,\tau) - \sqrt{2(n-k)} \Big |_{C^3} \le  \frac{C_0}{\tau^{1-350\, a}} \le \frac{1}{\tau^{1-400\, a}},\]
for $|y| \le 2\,\tau^a$, $\theta \in [0,2\pi]$, and all $\tau\in [\tau_0,\tau_1)$. 

We will specify the exact value of $a=\frac 1{1500}$ at the conclusion of the proof of the Proposition \ref{prop-first-step} in subsection \ref{sec-conclusion}  as follows: we require that 
 $1-400 a > 0$ in  \eqref{eq-desired-C2} (implying the  smallness of the $C^3$ norm of $v_\Omega$),  and also that  $3-1200a >2$ in \eqref{eqn-err5}.

In other words, we first assume we have all derivative estimates that we need, on a large set whose size depends on time as above.  Then we employ suitable  $L^2$ arguments for the projections of $\bar{v}_{\bf\Omega}(y,\theta,\tau) =  v_{\bf\Omega}(y,\theta,\tau) \phi(y,\theta,\tau)$ onto stable, neutral and unstable modes of the linearized mean curvature operator around the cylinder to be able to say that  the neutral mode continues dominating, which then yields to the right behavior of our solution on the  set $|y| \le 2 \ell$, where $\ell$ can be any  large but fixed number less that $L_\ep/2$.   In these $L^2$ arguments, we need the smallness of $v_{\bf\Omega}$  as stated in \eqref{eq-desired-C2} in order  to be able to control the errors coming form cut off functions and estimate them by exponentially small terms, which become negligible in our analysis. However, the right behavior of our solution that comes  from the $L^2$ theory,  results to an improvement of \eqref{eq-desired-C2} on the set $|y| \leq 2\ell$.

In the second step,   assuming we have the right behavior of our solution at the boundary $|y| = \ell$, using the supersolutions and subsolutions constructed in section \ref{sec-barriers} as upper and lower barriers, respectively, for $\ell \leq |y| \leq 2\,\tau^a$, we obtain  sharper $L^{\infty}$ estimates on the solution and its derivatives than in our a'priori assumption \eqref{eq-desired-C2} that was used in the previous paragraph.   The fact we can get sharper $L^{\infty}$ estimates allows us to close the bootstrap  argument.  This yields  the supersolutions and subsolutions we construct in section \ref{sec-barriers}  are indeed the barriers, as we have wanted as long as our solution is in the funnel.  More precisely,  we obtain exact  asymptotics  on compact sets, and we prove that sharp $L^{\infty}$ estimates hold for as long as our solution  stays in the funnel that is defined in  \eqref{eq-funnel}.

\bigskip

\subsection{Preliminary $L^2$-estimates} 

Recall that in the region $|y| \leq 2 \tau^{a} $ where $\bar v_\bom \neq 0$ we have the auxiliary bound \eqref{eq-desired-C2}. Using this bound we will show the following result.

\begin{lemma}
  \label{lem-error}
 There exist  uniform constants $B$ and $c_0$ so that
 as long as $\widebar{M}^{\bf\Omega} \in \mc F_{\tau}$, and  \eqref{eq-desired-C2} holds for $\tau\in [\tau_0,\tau_1]$ we have
 \begin{align*}
    \frac{d}{d\tau}(V_{\bf\Omega})_+ &\ge c_0 \,(V_{\bf\Omega})_- - \frac{B}{\tau^{2-800a}}\\
  \,\,  \Big  |\frac{d}{d\tau} (V_{\bf\Omega})_0 \Big | &\le \frac{B}{\tau^{2-800a}} \\
    \frac{d}{dt}(V_{\bf\Omega})_- &\le -c_0 (V_{\bf\Omega})_- + \frac{B}{\tau^{2-800a}}\\
    \frac{d}{d\tau}(\hat{V}_{\bf\Omega})_+ &\ge c_0\, (\hat{V}_{\bf\Omega})_+ - \frac{B}{\tau^{2-800a}}\
  \end{align*}
  where $(V_{\bf\Omega})_+(\tau), (V_{\bf\Omega})_0(\tau)$ and $(V_{\bf\Omega})_-(\tau)$ are defined by \eqref{eq-proj-norms}. 
  \end{lemma}

\begin{proof}
For the sake of simplicity of notation, we  drop the subscript ${\bf\Omega}$ in the proof of the lemma and we simply  denote  $v_\bom, \bv_\bom$ by $v, \bv$, respectively.  Recall that
\[(\bar{v})_{\tau} = \mc L \bar{v}+ \mc E_1,\]
where $\mc E_1$ is the sum of the two error terms in equation \eqref{eq-vbar-again}. More precisely we have
\[\mc E_1 = \mc{E}(u, v, \nabla v, \nabla^2 v)\,\varphi + \mc{E}_{\varphi},\]
where $\mc{E}(u, v, \nabla v, \nabla^2 v)$ and $\mc{E}_{\vp}$ are defined in \eqref{eqn-error-v} and \eqref{eqn-Ephi}, respectively. Note that by \eqref{eq-desired-C2} 
and the at least quadratic behavior of $\cE$, we have the pointwise, and hence the  $L^2$ estimates  
\be \label{eqn-E1234}
|\mc{E}(u, v, \nabla v, \nabla^2 v)|_{L^{\infty}(|y| \le 2\, \tau^a)} \le \frac{C_0}{\tau^{2-800a}}, \qquad \|\mc{E}(u, v, \nabla v, \nabla^2 v)\|_{\mc H} \le \frac{C_0}{\tau^{2-800a}}.\ee
Here $C_0$ denotes a uniform constant, even though it may change from line to line, but in a uniform way. Also note that we can choose $a > 0$ small  enough so that we have
$2-800a >1$ which means that error $\mc{E}(u, v, \nabla v, \nabla^2 v)$ is of a lower order than the
expected main-order behavior of $\bv$ which is $\sim \frac{|y|^2-2}{\tau}$.

On the other hand, using \eqref{eq-desired-C2} and the fact that the error $\mc{E}_{\vp}$ is supported on the set $\tau^a \le |y| \le 2\tau^a$, we get that $\|\mc{E}_{\vp}\|_{\mc H} \le C_0 \, e^{-\frac{\tau^{2a}}{4}}$, implying
\begin{equation}
\label{eq-error-diff-ineq}
\|\mc E_1\|_{\mc H} \le \frac{C_0}{\tau^{2-800a}}.
\end{equation}
 To derive differential inequalities for $V_+$, $V_-$, $V_0$, and $\hat{V}_+$ observe that
\[\frac{d}{d\tau} \bar{v}_+ = \mc{L} \bar{v}_+ + P_+\mc E_1\]
implying that
\[\frac{d}{d\tau} V_+^2 = 2\langle \bar{v}_+, \mc{L}\bar{v}_+ + P_+ \mc E_1\rangle\]
and hence
\[\frac{d}{d\tau} V_+^2 \ge c_0 V_+^2 - 2 \|\mc E_1\|_{\mc H}\, V_+.\]
This implies
\[\frac{d}{d\tau} V_+ \ge c_0 V_+ - \|\mc E_1\|_{\mc H},\]
which together with \eqref{eq-error-diff-ineq} yield the desired differential inequality for $V_+$. Other differential inequalities for $V_-$, $V_0$ and $\hat{V}_+$ are proved similarly.

\end{proof}

Next we formulate a direct consequence of Corollary \ref{cor-used-funnel} in the language of our parameter ${\bf\Omega}$ that we will use it below. 

\bigskip

 \begin{corollary}[Consequence of Corollary \ref{cor-rot-symm-good}]
\label{cor-directly-used}
Let $u_{\bf\Omega}$ be the rescaled graphical function of $\widebar{M}_{\tau}^{\tilde{X}_0,S} = e^{\frac{\tau}{2}}\, S(M_t - p_0)$, where $\tau = -\log(t_0 - t)$, as in Remark \ref{rem-uomega}.
For any $\Omega_1:= (\Omega_1^1, \cdots, \Omega_1^k) \in \mathbb{B}^k$, there exist ${\bf\Omega}':= (\Omega_2^\alpha, \Omega_3^{i\alpha})$ such that  $| \bom' | \leq C_0$, depending continuously on $\Omega_1$,  with the following significance. 
Let   $u_{\bom}(\cdot, \tau_0)$ denote   a RMCF initial data  given   in terms the parameters ${\bf \Omega} = (\Omega_0, \Omega_1)$ and $ {\bf\Omega}'=(\Omega_2^\alpha, \Omega_3^{i\alpha})$ by   \eqref{eq-initial-data}. 

 Then, if we have that $|t_0 - 1|< \ep$, 
$|\nabla_\theta u_\bom(y,\theta, \tau_0) | \leq C\,   \delta \, \tau_0^{-1}  $ on $|y| \leq 40  \, \tau_0^{a}$,  
 and that   \[|u_{\bf\Omega}(y,\theta,\tau)-\sqrt{2(n-k)}| \leq 1,\] 
 holds for all $|y| \le r(\tau)$, $\theta\in \mathbb{S}^{n-k}$, for $\tau \in [\tau_0,  \tau_1]$, where $r(\tau)\in [100, 2  \tau_0^{a} e^{\frac{\tau-\tau_0}2}]$ and  $\tau_1 < \infty$, we have
\begin{equation}
\label{eq-vtheta}
|\nabla_{\theta} u_{\bom}(y,\theta,\tau)| \le  C_0\,  \frac{\delta}{\tau_0}  \,  r(\tau) \, e^{-\frac{\sigma}2 \, ( \tau-\tau_0)},
\end{equation}
for  all  $|y| \le r(\tau)$,  $\tau \in [ \tau_0, \tau_1]$ and for a uniform constant $C_0$.  It follows that 
\begin{equation}\label{eq-vtheta2}
|\nabla_{\theta} u_{\bom}(y,\theta,\tau)| \le  C_0\, \frac{\delta}{\tau_0} \,  (1+ |y|)  \,  e^{-\frac{\sigma}{2} \, (\tau - \tau_0)}
\end{equation}
for all $|y| \leq 2 \,  \tau^a$,  $\theta\in \mathbb{S}^{n-k}$, for $\tau \in [\tau_0,  \tau_1]$,  for another  uniform constant $C_0$. 
\end{corollary}

\begin{proof} The proof of \eqref{eq-vtheta} is a direct consequence of Corollary \ref{cor-used-funnel} and Corollary \ref{cor-rot-symm-good}
by taking $L_\ep = 20 \, K_0 \tau_0^{a} $,  where we recall that $\ep = \frac{\sqrt{2(n-k)}}{4\tau_0}$. 
To prove \eqref{eq-vtheta2}, we divide the interval $[0, 2 \tau^{a}]$ into 
the regions  $A_m := \{y\in \mathbb{R}^k\,\, | \,\, 100\, (m-1) \le |y| \le 100\,m\}$, $1\leq m \leq l$, where 
$l := \Big[\frac{2\tau^{a}}{100}\Big]$,  and also call  $A_{m+1} := \{y\in \mathbb{R}^k\,\,\,|\,\,\, 100\, l  \le |y| \le 2 |\tau|^{a}\}$.

For each region $A_m$,  $1 \le m \le l$, we apply \eqref{eq-vtheta} with $r(\tau):= 100\,m$ to get 
\[ \big |\nabla_{\theta} u_\bom (y,\theta, \tau)\big  | \le C \, \frac{\delta}{\tau_0}\,100\, m\, e^{-\frac{\sigma}{2}(\tau - \tau_0)} \le C_0\, \frac{\delta}{\tau_0}\, (1+ |y|)\, 
e^{-\frac{\sigma}{2}(\tau - \tau_0)} \]
for all $y \in A_m$ and a uniform constant $C_0$. For the interval  $A_{m+1}$ we use again \eqref{eq-vtheta} with $r(\tau) = 2|\tau|^{a}$ to get
\[\big |\nabla_{\theta} u_\bom (y,\theta, \tau) \big  | \le C \, \frac{\delta}{\tau_0}\ |\tau|^{a}  e^{-\frac{\sigma}{2}(\tau - \tau_0)} \leq C_0\, \frac{\delta}{\tau_0}\, (1+ |y|)\, 
e^{-\frac{\sigma}{2}(\tau - \tau_0)} \]
for all $y \in A_{m+1}$ and a uniform constant $C_0$. Combining the above yields to \eqref{eq-vtheta2}.

\end{proof}

Recall that we defined $\tau_0$ so that $\ep = \frac{\sqrt{2(n-k)}}{4\tau_0}$. We next show the following important estimates  which  tell us 
that as long as we are in the funnel we can control all projections that come from changing the center  
$x_0'= U(x_0,\theta_0, t_0) \, \theta_0$  and the rotation $S$.  

\begin{lemma}
\label{lemma-control-theta}
There exists a uniform constant $C_0$ so that as long as  $\widebar{M}^{\bf\Omega}$  is in $\mc F_{\tau}$, for  $\tau\in [\tau_0,\tau_1]$, where   $\tau_0$ is  uniformly big (which is equivalent to  choosing $\ep$ small in \eqref{eq-Ueps}) and $\tau_1 >\tau_0$,  and as long as  \eqref{eq-desired-C2} holds,  we have that 
\be\label{eqn-VB2}  
|(V^\theta_{\bf\Omega})_+(\tau)|  +  \sum_{i=1}^k\sum_{\alpha=1}^{n-k+1} \big | \beta_{i\alpha}(\tau) \big |  \leq   C_0 \, \frac{\delta}{\tau} 
\ee
for all $\tau \in [\tau_0, \tau_1]$ and for a uniform constant $C_0$.  Recall that  $\beta_{i\alpha}(\tau)$ are as in \eqref{eq-proj-norms},  and $(V^\theta_\bom)_+$ is  the norm of the projection of $\bar{v}$ onto the eigenspace spanned by the $\{\theta_{\alpha}\}_{\alpha=1}^{n-k+1}$. 
\end{lemma}

\begin{proof}
For simplicity of notation we will drop all subscripts ${\bf\Omega}$ below. Let's first bound the first term in \eqref{eqn-VB2} and observe that it 
suffices  to show 
\be\label{eqn-bn5}  |\langle \bar{v}, \theta_i\rangle| \le C_0\, \frac{\delta}{\tau},\ee
where $\{\theta_i\}$, for $i\in \{1,\dots, n-k+1\}$ are the eigenvectors of our linearized operator $\mc{L}$ with the eigenvalue $\lambda = \frac12$, which  means that $\mc{L} \theta_i = \frac12\, \theta_i$. Since $\theta_i$ is purely spherical factor, all its $y_j$ derivatives vanish, thus by \eqref{eq-lin-operator} we have
$\Delta_{\theta} \theta_i = -(n-k)\,\theta_i.$
Combining this, integration by parts and Corollary \ref{cor-directly-used} yield
\be\begin{split}\label{eqn-bn11}
\big |\langle \bar{v}, \theta_i\rangle \big | &= \Big | \int_{\mathbb{R}^k} \int_{\mathbb{S}^{n-k}} \bar{v} \, \theta_i \,e^{-\frac{|y|^2}{4}}\, d\sigma dy
\Big | = \Big| \frac{1}{n-k} \int_{\mathbb{R}^k} \int_{\mathbb{S}^{n-k}} \bar{v} \,  \Delta_\theta  \theta_i  \,e^{-\frac{|y|^2}{4}}\, d\sigma dy \Big| \\&= \Big| \frac{1}{n-k}\,\int_{\mathbb{R}^k} \int_{\mathbb{S}^{n-k}} \nabla_{\theta} \bar{v} \, \nabla_{\theta} \theta_i  \,e^{-\frac{|y|^2}{4}}\, d\sigma dy \Big| \\
&\le C_0\, \int_{|y| \le 2\, \tau^a}\int_{\mathbb{S}^{n-k}} |\nabla_{\theta}\bar{v}|\, e^{-\frac{|y|^2}{4}}\, d\sigma dy, \\
&\le C_0\,\frac{\delta}{\tau_0}\,e^{-\frac{\sigma}{2}\, (\tau-\tau_0)}\, \int_{|y| \le 2\, \tau^{a}}\int_{\mathbb{S}^{n-k}} (1 + |y|)\,e^{-\frac{|y|^2}{4}} \, d\sigma\, dy \\
&\le C_0\, \frac{\delta}{\tau},
\end{split}
\ee
if $\tau_0$ is sufficiently big. 

To bound  the second term in \eqref{eqn-VB2},  we argue  similarly as  above. Integration by parts and Corollary \ref{cor-directly-used}   yield
\be
\begin{split}\label{eqn-bn2}
|\langle \bar{v},  y_i \theta_{\alpha}\rangle| &= \Big|\frac{1}{n-k} \int_{\mathbb{R}^k}  \int_{\mathbb{S}^{n-k}} \Delta_{\theta} \, \theta_{\alpha} \, \bar{v} \, y_i \, e^{-\frac{|y|^2}{4}}\, d\sigma dy\Big|\\
&= \Big|\frac{1}{n-k} \int_{\mathbb{R}^k} \int_{\mathbb{S}^{n-k}} \nabla_{\theta}\theta_{\alpha} \, \nabla_{\theta}\bar{v} \, y_i \, e^{-\frac{|y|^2}{4}}\, d\sigma dy\Big|\\
&\leq C\,  \frac{1}{n-k} \int_{|y| \le 2\, \tau^a} \int_{\mathbb{S}^{n-k}} | \nabla_{\theta}\bar{v} | \, |y|  \, e^{-\frac{|y|^2}{4}}\, d\sigma dy\\
&\le C_0\,\frac{\delta}{\tau_0}\,e^{-\frac{\sigma}{2}\, (\tau-\tau_0)}\, \int_{|y| \le 2\, \tau^{a}}\int_{\mathbb{S}^{n-k}} |y|\, (1 + |y|)\,e^{-\frac{|y|^2}{4}} \, d\sigma\, dy \\
&\le C_0\, \frac{\delta}{\tau},
\end{split}
\ee

Finally combining \eqref{eqn-bn11} and \eqref{eqn-bn2}  yields to the desired bounds in \eqref{eqn-VB2}.  

\end{proof} 

\subsection{Main-order quadratic behavior of $\bar v_{\bom}$ under the assumption \eqref{eq-desired-C2}}

In this subsection we will show the following crucial proposition that describes the main-order behavior of $\bar v_\bom$ in the $L^2$- sense and also 
in the $C^\infty$ sense, on compact subsets, provided auxiliary  assumption 
 \ref{eq-desired-C2} holds. This result together with the barriers constructed in section \ref{sec-barriers} 
will enable us to improve our auxiliary  assumption 
 \ref{eq-desired-C2}, thus
closing  our bootstrapping argument,  leading to the proof of the main Proposition  \ref{eq-desired-C2}. 

\sk

\noindent{\bf Notation.} 
Throughout this subsection, we will use {\em $C_0$ to denote a uniform constant that may change from line to line but in a uniform way}
and the notation $\cO( \delta \tau^{-1}) $ to denote quantities that such that  $\cO( \delta \tau^{-1}) \leq C_0 \delta \, \tau^{-1}$, for such a uniform $C_0$.  
The small constant $\delta$ here is the same as in \eqref{eq-funnel}.

\begin{proposition}
\label{prop-control}
There exists a uniform constant $C_0$ so that,  as  long as our solution $\widebar{M}^{\bf\Omega}$  is in $\mc F_{\tau_1}$,  for some $\tau_1 > \tau_0$,
where $\tau_0$ is given by \eqref{eqn-def-tau0},     and as long as  \eqref{eq-desired-C2} holds,  we have 
\begin{equation}
\label{eq-first-ineq}
(V_{{\bf\Omega}})_+ (\tau) + (V_{{\bf\Omega}})_- (\tau) \leq  C_0\,\delta  \, \Big( \sum_{i=1}^k\, \alpha_{ii}(\tau)^2\, \|y_i^2-2\|^2 \,\Big)^{\frac 12}
\end{equation}
for all $\tau \in [\tau_0, \tau_1]$. 
Furthermore, 
\be\label{eqn-g99}\bar v_\bom(y,\theta,\tau)=  \frac{ \sqrt{2(n-k)}}{ 4 \tau}\, ( |y|^2-2 )  + \cO( \delta \tau^{-1})\ee
holds in the $L^2$-sense,  for all $\tau\in [\tau_0,\tau_1]$.  The small constant $\delta$ here is the same as in \eqref{eq-funnel}. 
   \end{proposition}

\smallskip
Recall the definition of the symmetric spectral coefficient matrix 
$A = (\alpha_{ij})$  defined  in \eqref{eq-modes-names}, that is  
\be\alpha_{ij} = \frac{\langle\psi_{ij}, \bar{v}\rangle}{\|\psi_{ij}\|^2}, \qquad \mbox{where} \,\, \, \psi_{ii}(y) := y_i^2-2, \,\,\,  \psi_{ij}(y):=y_iy_i, \,  i \neq j.
\ee
The Proposition  above is a consequence of  the next lemma.

\begin{lemma} 
\label{lemma-eq-A}
As  long as our solution $\widebar{M}^{\bf\Omega}$  is in $\mc F_{\tau_1}$,  for some $\tau_1 > \tau_0$,
where $\tau_0$ is given by \eqref{eqn-def-tau0},   and as long as the bound  \eqref{eq-desired-C2}  and 
\be\label{eqn-AAA}
\frac{c_{n,k}}{\tau} \le |A(\tau)| \le \frac{C_{n,k}}{\tau}, \qquad \tau \in [\tau_0, \tau_1]
\ee  for  $c_{n,k}>0$ and $C_{n,k}< \infty, $ then 
  the matrix $A(\tau)$ satisfies 
\be\label{eqn-A}
\dot { A } = - \frac {\sqrt{2(n-k)}}{4} \, A^2 + \cO(\delta \, \tau^{-2})
\ee
for all $\tau \in [\tau_0, \tau_1]$.
\end{lemma} 

\begin{proof}

We begin by showing that $V_+(\tau) + V_-(\tau) \le C_0\,\delta \, V_0(\tau)$, for a uniform constant $C_0$. 
To this end,  observe first that  \eqref{eq-plustau00} and \eqref{eqn-VB2}, the neutral mode  projection  norm $V_0(\tau)$ satisfies 
\begin{equation}
\label{eq-V0-asymp}
V_0(\tau) = \Big( \sum_{i=1}^k\, \alpha_{ii}(\tau)^2\, \|\psi_{ii}\|^2 + 4\sum_{1 \le i < j  \le k}\, \alpha_{ij}(\tau)^2\,\|\psi_{ij}\|^2\,
+ O \big ( \delta \tau^{-1}  \big ) \Big)^{\frac 12}\, (1 + o(1))
\end{equation}
for all $\tau\in [\tau_0,\tau_1]$, where $|o(1)| \le \delta$, up to a uniform constant. 
 Furthermore, by \eqref{eq-initial-data} we have 
\[V_+(\tau_0) + V_-(\tau_0)  \le \,\frac{C_0\,\delta}{\tau_0} \le C_0\, \delta\, V_0(\tau_0),\]
where $C_0$ is a uniform constant. 

The fact that our solution belongs to $\mc F_{\tau_1}$, implies that 
\begin{equation}
\label{eq-this-holds}
\hat{V}_+(\tau) \le \frac{\delta}{\tau}, \,\,\, \mbox{for all} \,\,\,  \tau\in [\tau_0,\tau_1]. 
\end{equation}  
On the other hand, by Lemma \ref{lem-error} we get
$$\frac{d}{d\tau} V_-(\tau) \leq - c_0  \, V_-(\tau)  + \frac{B}{\tau^{2-800a}}.$$
Integrating in $\tau$ and using $V_-(\tau_0) \leq \frac{C_0\,\delta} {\tau_0}$ we get for all $\tau \in [\tau_0,\tau_1]$,
\begin{equation}
\label{eq-eta-small}
 V_-(\tau) \leq \frac {C_0\,\delta}{\tau_0 } e^{ - c_0 (\tau-\tau_0) }  + B\, e^{-c_0\, \tau} \,  \int_{\tau_0}^\tau   \,  \frac{1}{s^{2-800a}} \, e^{c_0\, s} \, ds.
 \end{equation}
Note that since the function $\frac \tau{\tau_0} e^{ - c_0 (\tau-\tau_0) } $ is decreasing for $\tau \geq \tau_0 > c_0^{-1}$, we have 
$\frac{\tau}{\tau_0} e^{ - c_0(\tau-\tau_0) } \leq 1$, that is $\frac 1{\tau_0} e^{ - c_0\, (\tau-\tau_0) } \leq \frac 1\tau$. 
Furthermore,  using integration by parts,
\[
 e^{-c_0 \tau}  \int_{\tau_0}^\tau  \frac 1{s^{2-800a}} \, e^{ c_0 s} \,  ds  \leq   \frac{C_0} {\tau^{2-800a}}  + C_0\,  e^{- c_0\,\tau}  \int_{\tau_0}^\tau   \frac 1{s^{3-800a}} \,  e^{c_0\,s} \, ds.  \]
Note that for $\tau \le 3\tau_0$ we have
\[2e^{-c_0\tau}\,\int_{\tau_0}^{\tau}\frac{e^{c_0\, s}}{s^{3-800a}}\, ds \le \frac{C_0}{\tau^{2-800a}}\]
while, for  $\tau \ge 3\tau_0$ we have
\[\begin{split}
e^{-c_0\, \tau}\,\int_{\tau_0}^{\tau}\frac{e^{c_0\, s}}{s^{2-800a}}\, ds &=  e^{-c_0\, \tau}\,\left(\int_{\tau_0}^{\tau/2}\frac{e^{c_0\, s}}{s^{3-800a}}\, ds + \int_{\tau/2}^{\tau}\frac{e^{c_0\, s}}{s^{3 - 800a}}\, ds\right)\\
&\le \frac{e^{-\frac{c_0\, \tau}{2}}}{\tau_0^{2-800a}} + \frac{C_0}{\tau^{2-800a}}\le \frac{C_0}{\tau^{2-800a}}
\end{split}
\]
for $\tau_0$ sufficiently big.
All these,  \eqref{eq-eta-small}, and $2-800a >1$, imply
\begin{equation}
\label{eq-V--beh}
V_-(\tau) \le \frac{C_0\, \delta}{\tau} + \frac{C_0}{\tau^{2 - 800a}}  \le \frac{C_0\, \delta}{\tau}
\end{equation}
for $\tau\in [\tau_0,\tau_1]$, which can be achieved by taking $\tau_0$ sufficiently big (since $\delta > 0$ is a small but fixed constant), and $C_0$ is a uniform constant that may change from line to line but in a uniform way. On the other hand, by \eqref{eqn-VB2} and \eqref{eq-this-holds}, we get
 \be\label{eqn-V-plus} V_+(\tau) \leq \hat{V}_+(\tau) + V^\theta_\Omega(\tau) \leq  \frac{\delta}{\tau} +  C_0\, \frac{\delta}{\tau}  \le C_0\, \frac{\delta}{\tau} \ee
for all  $\tau\in [\tau_0,\tau_1]$, as long as $\tau_0$ is sufficiently big. 

Combining \eqref{eq-V--beh} and \eqref{eqn-V-plus}  with  \eqref{eq-V0-asymp},  \eqref{eqn-AAA}, and \eqref{eq-eta-small} yield
\begin{equation}
\label{eq-desired-1}
V_+(\tau) + V_-(\tau) \le C_0\,\delta \,  V_0(\tau), \qquad \tau\in [\tau_0,\tau_1],
\end{equation}
proving our first assertion. 

\smallskip

Next, recall that  $\bv$ satisfies \eqref{eq-vbar-again}, that is $\bv_\tau = \cL \bv + \cE_1$, where  for simplicity we called $\cE_1:= \cE + \cE_\varphi $. 
Hence, the projections of  $\bv $ on the eigenspace generated by $\psi_{ij}$ satisfiy
\be\label{eqn-E25} \frac{\partial }{\partial \tau} \langle \bv, \psi_{ij} \rangle = \langle \cE_1, \psi_{ij} \rangle = - \tfrac 12 \, \langle \bv^2, \psi_{ij} \rangle
+ \langle  \cE_1 + \tfrac{\bv^2}2, \psi_{ij}\rangle=  - \tfrac 12 \, \langle \bv^2, \psi_{ij} \rangle
+ \langle  \hat \cE_1, \psi_{ij} \rangle
\ee
where $\hat \cE_1:= \cE_1 + \frac{\bv^2}2$. The reason we separated $-\frac{\bv^2}2$ from $\cE_1$ is that this is the main-order term,  in the sense that 
$\hat \cE_1$ is now of lower order than $-\frac{\bv^2}2$. 

We will first estimate the second term in \eqref{eqn-E25}.
Having subtracted from $\cE_1$ its main-order term, we now see, using the bounds \eqref{eq-desired-C2},  
\eqref{eq-vtheta2}  and $\ep \sim \tau_0^{-1}$ that we have the following pointwise estimate
\be\label{eqn-err5}
 |\hat  \cE_1 | \leq C_0 \, \Big ( \frac 1{\tau^{3-1200a}} + \frac{\delta^2}{\tau_0^2}  e^{-\sigma\,(\tau-\tau_0)}\, (1 + |y|)^2 \Big )  
\le C_0\, \frac{\delta^2}{\tau^2}\, (1 + |y|)^2 \ee
holding for all  $\tau \in  [\tau_0, \tau_1]$ provided  $\tau_0 \gg 1$, and $3 - 1200 a > 2$. In the last bound above we also used that $\tau^{-2}_0  e^{-\frac{\sigma}{4}(\tau-\tau_0)}
\leq \tau^{-2}$ which follows from the fact that $\tau^{-2}  e^{-\frac{\sigma}{4}(\tau-\tau_0)}$ is decreasing for $\tau_0 \gg 1$.  Finally, using \eqref{eqn-err5} we have
\begin{equation}
\label{eq-hatE1}
\big|\langle\hat{\mc E}_1, \psi_{ij}\rangle\big| \le C_0\, \int_{|y|\le 2\tau^{a}} \int_{\mathbb{S}^{n-k}} |\hat{\mc E}_1|\, (1 + |y|^2)\, e^{-\frac{|y|^2}{4}} \, d\sigma\, dy \le 
C_0\, \frac{\delta^2}{\tau^2}.
\end{equation} 

\smallskip 
Now let us look at $ - \frac 12 \, \langle \bv^2, \psi_{ij} \rangle$ which is the main-order term in \eqref{eqn-E25}.  Writing for simplicity 
$\bv = \bv_0 + \hat v$,  where $\hat v:= \bv_+ + \bv_-$,  we have 
\[ \langle \bv^2, \psi_{ij} \rangle = \langle \bv^2_0, \psi_{ij} \rangle +  \langle \hat v^2, \psi_{ij} \rangle + 2  \langle \hat v\, \bv_0, \psi_{ij} \rangle.
  \]

 \smallskip
The lemma follows from the next claim. 
\begin{claim} \label{claim-remainder}For $\tau \in [\tau_0, \tau_1]$, with $\tau_0$ sufficiently large, we have  
\be\label{eqn-hatv}  2 \, | \langle \hat v\, \bv_0, \psi_{ij} \rangle| + | \langle \hat v^2, \psi_{ij} \rangle |   = \cO\Big(\frac \delta{\tau^2}\Big). 
\ee

\end{claim} 

\begin{proof}[Proof of Claim] We will show that  $ \langle \hat v\, \bv_0, \psi_{ij} \rangle = \cO(\delta \, \tau^{-2})$, as the other term can be treated similarly. 
Recalling  the definition of our $L^2$   norm $\| \cdot \|_{\mc H}$ on the cylinder   $\cC_{n,k} := \R^k \times \mathbb{S}^{n-k}$ given in \eqref{eqn-L2cyl},  and using that $\| \psi_{ij} \| \leq c_n \, |y|^2$ we have 
\be\label{eqn-hardy5}    |  \langle \hat v\, \bv_0, \psi_{ij} \rangle \big | \leq  \big\|  \hat v\,  \sqrt{|\psi_{ij}|} \big\|_{\mc H} \cdot   \big\|  \bv_0 \, \sqrt{|\psi_{ij}|} \big\|_{\mc H} 
 \leq  C\,  \|  \bv_0 \, |y|   \|_{\mc H} \cdot \|  \hat v \, |y|  \|_{\mc H}. \ee
We will  next see   that 
\be\label{eqn-good15}
 \|  \hat v \, |y|  \|_{\mc H}  =  \Big ( \int_{\cC_{n,k}} |y|^2 \, \hat v^2 \, e^{-\frac{|y|^2}4} \, d\theta dy \Big )^{\frac 12} =  \cO\Big(\frac{\delta}{\tau}\Big)
 \ee
 To this end,  we will combine the  weighted Hardy inequality that can be shown similarly to Lemma 4.12 in \cite{ADS0}:
\be\label{eqn-hardy}  \int_{\cC_{n,k}} |y|^2 \hat v^2 \, e^{-\frac{|y|^2}4} \, d\theta dy \leq 16 \int_{\cC_{n,k}} |\nabla  \hat v|^2 \, e^{-\frac{|y|^2}4} \, d\theta dy + 4k \int_{\cC_{n,k}}  \hat v^2 \, e^{-\frac{|y|^2}4} \, d\theta dy\ee
and  Lemma 5.14  in \cite{ADS0},  which in our setting gives us that $\hat v = \bv- v_0$ satisfies 
\be\label{eqn-good4}  \|  \sqrt{2-\cL} \, \hat v  \|_{\mc{H}} = o\Big(\| \bv_0 \|_{\mc{H}}\Big) = o\Big(\frac 1{\tau}\Big) \leq \frac{\delta}{\tau}, \qquad \mbox{for all} \,\, \tau \in [\tau_0+1, \tau_1] \ee
where in the last inequality above we used  $\| \bv_0 \|_{\mc{H}} \sim \tau^{-1}$ which follows from \eqref{eqn-AAA}. Here we recall that for  any function $f$   defined on $\cC_{n,k}$ we have 
\[ \|  \sqrt{2-\cL} \, f  \|_{\mc{H}} =\Big (  \int_{\cC_{n,k} } \big ( |\nabla_y f|^2 + \tfrac 1{2(n-k)} \,
  |\nabla_\theta  f|^2 + f^2 \big ) \, e^{-\frac {|y|^2}4}  dy d\theta \Big )^{\frac 12}.\]
 If $\tau_1 < \tau_0+1$, then we skip \eqref{eqn-good4}, since we will also provide an estimate
holding on $[\tau_0, \tau_0+1]$. 
Note the reason we can  apply Lemma 5.14 in \cite{ADS0}  to obtain   \eqref{eqn-good4}  is that we have  ${\hat v}_\tau = \cL \hat v + P_+( \cE_1) +   P_-( \cE_1)$, 
 where $\mc E_1 = \mc E + \mc E_{\phi}$,  with $\cE$ satisfying the bounds in \eqref{eqn-E1234}, $\|\mc E_{\phi}\|_{\mc H} \le C_0\, e^{-\frac{\tau^{2a}}{4}}  \leq \frac{C_0}{\tau_0^{1-800a}} \, \| \bv \|_{\mc{H}} $, for a uniform $C_0$,  
  and    $\| \bv \| \geq c_0 \tau^{-1}$ (which follows by \eqref{eqn-AAA}),  so that we can bound 
\[  \| P_+ \cE\|_{\mc{H}} +  \|  P_- \cE\|_{\mc{H}} \leq 2  \|  \cE\|_{\mc{H}} \leq  \frac{C_0}{\tau^{2-800a}} 
= \frac{C_0}{\tau^{1-800a}} \cdot  \frac 1  \tau \leq \frac{C_0}{\tau_0^{1-800a}} \, \| \bv \|_{\mc{H}} \]
for a uniform constant $C_0$. 

Now  \eqref{eqn-hardy} and \eqref{eqn-good4} readily yield \eqref{eqn-good15} holds for $ \tau \in [\tau_0+1, \tau_1]$. It remains to show  \eqref{eqn-good15}  hold for $ \tau \in [\tau_0, \tau_0+1]$ (in the case that $\tau_0+1 < \tau_1$).  For this, we use again that $\hat v_\tau =  \cL \hat v + P_+( \cE_1) +   P_-( \cE_1)$, and Cauchy-Scwarz inequality  to get 
\bee
\begin{split} 
 \frac 12 \, \frac d{d\tau} & \int_{\cC_{n,k}} |y|^2 \, \hat v^2 \, e^{-\frac{|y|^2}4} \, d\theta dy =   \int_{\cC_{n,k}} |y|^2 \, \hat v \, \hat v_\tau  \, e^{-\frac{|y|^2}4} \, d\theta dy=
\int_{\cC_{n,k}} |y|^2 \,  \Big ( \hat v \, \cL v  +  \hat v  \, \big( P_+( \cE_1) +   P_-( \cE_1) \big ) \Big)   \, e^{-\frac{|y|^2}4} \, d\theta dy\\
&\leq  - \int_{\cC_{n,k}} |y|^2 \, |\nabla \hat v|^2    \, e^{-\frac{|y|^2}4} \, d\theta dy  + \int_{\cC_{n,k}} |y|^2 \,  \hat v^2    \, e^{-\frac{|y|^2}4} \, d\theta dy - 2  \int_{\cC_{n,k}}  y \cdot \nabla \hat v \,\,  \hat v    \, e^{-\frac{|y|^2}4} \, d\theta dy \\
&+  \int_{\cC_{n,k}} |y|^2\,  |\hat v |  | \cE_1|  \, e^{-\frac{|y|^2}4} \, d\theta dy\\
&\leq  C_0 \, \int_{\cC_{n,k}} |y|^2\,  \hat v^2   \, e^{-\frac{|y|^2}4} \, d\theta dy + C_0\,\, \int_{\cC_{n,k}} |y|^2\,  \cE_1^2   \, e^{-\frac{|y|^2}4} \, d\theta dy \\
&\leq 2 \, \int_{\cC_{n,k}} |y|^2\,  \hat v^2   \, e^{-\frac{|y|^2}4} \, d\theta dy
+ \frac{C_0}{\tau^{4-1600a}} 
\end{split} 
\eee 
where in the last line we used \eqref{eq-desired-C2} and that the error is at least quadratic in $v$ and its derivatives.
It follows that $\beta(\tau):= \frac d{d\tau}  \int_{\cC_{n,k}} |y|^2 \, \hat v^2 \, e^{-\frac{|y|^2}4} \, d\theta dy$ satisfies the differential inequality $\beta'(\tau) \leq 4 \beta(\tau) + \frac {C_0}{\tau^{4-1600a}}$. Integrating this on $[\tau_0, \tau]$,  for any $\tau \in [\tau_0, \tau_0+1]$,  we get  $\beta(\tau) \leq e^{4(\tau-\tau_0)} \big ( \beta(\tau_0) + \frac{C_0}{4\tau^{4-1600a}_0} \big ) $. Note that $\beta(\tau_0) \le C_0\, \frac{\delta^2}{\tau_0^2}$, and that we can choose $a > 0$ small so that still $4 - 1600a > 2$. All these imply that \eqref{eqn-good15} also holds for $\tau \in [\tau_0, \tau_0+1]$,  and therefore it holds for all $\tau \in [\tau_0, \tau_1]$. 

Next we claim that 
\begin{equation}
\label{eq-barv0}
\|\bar{v}_0\, |y| \|_{\mc H} = O(\tau^{-1}).
\end{equation}
Note that this follows by \eqref{eqn-hardy5} applied to $\bar{v}_0$, the fact that the neutral space is finite dimensional and hence
\[c_0\, (\|\bar{v}_0\| + \|\nabla \bar{v}_0\|) \le \|\bar{v}_0\|,\]
for a uniform constant $c_0$, together with \eqref{eqn-AAA} and Lemma \ref{lemma-control-theta}.
Desired estimate \eqref{eqn-hatv} for the first term, for $ \tau \in [\tau_0, \tau_1]$, immediately follows by \eqref{eqn-good15} and  \eqref{eq-barv0}. The second term in \eqref{eqn-hatv} can be estimated in a similar manner. This concludes the proof of  the claim. 

\end{proof} 

Having shown the Claim above, the lemma readily follows. More precisely, we have
\[\bar{v}_0 = \sum_{i,j=1}^k \alpha_{ij} \psi_{ij},\]
and hence
\[\langle\bar{v}_0^2, \psi_{ij}\rangle = \sum_{l,m=1}^k\sum_{p,q=1}^k \alpha_{lm} \alpha_{pq}\, \langle \psi_{lm}\psi_{pq},\psi_{ij}\rangle\]
Since $\bar{v} = \bar{v}_0 + \hat{v}$, by \eqref{eq-hatE1} and by Claim \ref{claim-remainder}, keeping in mind that $\|\bar{v}_0\|_{\mc H} \sim \tau^{-1}$ by \eqref{eqn-AAA} and Lemma \ref{lemma-control-theta}, the same computation as in Proposition 3.1 in \cite{DZ} concludes the proof of the Lemma.
\end{proof}

We are now in position to give the proof of Proposition \ref{prop-control}.

\begin{proof}[Proof of Proposition \ref{prop-control}]
Both  assertions of the Proposition   follow by the following claim. 

\begin{claim} Under the assumptions of Proposition \ref{prop-control}, the coefficients $a_{ij}$  of the matrix $A(\tau)$  satisfy the bounds 
\begin{equation}
\label{eqn-aij-bounds}
\frac{\sqrt{2(n-k)}}{6\tau}  < \alpha_{ii}(\tau) <  \frac{\sqrt{2(n-k)}}{3\tau}  \qquad  \mbox{and} \qquad |\alpha_{ij}(\tau)| <  \frac{\delta^{\frac34}}{\tau}, \,\, i \neq j \qquad
\end{equation}
for all $\tau\in [\tau_0,\tau_1]$.  
\end{claim}

\begin{proof}[Proof of Claim] We   simply write $\delta$ for any uniform multiple of it.
First note  that \eqref{eqn-aij-bounds}  holds at time $\tau_0$,  and in fact we have sharper bounds than those in \eqref{eqn-aij-bounds} at $\tau_0$.  Let $\tau_2 \le \tau_1$ be the maximal time up to which both estimates in \eqref{eqn-aij-bounds} hold, implying in particular 
\begin{equation}
\label{eq-V0-beh}
\frac{c_{n,k}}{\tau^2} \le |A(\tau)|^2 \le \frac{C_{n,k}}{\tau^2}
\end{equation}
for some $C_{n,k} > c_{n,k} >0$ and for all $\tau \in [\tau_0, \tau_2]$.
 If $\tau_2 = \tau_1$, we are done. Thus assume $\tau_2 < \tau_1$.

Hence, on the interval $[\tau_0,\tau_2]$ where \eqref{eqn-aij-bounds} holds, by Lemma \ref{lemma-eq-A}  we have 
\be\label{eqn-g31} \frac{d}{d\tau}\, a_{ij}  = - \frac{4}{\sqrt{2(n-k)}}\, \Big  ( a_{ij} ( a_{ii} + a_{jj} ) + \sum_{\ell \neq i,j}  a_{i\ell} a_{\ell j} \Big ) +  \cO(\delta \tau^{-2}),
\qquad  (i \neq j)\ee
where \bee\sum_{\ell \neq i,j}  a_{i\ell} a_{\ell j}  = \cO( \delta^{\frac 32} \tau^{-2} )\eee
and 
\be\label{eqn-g32} \frac{d}{d\tau}\, a_{ii}  = - \frac{4}{\sqrt{2(n-k)}} \, \Big ( a_{ii}^2 + \sum_{\ell \neq i}  a_{i\ell} a_{\ell i} \Big ) +  \cO(\delta \tau^{-2}). \ee

\smallskip 
Let us first analyze \eqref{eqn-g31} for $\tau \in [\tau_0,\tau_2]$ where  both bounds in  \eqref{eqn-aij-bounds} hold. We readily see that 
\bee\label{eqn-g333} \frac{d}{d\tau}\, a_{ij}  = - \frac{4}{\sqrt{2(n-k)}} \, a_{ij} ( a_{ii} + a_{jj} ) +  \cO( \delta  \tau^{-2} ), 
\qquad \mbox{for}\,\, i \neq j.\eee
Combining   \eqref{eqn-g333} with the  first bound in  \eqref{eqn-aij-bounds}, we   obtain that each $a_{ij}$, with $i\neq j$, satisfies 
\be
- \frac 8{3\tau} a_{ij} -   \frac{C_0\delta}{\tau^2}  \le \frac{d}{d\tau}\, a_{ij}  \le - \frac 4{3\tau}  \,  a_{ij} + \frac {C_0 \,  \delta}{\tau^2}
\ee
for all $\tau \in [\tau_0, \tau_2]$, for some uniform constant $C_0 >0$. Integrating  these differential inequalities gives us
\be\label{eqn-g35}  a_{ij}(\tau_0)\,  \big (\frac{\tau_0}{\tau} \big )^{\frac 83}  -  \frac{\tau_2  \, \delta } {3\tau}  \leq a_{ij}(\tau) \leq  a_{ij}(\tau_0)\,  \big (\frac{\tau_0}{\tau} \big )^{\frac 43}  + \frac {C_0 \, \delta }{\tau}.\ee
Now we recall that $|a_{ij}(\tau_0)| \leq C_0 \, \delta \, \tau_0^{-1}$ for a  uniform constant  still denoted by $C_0$,   and that \eqref{eqn-g35} holds for all $\tau \in [\tau_0,\tau_2]$.  Inserting this bound  in
\eqref{eqn-g35} we obtain that 
\be\label{eqn-g36}   |a_{ij}(\tau) | \leq \frac{C_0 \, \delta}{\tau}, 
\ee
for some new uniform constant $C_0 >0$.  Note that this is a stronger inequality than the second bound that we  assumed  in \eqref{eqn-aij-bounds}.

Now, let us turn our attention to \eqref{eqn-g32} and insert the bound \eqref{eqn-g36} that we have just shown to get 
\be\label{eqn-g34} \frac{d}{d\tau}\, a_{ii}  = - \frac{4}{\sqrt{2(n-k)}} \,  a_{ii}^2  +  \cO( \delta \tau^{-2} )\ee

To simplify the notation, we set ${\bar a}_{ii} = \frac 4{\sqrt{2(n-k)}} \, a_{ii}$, which satisfies the differential inequalities  
\be\label{eqn-g76}
  - {\bar a}_{ii}^2  -  C_0 \,  \delta \tau^{-2} \leq \frac{d}{d\tau} {\bar a}_{ii}  \leq  - {\bar a}_{ii}^2  +  C_0 \,  \delta \tau^{-2} \ee
(for a  uniform constant $C_0$),  
and we also have the conditions 
\be\label{eqn-g77} 
\frac{2}{3\tau} < {\bar a}_{ii}(\tau) < \frac 4{3\tau}
\ee
which directly follow by \eqref{eqn-aij-bounds}. 

A simple calculation using \eqref{eqn-g76} and \eqref{eqn-g77} shows that ${\tilde a}_{ii} := {\bar a}_{ii} - 4 \delta C_0\tau^{-1} $ satisfies the differential inequality 
$$\frac{d}{d\tau} {\tilde a}_{ii}  \leq  - {\tilde a}_{ii}^2$$which can be integrated in $\tau$ to show that for all $\tau \in [\tau_0, \tau_2]$ we have 
${\tilde a}_{ii}(\tau) \leq  \frac{{\tilde a}_{ii}(\tau_0)}{1+ {\tilde a}_{ii}(\tau_0) (\tau-\tau_0)}.$
Since ${\tilde a_{ii}}(\tau_0) =  \tau_0^{-1} + \cO(\delta \tau_0^{-1})$ we conclude  that
${\tilde a_{ii}}(\tau) \leq  \tau^{-1}  + \cO(\delta \tau^{-1})$, leading to the upper bound  ${\bar a}_{ii}(\tau)  \leq  \tau^{-1}  + C_0 \, \delta \tau^{-1}$, 
for some uniform $C_0 >0$. 
Similarly we have  lower bound ${\bar a}_{ii}(\tau)  \geq \tau^{-1} -  C_0 \, \delta \tau^{-1}$. Expressing  those bounds in terms of $a_{ii}(\tau) = \frac {\sqrt{2(n-k)}}{4}\,  {\bar a}_{ii}$,  
we  finally obtain
\be\label{eqn-g88}
\frac {\sqrt{2(n-k)}}{4\tau} -  \frac{C_0 \delta} { \tau}  \leq a_{ii}(\tau) \leq  \frac {\sqrt{2(n-k)}}{4\tau}+ \frac{C_0 \delta}{\tau}
\ee
holding for any  $\tau \in [\tau_0, \tau_2]$,  where  $C_0$ is a new uniform constant. 

Now let us finish the proof of the claim. At the maximal time $\tau_2$ (if it exists) that \eqref{eqn-aij-bounds} hold, we have that both \eqref{eqn-g36} 
and \eqref{eqn-g88} hold. Both bounds are stronger than these assumed in \eqref{eqn-aij-bounds}, therefore showing that such a time $\tau_2$ doesn't
exist, and proving  the claim.

\end{proof} 

We will now {\em finish the proof of Proposition \ref{prop-control}}. The first assertion in  \eqref{eq-first-ineq} simply follows by combining \eqref{eq-V0-asymp} and   \eqref{eq-desired-1} with the bounds  \eqref{eqn-aij-bounds}. 
The exact asymptotic behaviour  in our second assertion \eqref{eqn-g99}  follows from \eqref{eq-first-ineq} and our improved  bounds  \eqref{eqn-g36} and \eqref{eqn-g88}. This concludes the proof. 

\end{proof}

\bigskip

\subsection{The conclusion of the proof of Proposition \ref{prop-first-step}}\label{sec-conclusion}
 Our final goal in this section is to conclude the proof of  Proposition \ref{prop-first-step},  by showing  that our auxiliary assumption \eqref{eq-desired-C2}
 holds for all $\tau \in [\tau_0, \tau_1]$. In order to do that, our subsolutions and supersolutions constructed in section \ref{sec-barriers} will play an important role.

\begin{lemma}
\label{lemma-error-faraway}
As long as $\widebar{M}^{\bf\Omega} \in \mc F_{\tau_1}$,  for some $\tau_1 >\tau_0$, the following bound holds 
\begin{equation}
\label{eq-Linfty-need}
\Big |\, u_{\bf\Omega}(y,\theta,\tau) - \sqrt{2(n-k)}\,  \Big |_{C^3} \leq  \frac{1}{\tau^{1- 400a}},
\end{equation}
for all $|y| \le 2 \tau^{1/10}$, and all $\tau\in [\tau_0,\tau_1]$.
\end{lemma}

\begin{proof}
For simplicity of  notation, we drop the index $\bom$ and denote $u_\bom$ by $u$. Observe first that by \eqref{eq-initial-data}, estimate \eqref{eq-Linfty-need} holds at $\tau = \tau_0$. Assume the statement were false, and that there exists  some maximal time $\tau_2 < \tau_1$ so that $\widebar{M}^{\bf\Omega} \in \mc{F}_{\tau_1}$, but \eqref{eq-Linfty-need} holds  only for $\tau\in [\tau_0,\tau_2]$.

 Since $ \widebar{M}^{\bf\Omega} \in \mc F_{\tau_2}$ and \eqref{eq-Linfty-need} holds for all $\tau\in [\tau_0,\tau_2]$, by Proposition \ref{prop-control},
 $v := u- \sqrt{2(n-k)}$ satisfies 
\be\label{eqn-v555}
v (y,\theta,\tau)\,  \phi(y,\theta,\tau) = \frac{\sqrt{2(n-k)}}{4\tau}\, (|y|^2 - 2k)  + \cO(\delta \tau^{-1}),\ee
for all $\tau\in [\tau_0,\tau_2]$, in the $L^2$ sense, where similarly to the previous subsection, $\cO(\delta \tau^{-1})$ means that it is bounded by $C_0 \delta \tau^{-1}$ for a uniform constant $C_0$.  

Given the above, we may assume that   \eqref{eqn-v555}  and \eqref{eq-Linfty-need} hold  for $\tau\in [\tau_0,\tau_2]$. Then we have the following claim.

 \begin{claim}\label{claim-C0}   Let $\ell$ be a large but fixed  number so that the  barriers in   Definition \ref{def-barrier-set}
are defined for $|y| \geq \ell$.  Standard parabolic estimates imply the $C^0$ bound 
\begin{equation}
\label{eq-C0-inside}
 \Big  |  v(y,\theta,\tau) - \frac{\sqrt{2(n-k)}}{4\tau}\, (|y|^2 - 2k) \Big |_{C^0}  \leq \frac{C_0\, \delta}{\tau}
\end{equation}
for $|y| \le 2\ell$, $\theta \in \mathbb{S}^{n-k}$, and $\tau \in [\tau_0, \tau_2]$, where $C_0$ depends  on  $\ell$,  a big but fixed number that will not change throughout the proof. 
Thus, we  suppress the dependence of $C_0$ on $\ell$ on the right hand side of \eqref{eq-C0-inside}. 
\end{claim} 

\begin{proof}[Proof of Claim \ref{claim-C0}]
All our estimates below hold for all $\theta \in \mathbb{S}^{n-k}$. We recall that $v$ satisfies equation \eqref{eqn-vtau} with error given by \eqref{eqn-error-v}. It follows that 

$f:=  v- \frac{\sqrt{2(n-k)}}{4\tau}\, (|y|^2 - 2k)$  satisfies 
\begin{equation}
\label{eqn-ftau}
f_{\tau} -  \mathcal{L} f =  \mathcal{E}(v,\nabla v, \nabla^2v) + \tfrac{\sqrt{2(n-k)}}{4\tau^2}\, (|y|^2 - 2k)=:\widetilde \cE
\end{equation}

where for simplicity we have set  $\widetilde \cE :=  \mathcal{E}(v,\nabla v, \nabla^2v) + \tfrac{\sqrt{2(n-k)}}{4\tau^2}\, (|y|^2 - 2k)$.
Observe that since  $v(y,\theta,\tau)= u_\bom(y,\theta, \tau)  - \sqrt{2(n-k) } $ satisfies \eqref{eq-Linfty-need},  for all $\tau\in [\tau_0,\tau_2]$, 
 $|y| \leq 2\ell$, $\theta \in \mathbb{S}^{n-k}$,  
the error $\widetilde \cE$,  which is at least quadratic in $v$ and its derivatives, satisfies the bound
\be\label{eqn-E555}  \big \| \, \widetilde \cE(\cdot, \tau)  \, \|_{\cH} \,  \leq \frac{C(\ell)}{\tau^{2-800a}}, \qquad \mbox{on}\,\, |y| \leq 2\ell, \,\,
\tau \in [\tau_0,\tau_2]. \ee

 First,  one applies standard  interior $L^\infty$-estimates for solutions of the  equation \eqref{eqn-ftau} 
on  $\tau \in [\tau_0+1, \tau_2]$ (if $\tau_2 < \tau_0+1$
we ignore this case) to obtain the bound  
\begin{equation}		\label{eqn-f234}
\| f(\cdot, \tau) \|_{C^0([0, 2\ell])}   \leq  C(\ell)  \sup_{\tau' \in [\tau-1, \tau]}  \big ( \| f(\cdot, \tau') \|_{L^2([0, 4\ell])}  + \| \widetilde 
\cE(\cdot, \tau') \|_{L^2([0, 4\ell])}  \big )
\end{equation}
for all $\tau \in [\tau_0 +1, \tau_2]$. 
On the other hand, since our  cut off function $\phi(y,\tau) \equiv 1$ on $|y| \leq 4\ell$,  \eqref{eqn-v555} gives that
\[\big \| f(\cdot, \tau)   \cdot  \chi_{_{\{ |y| \leq 4\ell \}}}  \big \|_{\cH} = \big \| \big ( v(\cdot, \tau)    - \tfrac{\sqrt{2(n-k)}}{4\tau}\, (|y|^2 - 2k)  \big ) \cdot  \chi_{_{\{ |y| \leq 4\ell \}}}  
\big \|_{\cH}= o(\tau^{-1}) \]
for all $\tau \in [\tau_0, \tau_2]$, and therefore 
\begin{equation} \label{eq-f235}
  \sup_{\tau' \in [\tau-1, \tau]}  \,  \| f(\cdot, \tau')  \|_{L^2([0, 4\ell])}
\leq C(\ell)\,   \sup_{\tau' \in [\tau-1, \tau]} \|  f (\cdot, \tau) \cdot \chi_{_{|y| \leq 4\ell} }  \|_{\cH} \leq C(\ell)\, o(\tau^{-1}) 
\end{equation}
for all  $\tau \in [\tau_0 +1, \tau_2]$. Combining \eqref{eqn-E555}, \eqref{eqn-f234} and  \eqref{eq-f235}, yields that  \eqref{eq-C0-inside}
holds for $\tau \in [\tau_0 +1, \tau_2]$.

For the  bound \eqref{eq-C0-inside}  in the interval $\tau \in [\tau_0, \tau_0+1]$,  (or  $\tau \in [\tau_0, \tau_2]$ if $\tau_2 \leq \tau_0+1$)  one uses instead the  estimate
\begin{equation}		\label{eqn-f2345}
\| f(\cdot, \tau) \|_{C^0([0, 2\ell])}   \leq  C(\ell) \Big ( \|  f(\cdot, \tau_0) \|_{C^0([0, 4\ell])}  +  \sup_{\tau' \in [\tau_0, \tau_0+1]}   \| \widetilde 
\cE(\cdot, \tau') \|_{L^2([0, 4\ell])}  \Big ) 
\end{equation}
and uses the bound  $ \|  f(\cdot, \tau_0) \|_{C^0([0, 4\ell])} \leq  C(\ell)  \,  \delta \,  \tau^{-1}$ which follows from our initial condition, and  the bound 
$\sup_{\tau' \in [\tau_0, \tau_0+1]}  \| \widetilde \cE(\cdot, \tau') \|_{L^2([0, 4\ell])} \leq C(\ell) \,  \delta \, \tau^{-1}$ (which follows by \eqref{eqn-E555}) to conclude that \eqref{eq-C0-inside}
holds for $\tau \in [\tau_0, \tau_0+1]$,   (or  for $\tau \in [\tau_0, \tau_2]$ if $\tau_2 \leq \tau_0+1$). This concludes the proof of the claim. 
\end{proof}

Now, let us continue with the proof of  Lemma \ref{lemma-error-faraway}, omitting index ${\bf\Omega}$ for simplicity of notation.  Note that \eqref{eq-C0-inside} in particular yields that $q: = u^2  - 2(n-k)$ satisfies 
\begin{equation}
\label{eq-one-bound}
\frac{C_0}{8\tau}\, \ell^2 \le q(\ell, \theta, \tau) \le \frac{C_0}{\tau}\, \ell^2
\end{equation}
for all  $\tau\in [\tau_0,\tau_2]$, and for a uniform constant $C_0$, depending on $n$ and $k$.

Recall  the  $1$-parameter family of barriers  $M_{\tau}^-$ and $M_{\tau}^+$ that were constructed in  section \ref{sec-barriers} and they are  defined 
as in Definition \ref{def-barrier-set} and Lemmas \ref{lem-sup}, \ref{lem-sub}, and also satisfy  Lemma \ref{lemma-barrier-set}. They are represented by rotationally symmetric profiles $u^-(y, \theta, \tau)$ and $ \bar u^+(y,\theta,\tau)$,  respectively such that the corresponding 
$q^\pm(y, \theta,  \tau) := u^\pm (y, \theta, \tau)^2 - 2(n-k)$, are given by
\be\label{eqn-qpm}
q^-(y,\theta,  \tau) =-2(n-k) \, e^{|y|-\tau^b} \qquad \mbox{and} \qquad q^+(y, \theta,  \tau) = \tau^{-1} \, |y|^{2p}
\ee
for any  
$b \in (0, 1] $ and $ p >1$. Recall also that $\widebar{M}_{\tau}$ is the rescaled mean curvature flow whose initial data satisfies \eqref{eq-initial-data}, where we have dropped the index ${\bf\Omega}$ for the simplicity of notation. Denote by $\widebar{K}_{\tau}$, $K^+_{\tau}$ and $K^-_{\tau}$ subsets of $\mathbb{R}^{n+1}$  so that their boundaries are $\widebar{M}_{\tau}$, $M^+_{\tau}$ and $M^-_{\tau}$, respectively.

Using \eqref{eq-one-bound} and \eqref{eqn-qpm},  our next goal is to prove that we have
\begin{equation}
\label{eq-int-tau}
\operatorname{int}(K^-_{\tau}) \subset \operatorname{int} (\widebar{K}_{\tau} \cap |y| \ge \ell) \quad \mbox{and} \quad \operatorname{int}(\widebar{K}_{\tau}\cap \{\ell \le |y| \le L \, e^{\tau-\tau_0}\}) \subset\operatorname{int}(K^+_{\tau}),
\end{equation}
for all $\tau\in [\tau_0,\tau_2]$,  where $L$ is the same constant  as in Proposition \ref{the:first_singular_time} (we will choose  $L = \ve^{-\frac 1{300}}$).

Note first that by section \ref{sec-barriers} and \eqref{eq-initial-data}  it follows that for a small $\delta$ (which is the number in \eqref{eq-initial-data}), we have
\begin{equation}
\label{eq-int-0}
\operatorname{int}(K^-_{\tau_0}) \subset \operatorname{int} (\widebar{K}_{\tau_0} \cap \{\ell \le |y| \le L\}) \subset \operatorname{int}(K^+_{\tau_0}).
\end{equation}
In order to prove \eqref{eq-int-tau}, we need to compare our barriers $M^-_{\tau}$, $M^+_{\tau}$  with  $\widebar{M}_{\tau}$ at their inner  boundary $|y|=\ell$ and also
at $|y| = L \, e^{\tau-\tau_0}$. 

Let us first observe  that \eqref{eq-one-bound} implies that at $|y| = \ell$ we have $q^-(\ell,\theta,\tau) \le q(\ell,\theta,\tau) \le q^+(\ell,\theta,\tau)$, and hence the boundaries of  $M^-_{\tau}$, $M^+_{\tau}$  and $\widebar{M}_{\tau}$, 
lie on $\{|y| = \ell\}\times \mathbb{R}^{n-k+1}$, and the corresponding cross sections are strictly separated and ordered for all times, with the ordering that is consistent with \eqref{eq-int-0}.  More precisely, the cross section of $\widebar{K}_{\tau}$ in $\{|y| = \ell\}\times \mathbb{R}^{n-k+1}$ contains cross section of $K^-_{\tau}$, and is contained in the cross section of $K^+_{\tau}$, for all $\tau\in [\tau_0,\tau_2]$.

On the other hand, by \eqref{eq-help1111}, if $L$ is the same as in  the statement of Proposition \ref{the:first_singular_time},  by Corollary \ref{cor-directly-used},

after fixing the constants   $a, b, p$ and the number $L$  as follows
\be\label{eqn-fix}
a=\tfrac 1{1500}, \quad   b= \tfrac 1{400},  \quad  p = 150, \quad  L = \ve^{-\frac{1}{300}},
\ee
where we chose $a$ so that  $1-400 a > 0$ in  \eqref{eq-desired-C2} and $3-1200a >2$ in \eqref{eqn-err5}, we obtain 
\be\label{eqn-comp1}  -A_0 \ep_0 + \frac{A_1}{2}\, \ep\, L^2 \, e^{\tau-\tau_0} 
\le q(L\, e^{\frac{\tau-\tau_0}{2}},\theta,\tau)  \le A_0 \ep_0 + 2 A_2 \ep \, L^2\, e^{\tau-\tau_0},\ee
for all $\tau\in [\tau_0,\tau_2]$, where $A_0, A_1, A_2$ are the same numerical constants as in \eqref{eq-help1111}.
We will use this to compare  $q(y,\theta,\tau)$ with $q^-(y,\theta,\tau)$ and  $q^+(y,\theta,\tau)$  at $|y|= L \, e^{\frac{\tau-\tau_0}2}$. Indeed,  by 
\eqref{eqn-comp1},   using \eqref{eqn-qpm}, we need to see that 
  \begin{equation}
\label{eq-two-bound}
\begin{split}
-2(n-k)\, e^{L\, e^{\frac{\tau-\tau_0}{2}} - \tau^b}< -A_0 \ep_0 + \frac{A_1}{2}\, \ep\, L^2 \, e^{\tau-\tau_0} \qquad \mbox{and} \qquad 
 A_0 \ep_0 + 2 A_2 \ep L^2\, e^{\tau-\tau_0} <  \frac{L^{2p} e^{p(\tau-\tau_0)}}{\tau}
\end{split} 
\end{equation}
for the choices of exponents $b$, $p$ and constant $L$ as in \eqref{eqn-fix}.
To prove \eqref{eq-two-bound}, recall  that  $\ve^{-1} \sim \tau_0$ up to a constant. 
To verify  the first  bound in \eqref{eq-two-bound} it is sufficient to see that  $2(n-k)\, e^{L\, e^{\frac{\tau-\tau_0}{2}} - \tau^b}\ge  A_0 \ep_0 $ for all $\tau \geq \tau_0$,
which simply follows from the fact that for our choice of $b$, $L\, e^{\frac{\tau-\tau_0}{2}} - \tau^b$ is increasing in $\tau$, and therefore 
$2(n-k)\, e^{L\, e^{\frac{\tau-\tau_0}{2}} - \tau^b}\ge 2(n-k)\, e^{L - \tau^b_0} > A_0 \ep_0$, where in the last inequality we used that $\tau_0 \sim \ep^{-1}, L = \ve^{-\frac{1}{300}}$ and $b =\ve^{-\frac{1}{400}}$. The second bound in \eqref{eq-two-bound} readily follows from the fact that  $e^{p (\tau-\tau_0) }\,\tau^{-1}$ and $e^{(p-1) ( \tau-\tau_0)}\, \tau^{-1}$ increase in $\tau$ for $\tau \geq \tau_0 \gg 1$. Finally, \eqref{eqn-comp1} and \eqref{eq-two-bound} yield
\begin{equation}
\label{eq-two-bound1}
q^-(L\, e^{\frac{\tau-\tau_0}{2}},\theta,\tau) < q(L\, e^{\frac{\tau-\tau_0}{2}},\theta,\tau) < q^+(L\, e^{\frac{\tau-\tau_0}{2}},\theta,\tau).
\end{equation}

Observe that  these bounds from above and below on $q(L\, e^{\frac {\tau-\tau_0}2},\theta,\tau)$ imply  that at $|y| = L\, e^{\frac{\tau-\tau_0}{2}}$, the boundaries of $\widebar{K}_{\tau}$, $K^-_{\tau}$ and $K^+_{\tau}$ lie on $\{|y| = L \, e^{\frac{\tau-\tau_0}{2}}\}\times \mathbb{R}^{n-k+1}$, and the corresponding cross sections are strictly separated and ordered for all times, with the ordering that is consistent with the  bounds in \eqref{eq-two-bound1}.  More precisely, the cross section of $\widebar{K}_{\tau}$ in $\{|y| = L\, e^{\frac{\tau-\tau_0}{2}}\}\times \mathbb{R}^{n-k+1}$  is strictly contained in the cross section of $K^+_{\tau}$, and the cross section of $K^-_{\tau}$ in $\{|y| = L\, e^{\frac{\tau-\tau_0}{2}}\}\times \mathbb{R}^{n-k+1}$ is strictly contained in the cross section of $\widebar{K}_{\tau}$, for all $\tau\in [\tau_0,\tau_2]$.

By above analysis and Lemma \ref{lemma-barrier-set}, if $\widebar{M}_{\tau}$ ever touches $M^-_{\tau}$, it must happen at a point of first contact $(y_0, \theta_0,\bar{\tau})$, strictly in $|y| > \ell$. At $|y| > \ell$, both $\widebar{M}_{\tau}$ and $M^-_{\tau}$ are smooth surfaces, with $\widebar{M}_{\tau}$ satisfying the mean curvature flow equation, while $M^-_{\tau}$  has the property that its graph is a subsolution  to the graphical symmetric equation to the mean curvature flow up to the tip, while around the tip, after inverting the coordinates is a supersolution to the graphical mean curvature flow equation in inverted coordinates. The strong maximum principle immediately applies to get contradiction. We argue similarly that $\widebar{M}_{\tau}$ and $M^+_{\tau}$ also can not touch for all $\tau\in [\tau_0,\tau_2]$. This concludes the proof of \eqref{eq-int-tau}.

\smallskip
We next note  that by  Corollary \ref{cor:cylindrical_est},  we have that every point in $|y| \le \tau$ is neck-like, i.e. there exists a neighborhood around it in which the surface is $C^k$ close to some cylinder $\mathbb{R}^k\times\mathbb{S}^{n-k}$.   By \eqref{eq-C0-inside}, Corollary \ref{cor-directly-used}  and \eqref{eq-int-tau}, keeping in mind that $p = 150$,  we have that 
 \begin{equation}
\label{eq-C0-better}
\Big |u_{\bf\Omega}(y,\theta,\tau) - \sqrt{2(n-k)}\Big|_{C^0}\le \frac{1+ |y|^{2p}}{\tau} \le \frac{2^{301}}{\tau^{1-300a}},
\end{equation}
for $|y| \le 2\,\tau^{a}$, and all $\tau\in [\tau_0,\tau_2]$. 

To  conclude the proof of Lemma \ref{lemma-error-faraway} we need to prove the following improved estimate compared to \eqref{eq-Linfty-need} that we claim to hold as long as our solution stays in the funnel.

\begin{claim}
\label{claim-C3}
\[ \Big |u_{\bf\Omega}(y,\theta,\tau) - \sqrt{2(n-k)} \Big |_{C^3} \le \frac{1}{\tau^{1-350a}},\]
for $|y| \le 2 \, \tau^{a}$, and all $\tau\in [\tau_0,\tau_2]$.
\end{claim}

\begin{proof}
To prove the Claim, we proceed in two steps. In the first step, using \eqref{eq-C0-better} and Proposition \ref{prop-unique-cont} we show that $\widebar{M}_{\tau}$ is a graph over $\cC_{n,k}$, on a larger set, that is, for all $|y| \le 4\tau^{a}$ and all $\tau\in [\tau_0,\tau_2]$. In the second step we then conclude the proof of the Claim, using \eqref{eq-int-tau}, and standard interior parabolic estimates. 

Let us show the first step now.
Let $\ve$ be sufficiently small, and hence $\tau_0$ sufficiently big  such that $\frac{2^{300}}{\tau_0^{1-300a}} < \ve_2$, where $\ve_2$ is the same small constant as in Proposition \ref{prop-unique-cont}, and recall that  $\ve = \frac{\sqrt{2(n-k)}}{4\tau_0}$. We may assume that $L_{\ve}$ in Theorem \ref{thm-main0} and Theorem \ref{thm-main} is chosen so that the initial data $\widebar{M}_{\tau_0}$ is an $\ve$ graph over $\cC_{n,k}$ on $B_{8\,\tau_0^a}$. By taking $\tau_0$ big enough we may assume  $\tau_0^{a} \gg R_1$, where $R_1$ is the same constant as in Proposition \ref{prop-unique-cont}. The consequence of pseudolocality then implies that $\widebar{M}_{\tau}$ is an $\ve_0$-graph over $\cC_{n,k}$ on $B_{e^{(\tau-\tau_0)/2}\, (8\, \tau_0^{a} - R_1)}$, for $\tau\in [\tau_0, \tau_0 + 10]$. Note that if we take $\tau_0$ sufficiently big, then
\[e^{\frac{\tau-\tau_0}{2}} \, (8\, \tau_0^{a}  - R_1) >  \, 8\, \tau^{a} - R_1 , \qquad \tau\in [\tau_0,\tau_0 + 10],\]
implying that $\widebar{M}_{\tau}$ is an $\ve_0$ graph over $\cC_{n,k}$ for $|y| \le 8\, \tau^a - R_1$, and all $\tau\in [\tau_0,\tau_0 + 10]$. Since $8\, \tau^{a} - R_1 > 4\, \tau^{a}$, if $\tau_2 < \tau_0 + 10$ we are done. The reason for that is that  by   \eqref{eq-C0-inside}, Corollary \ref{cor-directly-used}, \eqref{eq-int-tau}, and the fact that $\widebar{M}_{\tau}$ is a graph over $\cC_{n,k}$ on $B_{4\tau^a}$, and all $\tau\in [\tau_0,\tau_2]$,  we have that \eqref{eq-C0-better} holds for all $|y| \le 4\tau^a$, and all $\tau\in [\tau_0,\tau_2]$. Standard interior parabolic estimates in this case conclude the proof of the Claim.

Assume now $\tau_2 > \tau_0 + 10$. Recall again that by Corollary \ref{cor:cylindrical_est} we have that every point in $|y| \le \tau$ is neck-like, i.e. there exists a neighborhood around it in which the surface is $C^k$ close to some cylinder $\mathbb{R}^k\times\mathbb{S}^{n-k}$.  By \eqref{eq-C0-inside}, Corollary \ref{cor-directly-used}, \eqref{eq-int-tau}, and the fact that $\widebar{M}_{\tau}$ is a graph over the cylinder $\cC_{n,k}$, for $|y| \le e^{(\tau-\tau_0)/2}\, (8\, \tau_0^a - R_1)$ and all $\tau\in [\tau_0,\tau_0+10]$,  it follows  that \eqref{eq-C0-better} holds for all $|y| \le e^{(\tau-\tau_0)/2}\, (8\, \tau_0^a - R_1)$ and all $\tau\in [\tau_0,\tau_0+10]$. Since $\frac{2^{301}}{\tau^{1-300a}} \le \frac{2^{301}}{\tau_0^{1-300a}} < \ve_2$, for all $\tau \ge \tau_0$, we have that $\widebar{M}_{\tau}$ is an $\ve_2$-graph over $\cC_{n,k}$, on $B_{e^{(\tau-\tau_0)/2}\, (8\, \tau_0^a - R_1)}$, for all $\tau\in [\tau_0,\tau_0+10]$. The consequence of pseudolocality then implies that $\widebar{M}_{\tau}$ is an $\ve_0$-graph over $\cC_{n,k}$ on a ball of radius $e^{\big(\tau-(\tau_0+10)\big)/2}\, \Big(e^5\, (8\, \tau_0^a - R_1) - R_1\Big)$, for all $\tau\in [\tau_0+10, \tau_0 + 20]$. We keep repeating the process above until we reach the time $\tau_2$.

More precisely, we have the following.
For simplicity, denote by $f_0 = 8 \tau_0^a$, $f_1(\tau) = e^{(\tau-\tau_0)/2}\, (f_0 - R_1)$ for $\tau\in [\tau_0,\tau_0+10]$, and 
\[f_k(\tau) = e^{(\tau-(\tau_0+10(k-1)))/2}\, (f_{k-1}(\tau) - R_1) \,\,\, \mbox{for all} \,\,\,\tau\in [\tau_0 + 10(k-1), \min\{\tau_2, \tau_0 + 10k\}].\]
By repeating the process above we have that $\widebar{M}_{\tau}$ is an $\ve_2$-graph over $\cC_{n,k}$ on $B_{f_k(\tau)}$ for all $\tau\in  [\tau_0 + 10(k-1), \min\{\tau_2, \tau_0 + 10k\}]$. By Lemma \ref{lemma-beh-fk} in Appendix we have that as long as $\tau_0^a\gg R_1$,  for $k \ge 1$ we have 
\[f_k(\tau) \sim 8 \tau_0^a\, e^{(\tau-\tau_0)/2} \ge 4\, \tau^a,\]
for all $\tau \ge \tau_0$, as long as $\tau_0$ is sufficiently big. This concludes the proof of the first step of the Claim that $\widebar{M}_{\tau}$ is an $\ve_2$-graph over $\cC_{n,k}$ on $B_{4\tau^a}$, for all $\tau\in [\tau_0,\tau_2]$.

To conclude the proof of the Claim note that  by \eqref{eq-C0-inside},  Corollary \ref{cor-directly-used}, \eqref{eq-int-tau}, and the fact that $\widebar{M}_{\tau}$ is a graph over $\cC_{n,k}$ on $B_{4\tau^{1/10}}$, and all $\tau\in [\tau_0,\tau_2]$  we have that \eqref{eq-C0-better} holds for all $|y| \le 4\tau^{1/10}$, and all $\tau\in [\tau_0,\tau_2]$. Standard interior parabolic estimates finally yield the desired $C^3$ estimate.

\end{proof}

Finally, to {\em conclude the proof of Lemma \ref{lemma-error-faraway}}, we note that the $C^3$ estimate that was shown in Claim \ref{claim-C3}  is  sharper than \eqref{eq-Linfty-need} 
which contradicts the maximality of $\tau_2 < \tau_1$. This shows \eqref{eq-Linfty-need} needs to hold for all $\tau\in [\tau_0,\tau_1]$.
\end{proof}

\section{Shooting argument and the proof of Theorem \ref{thm-main}}\label{sec-shooting}

Recall the discussion in Section \ref{sec-in-data},  where for any $\bom = (\Omega_0, \Omega_1^i) \in \mathbb{B}^{k+1}$, we defined 
 the rescaled flow $\widebar M^\bom_\tau$ whose profile function at the initial time $u_\bom(\cdot, \tau_0)$ satisfies \eqref{eq-initial-data},  and  that parameters the $\Omega_0, \Omega_1$ are related to choosing the time around which we rescale, i.e. the scaling factor, and choosing the spatial center in the $\mathbb{R}^k$ direction around which we rescale the original flow. We also recall Corollary  \ref{cor-directly-used} which shows  that all other parameters  corresponding to angular modes (which we denoted in Corollary \ref{cor-directly-used} as $\bf\Omega'$) are chosen to depend continuously on $\bf\Omega$ so that all results proved in section \ref{sec-L2-funnel} hold.

We will first employ degree theory and shooting type of arguments to show there exists an  ${\bf\Omega} \in \mathbb{B}^{k+1}$,  so that 
the solution $\widebar{M}_\tau^{\bf\Omega} \in \mc S_{\tau}$  starting with initial data given by \eqref{eq-initial-data} stays in the funnel $\mc F_{\tau}$, for all $\tau\in [\tau_0,\infty)$. 
We remind the reader that  $\mc{S}_{\tau}$ denotes the  set of all hypersurfaces  which are locally graphs over our fixed cylinder $\mc{C}_{n,k}$ on a set containing $|y| \le  2\, \tau^a$, where $a$ is the small number chosen in the previous section (see \eqref{eqn-fix}). Note
also  that choosing the right parameter ${\bf\Omega} \in \mathbb{B}^{k+1}$ corresponds to choosing the right center of the neck, rescaling (dilating) solution with respect to the right singular time, and therefore also choosing the right rigid rotation of the cylinder. After that we will provide the proof of Theorem \ref{thm-main}.

\smallskip 
Let $\mathbb{S}^k=\partial \mathbb{B}^{k+1}$ denote the unit sphere. 
Define the map $\gamma_{\tau}: \mc{F}_{\tau}^c \to \mathbb{S}^k$ by
\[\gamma_{\tau}(\widebar{M}^\bom_{\tau}) = \frac{\tilde{\Gamma}_+^{\bf\Omega}(\tau)}{\|\tilde{\Gamma}^{\bf\Omega}_+(\tau)\|},  \]
where $\mc{F}_{\tau}^c$ is the complement of $\mc{F}_{\tau}$ in  $\cS_\tau$. Here $\tilde{\Gamma}_+^\bom$ is defined as follows 
\[\tilde{\Gamma}_+^\bom := \Big(\frac{\langle \bar{v}_{\bf\Omega}, 1\rangle}{\|1\|^2}, \frac{\langle \bar{v}_{\bf\Omega}, y_1\rangle}{\|y_1\|^2}, \dots \frac{\langle \bar{v}_{\bf\Omega}, y_k\rangle}{\|y_k\|^2}\Big) \in \mathbb{R}^{k+1},\]
 and thus $\frac{\tilde{\Gamma}_+^\bom}{\|\tilde{\Gamma}_+^\bom\|} \in \mathbb{S}^k$. 
\begin{remark} By \eqref{eq-hatV100} we find that for $\widebar{M}^{\bf\Omega}_{\tau_0} \in \mc{F}_{\tau_0}$ we need \eqref{eq-motivation-funnel}. Therefore, if $|\bom | \geq \frac 12$, then $\widebar{M}^{\bf\Omega}_{\tau_0} \in \mc{F}_{\tau_0}^c.$

\end{remark} 
Let $a > 0$ be the  small number as fixed in the previous section (see \eqref{eqn-fix}). Next we define the notion of {\it locally cylindrically maximal} solution to the rescaled MCF.

\begin{definition}
\label{def-max-solution}
 We say $\widebar{M}^{\bf\Omega}_{\tau}$ for $\tau\in [\tau_0,\tau_1)$ is {\it locally cylindrically maximal} RMCF solution if  $\tau_1$ is the maximal time up to which $\widebar{M}^{\bf\Omega}_{\tau}$ can be locally written as a graph of a $C^3$ function $u_{\bf\Omega}(y,\theta,\tau)$ with
\[ \Big |u_{\bf \Omega}(y,\theta,\tau) - \sqrt{2(n-k)} \Big |_{C^3} < \frac{1}{\tau^{1-500a}}\]
for all $|y| \le  2\tau^{a}$ and all $\tau\in [\tau_0,\tau_1)$.

\end{definition}

\begin{lemma}[Exit lemma]
\label{lemma-exit}
We have the following:
\begin{enumerate}
\item[(a)]
If $\widebar{M}^{{\bf\Omega}}_{\tau}$, for $\tau\in [\tau_0,\tau_1)$ is locally cylindrically maximal rescaled MCF solution with initial data \eqref{eq-initial-data} in $\mc F_{\tau_0}$, then there exists $\tau_2\in [\tau_0,\tau_1)$ such that $\widebar{M}^{{\bf\Omega}} \in \partial\mc F_{\tau_2}$, or else $\widebar{M}^{\bf\Omega} \in \mc F_{\tau}$, for all $\tau\in [\tau_0,\infty)$.
\item[(b)] 
If $\widebar{M}^{\bf\Omega} \in \partial\mc F_{\tau_2}$, and $\widebar{M}^{\bf\Omega}_{\tau} \in \mc F_{\tau}$ for all $\tau\in [\tau_0,\tau_2]$, then
\[\frac{d}{d\tau} \Big(\tau\,(\hat{V}_{{\bf\Omega}})_+)\Big)\Big |_{\tau = \tau_2} > 0.\]
\end{enumerate}
\end{lemma}

\begin{proof}
To prove part (a), assume $\widebar{M}^{\bf\Omega}_{\tau}$ is a locally cylindrically maximal solution for $\tau\in [\tau_0,\tau_1)$ in the sense of Definition \ref{def-max-solution}. Assume $\widebar{M}^{\bf\Omega}_{\tau} \notin \partial \mc F_{\tau}$, for any $\tau\in [\tau_0,\tau_1)$. If $\tau_1 = \infty$ we are done, since in this case we would have $\widebar{M}^{\bf\Omega} \in \mc F_{\tau}$ for all $\tau\in [\tau_0,\infty)$. Thus assume $\tau_0 < \tau_1 < \infty$.  In this case, by Lemma \ref{lemma-error-faraway}, we would have 
\[\Big |u_{\bf \Omega}(y,\theta,\tau) - \sqrt{2(n-k)} \Big |_{C^3} < \frac{1}{\tau^{1-400a}},\]
 for all $|y| \le 2\, \tau^{a}$ and all $\tau\in [\tau_0,\tau_1)$, which would contradict the maximality of $\tau_1$. Hence, part (a) holds.

We will below suppress the dependence of $\Omega$ due to simplifying the notation.      Let us write $V_+^2 = \hat{V}_+^2 + (V^{\theta}_+)^2$, where $V^{\theta}_+$ is the norm of the projection of our solution to positive eigenmodes $\{\theta_1, \dots \theta_{n-k+1}\}$, and where $\hat{V}_+$ is the norm of the projection of our solution to the eigenspace spanned by $\{1, y_1, \dots, y_k\}$. Note that $\hat{V}_+ = |\tilde{\Gamma}_+|$, where on the right hand side we have a euclidean norm in $\mathbb{R}^{k+1}$.

To show part (b), we argue as follows. Since $\widebar{M}^{\bf\Omega}\in \mc{F}_{\tau_2}$, by Lemma \ref{lemma-error-faraway} we have that \eqref{eq-Linfty-need} holds for all $\tau\in [\tau_0,\tau_2]$. By Lemma  \ref{lem-error} we then have
\[\frac{d}{d\tau}\hat{V}_+ \ge c_0\, \hat{V}_+  - \frac{B}{\tau^{2-800a}},\] 
where $c_0 > 0$ is a uniform constant. This implies
\[\frac{d}{d\tau}(\tau \hat{V}_+)\ge \hat{V}_+ + c_0\, \tau \hat{V}_+  -  \frac{B}{\tau^{1-800a}}.\]
At the exit time $\tau_2$ we have that $\tau_2 \,\hat V_+(\tau_2) = \delta \, \sqrt{2(n-k)}$. Combining this with above
\[\frac{d}{d\tau}(\tau\,\hat{V}_+) \Big|_{\tau=\tau_2} \ge c_0\,\delta \, \sqrt{2(n-k)} - \frac{B}{\tau_2^{1-800a}} > 0,\]
provided that the right hand side is positive, which can be achieved by taking $\tau_0$ large enough, as long as $1 - 800a > 0$ ($\delta > 0$ is a small but fixed number).
\end{proof}

\begin{lemma}[Shooting lemma]
\label{lemma-shooting}
Let $\widebar{M}^{\bf\Omega}_{\tau}$ be our family of solutions with initial data \eqref{eq-initial-data}, with ${\bf\Omega} \in \mathbb{B}^{k+1}$, and  $\widebar{M}^{\bf\Omega}_{\tau_0} \in \mc{F}_{\tau_0}^c$, for ${\bf\Omega} \in \partial \mathbb{B}^{k+1} = S^k$. If the map
\[\gamma_{\partial} : {\bf\Omega} \in \mathbb{S}^k\to \gamma_{\tau_0}(\widebar{M}^{\bf\Omega}_{\tau_0}) \in \mathbb{S}^k,\]
has nonzero degree, then there exists an ${\bf\Omega}\in \mathbb{B}^{k+1}$, so that the solution $\widebar{M}^{\bf\Omega}_{\tau}$ with initial data  defined by \eqref{eq-initial-data} is defined for $\tau\in [\tau_0,\infty)$, and $\widebar{M}^{{\bf\Omega}}_{\tau} \in \mc F_{\tau}$, for all $\tau\in [\tau_0,\infty)$.
\end{lemma}

\begin{proof}
First of all note that the map $\gamma_{\partial}$ is well defined due to \eqref{eq-motivation-funnel}, which in particular implies that $\widebar{M}^{\bf\Omega}_{\tau_0} \in \mc{F}_{\tau_0}^c$, for ${\bf\Omega} \in \partial \mathbb{B}^{k+1} = \mathbb{S}^k$. For each ${\bf\Omega} \in \mathbb{B}^{k+1}$ we let $\widebar{M}^{\bf\Omega}_{\tau}$ for $\tau\in [\tau_0, \tau_1({\bf\Omega}))$ be the locally cylindrical maximal RMCF solution with initial data $\widebar{M}^{\bf\Omega}_{\tau_0}$.

We argue by contradiction and assume that none of the solutions stay in $\mc{F}_{\tau}$ for all $\tau < \tau_1(\Omega)$.
 By the exit lemma \ref{lemma-exit},  there is the first $\tau_{ex}({\bf\Omega}) \in [\tau_0,\tau_1({\bf\Omega}))$ at which $\widebar{M}^{\bf\Omega}_{\tau}$ hits $\mc F_{\tau}^c$. As in Lemma 3.2 in \cite{AV} we argue that $\tau_{ex}({\bf\Omega})$ depends continuously on ${\bf\Omega}$ by part (b) of the exit lemma. Thus, we can extend the map $\gamma_{\partial}$ continuously to a map $\gamma_{ex}$ from $\mathbb{B}^{k+1}$ to $\mathbb{S}^k$, by
\[\gamma_{ex}({\bf\Omega}) = \gamma_{\tau_{ex}({\bf\Omega})}(\widebar{M}^{\bf\Omega}_{\tau_{ex}}),\]
which is impossible since $\gamma_{\partial}$ is a map of nonzero degree. This implies there exists an ${\bf\Omega} \in \mathbb{B}^{k+1}$ so that $\widebar{M}^{\bf\Omega}_{\tau} \in \mc F_{\tau}$ for all $\tau\in [\tau_0, \tau_1({\bf\Omega}))$, and hence by Lemma \ref{lemma-exit} we have $\tau_1({\bf\Omega}) = \infty$ as claimed.
\end{proof}

We are now ready to prove Theorem \ref{thm-main}.

\begin{proof}[Proof of Theorem \ref{thm-main}]
By Lemma \ref{lemma-shooting} there exists an ${\bf \Omega} \in \mathbb{B}^{k+1}$ so that $\widebar{M}_{\tau}^{\bf\Omega}$ starting with initial data satisfying \eqref{eq-initial-data} stays in $\mc F_{\tau}$ for all $\tau\in [\tau_0,+\infty)$. By Lemma \ref{lemma-error-faraway} we have \eqref{eq-Linfty-need} is satisfied for all $\tau\in [\tau_0,\infty)$. Then by Proposition \ref{prop-first-step} we have
\begin{equation}\label{eq-V0-roughly}
\frac{1}{4(1+2\eta_0)\, \tau}   \le V_0(\tau) \le \frac{1}{4(1-2\eta_0) \, \tau},
\end{equation} 
for all $\tau\in [\tau_0,\infty)$, where $|\eta_0| \le C_0\, \delta \to 0$ as $\delta\to 0$, and $\delta$ is a fixed small number as in \eqref{eq-small-ini} , depending on initial data \eqref{eq-initial-data}.

Next, Lemma \ref{lem-error} and estimate \eqref{eq-V0-roughly} yield
\[\frac{d}{d\tau} V_+ \ge c_0 \,V_+ - \eta \,V_0,\]
\[\left|\frac{d}{d\tau} V_0\right| \le \eta V_0,\]
\[\frac{d}{d\tau} V_- \le -c_0 V_- + \eta V_0,\]
for all $\tau\in [\tau_0,\infty)$, where $\eta(\tau) =  \frac{C_0}{\tau^{1-800a}} \to 0$ as $\tau\to\infty$.  

We can now apply the ODE Lemma of Merle-Zaag (\cite{MZ}) to conclude that either $V_+(\tau) + V_-(\tau) = o(V_0(\tau))$ or $V_+(\tau) + V_0(\tau) = o(V_-(\tau))$, as $\tau\to \infty$. If the latter were true, by Proposition \ref{prop-control} we would have $V_0(\tau) = o(V_-(\tau)) \le o(V_0(\tau))$. This would imply $V_0(\tau) \equiv 0$, contradicting \eqref{eq-V0-roughly}.

Hence,\begin{equation}
\label{eq-dom-V0}
V_+(\tau) + V_-(\tau) = o(V_0(\tau)),
\end{equation}
where $o(V_0(\tau)) = V_0(\tau) o(1)$ and $o(1) \to 0$ as $\tau\to\infty$. This enables us that in place of \eqref{eq-desired-1} we use  \eqref{eq-dom-V0}, and then
repeating the same arguments as in the proof of Proposition \ref{prop-control} yields
\[\Big(u(y,\theta,\tau) - \sqrt{2(n-k)}\Big) \phi(y,\theta,\tau) = \frac{\sqrt{2(n-k}}{4\tau}\, \big(|y|^2 - 2k\big)  + o(\tau^{-1}),\]
for all $\tau\in [\tau_0,\infty)$, in the $L^2$ sense, where $o(\tau^{-1})$ is understood to be $\tau^{-1} o(1)$, with the $\lim_{\tau\to\infty} o(1) = 0$.
Similarly as in  Claim \ref{claim-C0},  the  same asymptotics  hold on compact sets in the $C^0$ sense,   and by 
standard interior estimates one concludes  that  the desired asymptotics hold in the $C^{\infty}$ sense on compact sets as well, concluding the proof of Theorem \ref{thm-main}.
\end{proof}

\section{Stability of nondegenerate neckpinches}
\label{sec-stability}

The goal in this section is to prove Theorem \ref{thm-open-neck} using Theorem \ref{thm-main}. In the proof of Theorem \ref{thm-open-neck} we will need the following observation.

\begin{remark}\label{remark-singularities}  Let  $\mathcal{M} = \{M_t\}_{t \in [0, T)}$ be a smooth mean curvature flow of closed hypersurfaces in $\mathbb{R}^{n+1}$.  For a fixed $t_0 \in (0, T)$ and  a constant $\rho  > 0$, define the rescaled 
hypersurface $\widetilde{M}_0 = \rho \,  M_{t_0}$, and let $\widetilde{\mathcal{M}} = \{\widetilde{M}_s\}_{s \in [0, S)}$ be the maximal smooth mean curvature flow with initial condition $\widetilde{M}_0$. If $\widetilde{\mathcal{M}}$ develops a nondegenerate neckpinch singularity at time $S$ in the sense of Definition \ref{def-neckpinches},  then the original flow $\mathcal{M}$ develops a nondegenerate neckpinch singularity at time $T = t_0 + \rho^{-2}S$.
This follows from  the fact that  by  the parabolic scaling invariance and uniqueness of the mean curvature flow, the rescaled flow $\widetilde{M}_s$ is uniquely related to the original flow $M_t$ by
$\widetilde{M}_s = \rho \,  M_{t_0 + \rho^{-2}s}$, where $t=t_0+ \rho^{-2}s$,  and the fact that scaling by a constant doesn't change the nature of singularities. 
\end{remark}

\begin{proof}[Proof of Theorem \ref{thm-open-neck}]
Let $\{N_t\}_{t\in [0,T_0)}$ be a smooth mean curvature flow developing a nondegenerate neckpinch singularity at $(p_0,T_0)$, modeled on the cylinder  $\cC_{n,k}$ (see Definition \ref{def-neckpinches}).  In \cite{SX} it was proved that the singularity at $p_0$ is isolated. Assume with no loss of generality that $p_0=0$,
that is the singularity happens at $(0,T)$. 

By Theorem 1.2 in \cite{SX} we have that for every $K > 0$ and $\alpha \in (0,1)$ there exists a $\bar{\tau}_0$ so that for all $\tau \ge \bar{\tau}_0$, the parabolically rescaled mean curvature flow $\widebar{N}_{\tau}$ is a graph over $\mc{C}_{n,k}\cap B_{K\sqrt{\tau}}$, with the height function $\bar u(y,\theta,\tau)$,  measured from the center of the cylinder, 
satisfying 
\[\Big\| \bar u (y,\theta,\tau) - \sqrt{2(n-k)} - \frac{\sqrt{2(n-k)}}{4\tau}\, \big(|y|^2 - 2k\big)\Big\|_{ H^1\big(\mc{C}_{n,k}\cap B_{K\sqrt{\tau}}\big)} = o(\tau^{-1-\alpha}).\]

We can now apply Proposition 3.1 in \cite{SX}, in the same way the authors apply it in the proof of Theorem 1.3 in \cite{SX}, to a set that is much smaller than $K\sqrt{\tau}$, and as a result derive the $C^1$ asymptotics on a smaller set, but with a better decay rate. More precisely, Proposition 3.1 in \cite{SX} yields
\begin{equation}
\label{eq-C1-form}
\Big\| \bar u(y,\theta,\tau) - \sqrt{2(n-k)} - \frac{\sqrt{2(n-k)}}{4\tau}\, \big(|y|^2 - 2k\big)\Big\|_{{C^1\big(\mc{C}_{n,k}\cap B_{\tau^{1/98}}\big)}} = o(\tau^{-1-\alpha + 1/49}).
\end{equation}
 Note that we can take $\alpha$ as close to one as we want to guarantee that $\beta :=  \alpha - 1/49 > 0$. 

Let $\bar{\ve}_0 > 0$ be the same small constant as in Theorem \ref{thm-main0} and let $\bar \tau_0$ be such that  \eqref{eq-C1-form} holds for $\tau \geq \bar \tau_0$. 
Then, $\bar{\tau}_0 = -\log(T_0 - \bar{t}_0)$, for some unrescaled time $\bar t_0 < T_0$ (recall that  $T_0$ is  the  time at which the flow $N_t$ develops a singularity at $p_0=0$). 
Next, let $t_0$ be an unrescaled time such that  $\bar{t}_0 \le t_0 < T_0$ and call $\tau_0 = -\log(T_0 - t_0)$. We also require that   $t_0$ is chosen so that 
$\tau_0 \geq 4 \bar \tau_0$ and 
$\ve := \frac{\sqrt{2(n-k)}}{4\tau_0} \ll \bar{\ve}_0$.  

\smallskip

Now let us proceed with the proof of the theorem. Take any $M_0 := \{x+ U_0(x)\, \nu(x) \,\,\, |\,\,\, x\in N_0\}$, and let $M_t$ be the mean curvature flow solution starting at $M_0$. Our goal is to show the following:
 {\em  if $\|U_0(x)\|_{C^2} < \tilde{\ve}$ is sufficiently small, then the mean curvature flow $\{M_t\}_{t\in [0,T)}$ develops a nondegenerate neckpinch singularity in the sense of Definition \ref{def-neckpinches}. }

Observe first, that by the continuity of mean curvature flow with respect to its  initial data, there exists an $\bar{\ve}_1 = \bar{\ve}_1(N_0,t_0)$, such that if $\|U_0(x)\|_{C^2} < \tilde{\ve} \le  \bar{\ve}_1$, then the singular time of the mean curvature flow $M_t$ starting with $M_0$ as its initial data  has the property that $T(M_0) \ge t_0$, and the surfaces $M_t$ remain $O(\tilde{\ve})$ close to $N_t$ for all $t\in [0,t_0]$. We can choose $\tilde{\ve} \ll \ve$, so that $M_t$ stays $\ve^2$ close to $N_t$, for all $t\in [0,t_0]$. Note that by choosing $\tilde{\ve}$ even smaller we can guarantee that $M_t$ stays even closer than $\ve^2$ to $N_t$, for all $t\in [0,t_0]$.

Furthermore, continuity  of mean curvature flow on the  initial data  and \eqref{eq-C1-form} imply that at times $\tau\in [\tau_0/2,\tau_0]$, the rescaled hypersurfaces $\widebar{M}_{\tau}$, where we do parabolic rescaling by $(T_0 - t)^{-1/2}$, can be written as graphs $ u(y,\theta,\tau)$ over $\mc{C}_{n,k}\cap B_{\tau^{1/100}}$, so that
\[
\Big\|  u(y,\theta,\tau) - \sqrt{2(n-k)} - \frac{\sqrt{2(n-k)}}{4\tau}\, \big(|y|^2 - 2k\big)\Big\|_{C^1\big(\mc{C}_{n,k}\cap B_{\tau^{1/100}}\big)} = o(\tau^{-1-\beta}),
\]
implying
\begin{equation}
\label{eq-starting-pt}
\Big\|  u(y,\theta,\tau) - \sqrt{2(n-k)} \, \big(1 + \frac{2k}{4\tau}\big) - \frac{\sqrt{2(n-k)}}{4\tau}\,|y|^2 \Big\|_{C^1\big(\mc{C}_{n,k}\cap B_{\tau^{1/100}}\big)} = o(\tau^{-1-\beta}),
\end{equation}
where  $\beta  > 0$, and all $\tau\in [\tau_0/2,\tau_0]$;  (here we used that $\tau_0 \geq 4 \bar \tau_0$,   thus  all $\tau  \in [\tau_0/2,\tau_0]$
satisfy $\tau \ge  2 \bar \tau_0$ which means that we can apply  \eqref{eq-C1-form}). 

Using the definition of parabolic rescaling given by \ref{eq-rescaling}, we can rewrite \eqref{eq-starting-pt} in terms of the corresponding unrescaled height function $U(x,\theta,t)$ of the flow $M_t$, as 
\be\label{eqn-good14}\begin{split}
& \Big\|U(x,\theta,t) - \sqrt{T_0 - t} \, \Big(\sqrt{2(n-k)} + \frac{2k\sqrt{2(n-k)}}{4\tau }\Big) - \frac{\sqrt{2(n-k)}\, |x|^2}{4 \tau\, \sqrt{T_0-t}} \Big\|_{C^0(\mc{C}_{n,k}\cap B_{R(t)})} = \\
& \qquad \qquad = \sqrt{T_0 - t}\, \cdot \, o\big((-\log(T_0-t))^{-1-\beta}\big),
\end{split}
\ee
and
\be\label{U-theta}
|\nabla_\theta  U (x,\theta,t) | = \sqrt{T_0 - t}\, \cdot \, o\big((-\log(T_0-t))^{-1-\beta}\big), \qquad \mbox{for}\,\, |x| \leq R(t),\,\, \theta \in \mathbb{S}^{n-k}
\ee
where $R(t) := {\sqrt{T_0-t}\, \big(-\log(T_0-t)\big)^{1/100}}$, for all $t\in [T_0 - \sqrt{T_0 - t_0} , t_0]$. Note that in the above formula we maintained 
the  $\tau= - \log (T_0-t)$ for the sake of simplifying the notation. 

\smallskip

In order to apply Theorem \ref{thm-main0}, which will lead to the proof of Theorem \ref{thm-open-neck}, we  center   the flow $M_t$ 
at $t_0$ and  rescale it by $\rho = \sqrt{T_0 - t_0}\, (1 + a \ep)$, where $a:=  \frac{2k\, \ve}{\sqrt{2(n-k)}}$ and  $\ve := \frac{\sqrt{2(n-k)}}{4\tau_0}$.
That is, we define   \[\tilde{M}_s = \rho^{-1}\, M_{t_0+\rho^2  s}, \qquad \mbox{where} \,\,\, s\, \big (1+\tfrac{2k\ve}{\sqrt{2(n-k)}}\big)^2 \in [1 - e^{\tau_0/2}, 0].
 \]
Then the corresponding rescaled height function is $\tilde U(\tilde{x},\theta,s) = \rho^{-1}U(x,\theta,t)$, with $\tilde{x} = \rho^{-1} x$ and $t = t_0 + \rho^2s$. We will show next that $\tilde  U(x, \theta, 0)$ satisfies condition   \ref{eq-Ueps} in 
Theorem \ref{thm-main0}. 

To this end, observe first that  the above change of variables gives 
\be\label{eqn-good151} T_0-t   = T_0 - t_0 - \rho^2 s  = (T_0-t_0)\, \big (  1 - s\,  (1 + a\ep )^2 \big )
\ee
where $a :=  \frac{2k\ve}{\sqrt{2(n-k)}}$. 
Furthermore, restricting our focus to the localized interval $s \in [-1,0]$, the flowing time $\tau(s)$ satisfies $\tau(s) = -\log(T_0 - t) = \tau_0 - \log\big(1 - s(1+a\ve)^2\big)$.  Because $s$ is uniformly bounded on this interval, $\tau(s) = \tau_0 + O(1)$. This allows us to apply the Taylor expansion 
\be\label{eqn-good16}
\frac{1}{\tau(s)} = \frac{1}{\tau_0} - O(\tau_0^{-2}).
\ee

Substituting the formulas \eqref{eqn-good151} and  \eqref{eqn-good16}   into the profile function expansion \eqref{eqn-good14},  and noting that the $O(\tau_0^{-2})|\tilde{x}|^2$ remainder is of order $O(\tau_0^{-1.98}) = o(\tau_0^{-1})= o(\ep)$, gives us  that the height function $\tilde{U}(\tilde{x},\theta,s)$ satisfies 
\begin{equation}
\label{eq-very-useful123}
\Big\|\tilde{U}(\tilde{x},\theta,s) - \sqrt{2(n-k)}\, \sqrt{1-s(1+a\ve)^2} - \ve \, |\tilde{x}|^2\, \frac{(1+a\ve)}{\sqrt{1-s(1+a\ve)^2}} \Big\|_{C^0(\mc{C}_{n,k}\cap B(\tilde{R}(s))}= o\big(\ep\big), 
\end{equation}
and
\be
\label{U-theta2}
|\nabla_\theta  \tilde U (\tilde x,\theta,t) | = o\big( \ep \big)
\ee
where $\tilde{R}(s) := \frac{\sqrt{1-s(1+a\ve)^2}}{1+a\ve}\, \big ( \tau_0\, (1- s(1+a\ve)^2 ) \big)^{1/100}$, for  $s \in [-1,0] \subset [1 - e^{\tau_0/2},0]$. 
Keep in mind that $\ve = \frac{\sqrt{2(n-k)}}{4\tau_0} \ll \bar{\ve}_0$ and $\tau_0 = - \log (T_0- t_0)$.  

Note that \eqref{eq-very-useful123}-\eqref{U-theta2} imply 
\begin{equation}
\label{eq-first-cond}
\tilde{U}(\tilde{x},\theta,0) = \sqrt{2(n-k)} - \ve\, |\tilde{x}|^2 + o(\ve), \qquad |\nabla_\theta  \bar U (x,\theta,0) | = o\big(\ep\big)
\end{equation}
for all $|\tilde{x}| < 2 L_{\ve}$, where we can take e.g. $L_{\ve} = \ve^{-1/300}$. 

Furthermore,  note that \eqref{eq-very-useful123}  implies that 
\begin{equation}
\label{eq-very-useful124}
\big\|\tilde{U}(\tilde{x},\theta,s) - \sqrt{2(n-k)}\,  \sqrt{1-s(1+a\ve)^2 }\big\|_{C^0(\mc{C}_{n,k}\cap B( 2 L_\ep))} = \cO \big ( \ep  L_\ep^2 \big )
\end{equation}
and \eqref{U-theta2} implies 
\be\label{eq-useful125}
|\nabla_{\theta}\tilde{U}(\bar x,\theta,0)| = o(\ve),
\ee
for $|\bar x| \leq L_\ep$, and all  $ s \in  [-1,0]   \subset [1 - e^{\tau_0/2},0]$.  
The bound in \eqref{eq-very-useful124} can be improved,  via  standard interior Schauder estimates for quasilinear equations,  to
\begin{equation}
\label{eq-very-useful125}
\big\|\tilde{U}(\tilde{x},\theta,s) - \sqrt{2(n-k)} \, \sqrt{1-s(1+a\ve)^2} \big\|_{C^4 \big (\mc{C}_{n,k}\cap B(L_\ep) \big )}= \cO \big ( \ep L_\ep^2 \big )  
\end{equation}
for $s \in  [-\frac12,0]  \subset [1 - e^{\tau_0/2},0]$. The only extra fact that we need in order to apply standard Schauder theory  is a uniform
gradient bound $|\nabla \tilde U(\tilde{x},\theta,s)| \leq C$, holding for  $(\tilde x, \theta) \in C_{n,k} \cap B( 2 L_\ep)$, $s \in [-1,0]$,  which can easily be shown to hold by expressing  \eqref{eq-starting-pt}
in terms of $\tilde U(\tilde x,  \theta, s)$. 

Take $L_\ep =  \ve^{-1/300}$.  The last two bounds show that 
$\tilde{M}_0$ is $O(\ve^{\frac{149}{150}})$-close in the $C^4$-sense on the ball  $B(L_\ep)$  to the shrinking cylinder $\mc{C}_{n,k}$,
 with profile function $\tilde{U}(\tilde{x},\theta,s)$ satisfying
$|\nabla_{\theta}\tilde{U}(\tilde x,\theta,0)| = o(\ve),$
on the ball $B(L_\ep)$ centered at $p_0=0$. 
By Theorem \ref{thm-main} we have the mean curvature flow $\tilde{M}_s$ starting with ${\tilde M}_0 = \rho^{-1} M_{t_0}$ develops a nondegenerate neckpinch singularity in the sense of Definition \ref{def-neckpinches}. By Remark  \ref{remark-singularities} we conclude that the mean curvature flow $M_t$ develops a nondegenerate neckpinch singularity in the sense of Definition \ref{def-neckpinches} as desired. This concludes the proof of Theorem \ref{thm-open-neck}.

\end{proof}

\section{Appendix}
\label{sec-app}

\begin{lemma}
\label{lemma-beh-fk}
Let $t_k = t_0 + 10k$ for integers $k \ge 0$, and let $I_k = [t_{k-1}, t_k]$. Let $f_0(t_0) = 8 t_0^{1/10}$. Define the sequence of continuous functions $f_k : I_k \to \mathbb{R}$ for $k \ge 1$ by 
\[
f_k(t) = e^{(t-t_{k-1})/2} (f_{k-1}(t_{k-1}) - C_1)
\]
If $t_0^{1/10} \gg C_1$, then for all $k \ge 1$ and $t \in I_k$, the sequence behaves asymptotically as:
\[
f_k(t) \sim 8 t_0^{1/10} e^{(t-t_0)/2}
\]
\end{lemma}

\begin{proof}
We proceed by deriving a closed-form expression for $f_k(t)$. First, we evaluate the recurrence at the interval boundaries. Let $A_k = f_k(t_k)$ for $k \ge 0$. Evaluating the definition of $f_k(t)$ at the right endpoint $t = t_k$ yields:
\[
A_k = e^{(t_k-t_{k-1})/2} (A_{k-1} - C_1)
\]
Since $t_k - t_{k-1} = 10$ for all $k \ge 1$, this simplifies to a linear non-homogeneous recurrence relation:
\[
A_k = e^5(A_{k-1} - C_1)
\]

This yields 
\[
A_{k-1} = e^{5(k-1)} A_0 - C_1 \sum_{j=1}^{k-1} e^{5j}
\]
with $A_0 = 8 t_0^{1/10}$.

 Subtracting $C_1$ from both sides, we can absorb it into the geometric sum by starting the index at $j=0$:
\[
A_{k-1} - C_1 = e^{5(k-1)} A_0 - C_1 \sum_{j=0}^{k-1} e^{5j}
\]

Evaluating the geometric series gives:
\[
A_{k-1} - C_1 = e^{5(k-1)} (8 t_0^{1/10}) - C_1 \frac{e^{5k} - 1}{e^5 - 1}
\]

To align the exponential terms, we rewrite the sum by extracting $e^{5(k-1)}$ from the dominant part of the fraction:
\[
A_{k-1} - C_1 = e^{5(k-1)} \left( 8 t_0^{1/10} - \frac{C_1 e^5}{e^5 - 1} \right) + \frac{C_1}{e^5 - 1}
\]

Now, we substitute this closed-form expression back into the definition of $f_k(t)$ for $t \in I_k$,
\[
f_k(t) = e^{(t-t_{k-1})/2} \left[ e^{5(k-1)} \left( 8 t_0^{1/10} - \frac{C_1 e^5}{e^5 - 1} \right) + \frac{C_1}{e^5 - 1} \right]
\]

Distribute the leading exponential factor. For the first term, the exponents combine algebraically. Since $t_{k-1} = t_0 + 10(k-1)$, we have:
\[
\frac{t - t_{k-1}}{2} + 5(k-1) = \frac{t - (t_0 + 10(k-1))}{2} + 5(k-1) = \frac{t - t_0}{2}
\]

The function $f_k(t)$ then simplifies to
\[
f_k(t) = e^{(t-t_0)/2} \left( 8 t_0^{1/10} - \frac{C_1 e^5}{e^5 - 1} \right) + e^{(t-t_{k-1})/2} \frac{C_1}{e^5 - 1}
\]

To establish the asymptotic behavior, independent of $k \ge 1$, when $t_0$ is sufficiently large, we factor out the dominant term $8 t_0^{1/10} e^{(t-t_0)/2}$:
\[
f_k(t) = 8 t_0^{1/10} e^{(t-t_0)/2} \left[ 1 - \frac{C_1}{8 t_0^{1/10}} \frac{e^5}{e^5 - 1} + \frac{C_1}{8 t_0^{1/10}} \frac{e^{(t-t_{k-1})/2} e^{-(t-t_0)/2}}{e^5 - 1} \right]
\]

Note that $e^{(t-t_{k-1})/2} e^{-(t-t_0)/2} = e^{-(t_{k-1}-t_0)/2} = e^{-5(k-1)}$. Grouping the error terms yields:
\[
f_k(t) = 8 t_0^{1/10} e^{(t-t_0)/2} \left[ 1 - \frac{C_1}{8 t_0^{1/10} (e^5 - 1)} \left( e^5 - e^{-5(k-1)} \right) \right].
\]

Because $k \ge 1$, the term $e^{-5(k-1)}\in (0, 1]$. Under the initial assumption that $t_0^{1/10} \gg C_1$, the entire bracketed error term is strictly bounded by $1 - \mathcal{O}(C_1 t_0^{-1/10})$. 

Therefore, the lower-order constant terms become asymptotically negligible, confirming that universally on $t \in I_k$ for all $k \ge 1$:
\[
f_k(t) \sim 8 t_0^{1/10} e^{(t-t_0)/2}. 
\]
\end{proof}

\end{document}